\documentclass[11pt]{article}

 \usepackage{amsmath,amssymb,amscd,amsthm,esint}
 
\usepackage{graphics,amsmath,amssymb,amsthm,mathrsfs}

\usepackage{graphics,amsmath,amssymb,amsthm,mathrsfs,amsfonts}

\usepackage{sidecap}
\usepackage{float}
\usepackage{extarrows}
\usepackage{booktabs}
\usepackage{verbatim}
\usepackage{hyperref}
\usepackage[usenames,dvipsnames]{xcolor}

\newtheorem{thm}{Theorem}[section]
\newtheorem{lemma}[thm]{Lemma}

\theoremstyle{definition}
\newtheorem{remark}[thm]{Remark}

\def\XXint#1#2#3{{\setbox0=\hbox{$#1{#2#3}{\int}$}
         \vcenter{\hbox{$#2#3$}}\kern-.5\wd0}}

\def\R{\mathbb{R}}

\def\e{\varepsilon}

\def\loc{\text{loc}}

\numberwithin{equation}{section}

\begin{document}

\title{Homogenization of Boundary Value Problems \\ in
Perforated Lipschitz Domains}

\author{
Zhongwei Shen\thanks{Supported in part by NSF grant DMS-1856235.}}
\date{}
\maketitle

\begin{abstract}

This paper is concerned with  boundary regularity estimates in the homogenization of elliptic equations with rapidly oscillating and high-contrast 
coefficients. We establish uniform  nontangential-maximal-function estimates for the Dirichlet, regularity, and Neumann problems with $L^2$ boundary data
in a periodically perforated Lipschitz domain.

\medskip


\medskip

\noindent {\it MR (2020) Subject Classification}: 35B27; 35J25.

\end{abstract}


\section{\bf Introduction}\label{section-1}

This paper is concerned with boundary regularity estimates in the homogenization of elliptic equations with rapidly oscillating and
high-contrast coefficients  in perforated Lipschitz domains.
Let $A=A(y)$ be a real-valued $d\times d$ matrix 
satisfying the ellipticity condition,  
\begin{equation}\label{ellipticity}
\mu |\xi|^2 \le  (A \xi)\cdot  \xi   \quad \text{ and } \quad \| A\|_\infty \le \mu^{-1},
\end{equation}
for any $\xi \in \mathbb{R}^d$ and some $\mu>0$, and  the periodicity condition,
\begin{equation}\label{periodicity}
A(y +z)=A (y) \quad \text{ for any } y\in \mathbb{R}^d \text{ and } z\in \mathbb{Z}^d.
\end{equation}
Let $\omega$ be a connected and unbounded open set in $\mathbb{R}^d$. Assume that $\omega$ is 1-periodic; i.e.,
its characteristic function  is periodic with respect to $\mathbb{Z}^d$.
We also assume that 
each of connected components of
$
\mathbb{R}^d\setminus \omega 
$
is the closure of  a bounded open set $F_k$ with Lipschitz boundary,
and  that
\begin{equation}\label{dis}
\min_{k\neq \ell} \text{dist} (F_k, F_\ell ) >0.
\end{equation}
Let $\Omega$ be a bounded Lipschitz domain in $\mathbb{R}^d$ and
\begin{equation}\label{O-e}
\Omega^\e =\Omega \cap \e \omega=\Omega\setminus \overline{\e F},
\end{equation}
where $0<\e\le 1$ and $F=\cup_k F_k$. We will assume that 
\begin{equation}\label{F-k}
 \text{\rm dist} (\partial\Omega, \e F)\ge  \kappa \e
\end{equation}
for some $\kappa\in (0, 1)$.
As a result, $\partial \Omega^\e=\partial \Omega \cup \Gamma^\e$, where 
$\Gamma^\e =\Omega \cap \partial (\e \omega)$, and
dist$(\partial \Omega, \Gamma^\e)\ge \kappa \e$.
For $0\le \delta\le 1$, define
\begin{equation}\label{Lambda-e}
\Lambda_{\delta} (x)
=\left\{
\aligned
& 1  & \quad & \text{ if } x\in \overline{\omega},\\
& \delta & \quad & \text{ if } x\in  F.
\endaligned
\right.
\end{equation}
We are interested in the regularity estimates, which are uniform in $0<\e\le 1$ and $0\le \delta\le 1$, 
for the elliptic operator,
\begin{equation}\label{op-L}
\mathcal{L}_{\e, \delta}
=-\text{\rm div} 
\big( A^\e_{\delta}  (x) \nabla \big)
\end{equation}
in  $\Omega$,
where
\begin{equation}\label{A-e}
A^\e_{\delta}  (x)
=\big[ \Lambda_{ \delta}  (x/\e) \big]^2 A(x/\e).
\end{equation}
The operator $\mathcal{L}_{\e, \delta}$ arises naturally  in the modeling of acoustic propagation in porous media, periodic electromagnetic structures, and
soft inclusions in composite materials \cite{OSY-1992, JKO-1993}.

In the case $\delta=1$,
the regularity estimates for $\mathcal{L}_{\e}  =-\text{\rm div}(A(x/\e)\nabla)$
have been studied extensively in recent years.
Using a compactness method,
the interior Lipschitz estimate and  the boundary Lipschitz estimate for the Dirichlet problem in
a $C^{1, \alpha}$ domain were established by M. Avellaneda and F. Lin  in a seminal work \cite{AL-1987}.
The boundary Lipschitz estimate for the Neumann problem in a $C^{1, \alpha}$ domain 
was obtained in \cite{KLS-2013-N}.
In \cite{KS-2011-H, KS-2011-L}, C. Kenig  and the present author investigated  the  $L^p$ Dirichlet, regularity, 
and  Neumann problems  in Lipschitz domains.
We obtain  nontangentail-maximal-function estimates for the sharp ranges of $p$'s in the scalar case \cite{KS-2011-H},
 and for $p$ close to $2$ in the case of elliptic systems \cite{KS-2011-L}.
The results  in \cite{KS-2011-H, KS-2011-L} extend an earlier work of B. Dahlberg (unpublished) on the $L^p$ Dirichlet problem in the scalar case.
They also extend the classical work of B. Dahlberg, E. Fabes, D. Jerison, C. Kenig,  J. Pipher, G. Verchota, and many others 
 on the $L^p$ boundary value problems
for elliptic equations and systems in Lipschitz domains to the periodic homogenization setting.
We refer the reader to \cite{Kenig-book} for references on boundary value problems in nonsmooth domains,
to  \cite{Shen-book} for further references on periodic homogenization, and to \cite{Armstrong-book}
for related work on large-scale regularity  in stochastic homogenization.

In this paper we will be concerned with the case $0\le \delta<1$ in perforated domains,
where $\delta^2$ represents the conductivity ratio of the disconnected matrix block subset $\Omega \cap \e F $ to the connected subregion $\Omega^\e$.
Notice that the operator $\mathcal{L}_{\e, \delta}$  is not elliptic uniformly in $\delta$.
In the case where $A=I$,  $\Omega$ is   $C^{1, \alpha}$ and $\omega$  sufficiently smooth, using the compactness method in \cite{AL-1987}, 
the $W^{1, p}$ and Lipschitz estimates were  obtained  by L.-M. Yeh \cite{Yeh-2010, Yeh-2011, Yeh-2015, Yeh-2016}.
Also see earlier work in \cite{Schweizer-2000, Masmoudi-2004} for uniform estimates in  the case $\delta=0$.
In \cite{Chase-Russell-2017, Chase-Russell-2018}, B. Russell established the  large-scale interior Lipschitz estimates
for $0\le \delta< 1$, using an approximation method
originated in \cite{Armstrong-2016}.
In the stochastic setting with $\delta=0$, S. Armstrong and P. Dario \cite{Armstrong-2018} obtained 
 large-scale regularity results for the random conductance model on a supercritical percolation.
In this paper we shall be  mainly interested in the nontangential-maximal-function estimates, which are  uniform in  $\e$ and $\delta$,
for the  Dirichlet, regularity, and Neumann problems,
under the assumption that both $\Omega$ and $\omega$ are domains with Lipschitz boundary.

More precisely, we consider  the Dirichlet problem,
\begin{equation}\label{DP}
\left\{
\aligned
\text{div} \big(A^\e _{\delta}  (x)  \nabla u_{\e, \delta} \big) & =0 &\quad & \text{ in } \Omega,\\
u_{\e, \delta} & =f &\quad & \text{ on } \partial\Omega,
\endaligned
\right.
\end{equation}
and the Neumann problem,
\begin{equation}\label{NP}
\left\{
\aligned
\text{div} \big(A^\e _{ \delta}   (x) \nabla u_{\e, \delta} \big) & =0 &\quad & \text{ in } \Omega,\\
\frac{\partial  u_{\e, \delta}}{\partial \nu_\e} & =g &\quad & \text{ on } \partial\Omega,
\endaligned
\right.
\end{equation}
where $\frac{\partial u_{\e, \delta}}{\partial \nu_\e}  =n(x) \cdot A(x/\e)\nabla u_{\e, \delta}$ and
$n$ denotes the outward unit normal to $\partial\Omega$.
Our main results in this paper give nontangential-maximal-function estimates, which  are uniform in both $\e\in (0, 1]$ and $\delta\in [0, 1]$,
 for (\ref{DP}) and (\ref{NP}).
 These  estimates 
 are new even for the case  where $A=I$
  and $\omega$, $\Omega$ are smooth.

For a function $u\in L^2(\Omega)$,
let $N(u)$ denote the (generalized) nontangential maximal function of $u$, defined by
\begin{equation}\label{max}
N(u) (x)
=\sup\left\{
\left(\fint_{B(y, d(y)/4)} |u|^2\right)^{1/2}:\  y\in \Omega \text{ and } 
|y-x|<C_0 \, d(y)\right\}
\end{equation}
for $x\in \partial\Omega$, where
$d(y)=\text{\rm dist} (y, \partial\Omega)$ and $C_0=C_0(\Omega) >1$ is sufficiently large.
We impose a H\"older continuity condition on $A=A(y)$ in $\omega$,
\begin{equation}\label{smoothness}
 \|A\|_{C^{0, \sigma}  (\omega)}\le M
 \quad \text{  for some $\sigma \in (0,1)$ and $M>0$.}
\end{equation}
No smoothness condition is needed for $A$ in $F$.

\begin{thm}\label{main-thm-1}
 Let $0<\e\le 1$ and $0\le \delta \le 1$.
Assume that $A$ satisfies conditions \eqref{ellipticity}, \eqref{periodicity},  \eqref{smoothness},  and is symmetric.
Let $\omega$ be a connected, unbounded and  1-periodic open set with Lipschitz boundary and
satisfying \eqref{dis}.
Let $\Omega$ be a bounded Lipschitz domain satisfying \eqref{F-k}.
Then, for any $f\in L^2(\partial\Omega)$, the unique solution $u_{\e, \delta}$ of \eqref{DP}
with $N(u_{\e, \delta}) \in L^2(\partial\Omega)$ satisfies the estimate,
\begin{equation}\label{D-estimate}
\| N(u_{\e, \delta}) \|_{L^2(\partial\Omega)}
\le C \, \| f\|_{L^2(\partial\Omega)},
\end{equation}
where  $C$ 
depend only on $d$, $\mu$, $(M, \sigma)$ in (\ref{smoothness}), $\omega$, $\kappa$, and the Lipschitz character of $\Omega$.
\end{thm}

\begin{thm}\label{main-thm-2}
 Let $0<\e\le 1$ and $0\le \delta \le 1$.
Let $A$, $\omega$ and $\Omega$ satisfy  the same conditions  as in Theorem \ref{main-thm-1}.
Suppose $f\in H^1(\partial\Omega)$.
Then the unique weak solution $u_{\e, \delta}$ in $H^1(\Omega)$  of \eqref{DP} 
satisfies the estimate,
\begin{equation}\label{R-estimate}
\| N ( \nabla u_{\e, \delta})  \|_{L^2(\partial\Omega)}
\le C\,  \| f\|_{H^1 (\partial\Omega)},
\end{equation}
where $C$ depends only on $d$, $\mu$, $(M,\sigma )$, $\omega$, $\kappa$, and the Lipschitz character of $\Omega$.
\end{thm}

\begin{thm}\label{main-thm-3}
 Let $0<\e\le 1$ and $0\le \delta \le 1$.
Suppose $A$, $\omega$ and $\Omega$ satisfy the same conditions as in Theorem \ref{main-thm-1}.
Then, for any $g\in L^2(\partial\Omega)$ with $\int_{\partial\Omega} g\, d \sigma =0$,
the weak solutions in $H^1(\Omega)$ of \eqref{NP} satisfy the estimate
\begin{equation}\label{N-estimate}
\| N(\nabla u_{\e, \delta} ) \|_{L^2(\partial\Omega)}
\le C\, \| g\|_{L^2 (\partial\Omega)},
\end{equation}
where  $C$ depends only on $d$, $\mu$, $(M, \sigma)$, $\omega$, $\kappa$, and the Lipschitz character of $\Omega$.
\end{thm}

 A few remarks are in order.

\begin{remark}\label{remark-01}

By (\ref{F-k}),  the matrix $A^\e_{\delta}(x)$ is H\"older continuous near $\partial\Omega$.
As a result, the existence  and uniqueness of solutions $u_{\e, \delta}$ with $N(u_{\e, \delta})\in L^2(\partial\Omega)$  in Theorem \ref{main-thm-1}
and with $N(\nabla u_{\e, \delta})\in L^2(\partial\Omega)$ in Theorems \ref{main-thm-2} and \ref{main-thm-3} are more or less well known \cite{Kenig-book}.
The main contribution of this paper is that the constants $C$ in \eqref{D-estimate}, \eqref{R-estimate}, and \eqref{N-estimate}
do not depend on $\e\in (0, 1)$ and $\delta \in [0, 1]$.
\end{remark}

\begin{remark}

The boundary data in Theorem \ref{main-thm-1} is taken in the sense of nontangential convergence.
The boundary data in Theorems \ref{main-thm-2} and \ref{main-thm-3} may also be taken in the sense of nontangential convergence.
In (\ref{D-estimate}), the nontangential maximal function $N(u_{\e, \delta})$ may be replaced by
$\widetilde{N}(u_{\e, \delta})$, defined by
$$
\widetilde{N} (u) (x) =\sup\big\{ |u(y)|: \, y\in \Omega^\e \text{ and } |y-x|< \widetilde{C}_0\,  \text{dist}(y, \partial\Omega)\big\}.
$$
This follows from the proof of Theorem \ref{main-thm-1}.
\end{remark}

\begin{remark}\label{remark-02}

Let $\delta>0$ and $u_{\e, \delta}$ be  a weak solution of $\text{\rm div}(A^\e_{ \delta} \nabla u_{\e, \delta} ) =0$ in $\Omega$.
Then $\text{\rm div}(A(x/\e)\nabla u_{\e, \delta}) =0$ in $\Omega^\e$ and $\Omega \setminus \overline {\Omega^\e} $.
Moreover,
\begin{equation}\label{trace}
\left(\frac{\partial u_{\e, \delta}}{\partial \nu_\e}  \right)_+
=\delta^2
\left(\frac{\partial u_{\e, \delta} }{\partial \nu_\e} \right)_- 
\quad \text{ and } \quad
(u_{\e, \delta} )_+ = (u_{\e, \delta})_-
\quad \text{ on }  \Gamma^ \e= \Omega \cap \partial \Omega^\e,
\end{equation}
where $\pm$ indicates the traces taken from  $\Omega^\e$ and $\Omega \setminus \overline{\Omega^\e} $,
respectively.
In the case $\delta=0$, the Dirichlet problem \eqref{DP} is reduced to the mixed boundary value problem in $\Omega^\e$, 
\begin{equation}\label{DP-0}
\left\{
\aligned
\text{\rm div} (A (x/\e) \nabla u_{\e, 0}) & =0 & \quad &\text{ in } \Omega^\e, \\
n\cdot A(x/\e) \nabla u_{\e, 0} &=0 & \quad & \text{ on } \Gamma^\e,\\
u_{\e, 0} & =f & \quad&  \text{ on } \partial\Omega,
\endaligned
\right.
\end{equation}
while \eqref{NP} is reduced to the Neumann problem in $\Omega^\e$, 
\begin{equation}\label{NP-0}
\left\{
\aligned
\text{\rm div} (A (x/\e) \nabla u_{\e, 0}) & =0 & \quad &\text{ in } \Omega^\e, \\
n\cdot A(x/\e) \nabla u_{\e, 0} &=0 & \quad & \text{ on } \Gamma^\e,\\
n\cdot A(x/\e) \nabla u_{\e, 0}  & =g & \quad&  \text{ on } \partial\Omega.
\endaligned
\right.
\end{equation}
In this case  we shall   extend the  solution $u_{\e, 0}$  in $\Omega^\e $ to $\Omega$
 by solving the Dirichlet problem,
\begin{equation}\label{DP-e}
\text{\rm div} (A(x/\e)\nabla u_{\e,0}  ) =0 \quad \text{ in } 
\e F_k 
\quad \text{ and }\quad
(u_{\e, 0} )_-= (u_{\e, 0})_+ \quad \text{ on } \partial (\e F_k),
\end{equation}
for each $\e F_k$ contained in $\Omega$.
Consequently, \eqref{trace} continues to hold in the case $\delta=0$.
As we pointed out earlier,
the estimates in Theorems \ref{main-thm-1}, \ref{main-thm-2}, and \ref{main-thm-3}
for the boundary value problems \eqref{DP-0} and \eqref{NP-0} are new even in the case 
of Laplace's equation $\Delta u =0$ in $\Omega^\e$.

\end{remark}

We now describe our approaches to Theorems  \ref{main-thm-1}, \ref{main-thm-2}, and \ref{main-thm-3}.
Our main tool is   the Rellich estimates,
\begin{equation}\label{R-e}
\Big\|
\frac{\partial u_{\e, \delta}} {\partial \nu_\e} \Big \|_{L^2(\partial\Omega)}
\approx
\|\nabla_{\tan} u_{\e, \delta} \|_{L^2(\partial\Omega)},
\end{equation}
for  weak solutions of $\mathcal{L}_{\e, \delta} (u_{\e, \delta} ) =0$ in $\Omega$.
To prove (\ref{R-e}),
we first use local Rellich estimates for the elliptic operator 
$\text{\rm div}(A\nabla$) and condition \eqref{F-k}  as well as a covering argument to reduce the problem to the boundary layer estimates,
\begin{equation}\label{layer-e}
\aligned
\frac{1}{\e}
\int_{\Sigma_{\kappa \e}} 
|\nabla u_{\e, \delta}|^2\, dx
 & \le C \int_{\partial\Omega} 
|\nabla_{\tan } u_{\e, \delta} |^2\, d\sigma 
+ C \int_\Omega |\nabla u_{\e, \delta} |^2\, dx,\\
\frac{1}{\e}
\int_{\Sigma_{\kappa \e}}
|\nabla u_{\e, \delta}|^2\, dx
 & \le C \int_{\partial\Omega} 
\Big | \frac{\partial u_{\e, \delta} }{\partial \nu_\e} \Big|^2\, d\sigma 
+ C \int_\Omega |\nabla  u_{\e, \delta} |^2\, dx,
\endaligned
\end{equation}
where $\Sigma_t =\big\{ x\in \Omega: \ \text{\rm dist} (x, \partial\Omega) < t \big\}$.
The  estimates in (\ref{layer-e}) are then proved by establishing a sharp convergence rate in $H^1(\Omega)$
for a two-scale expansion.
We remark that a similar approach has been used in \cite{Shen-2017-APDE, GSS-2017} in  the case $\delta=1$ for elliptic systems of elasticity.
To deal with the fact that  the operator $\mathcal{L}_{\e, \delta}$ is not uniformly elliptic in $\delta$,
techniques of extension are used to treat the block regions  $\e F_k $ with small ellipticity constants.

With the Rellich estimates (\ref{R-e}) at our disposal, we follow an approach used in \cite{KS-2011-H}
for scalar elliptic equations with periodic coefficients.
To prove Theorem \ref{main-thm-1},
we  show that
\begin{equation}\label{N-M}
N(u_{\e, \delta}) \le C \big[ \mathcal{M}_{\partial\Omega} (|f|^{p_0} ) \big] ^{1/p_0} \quad \text{ on } \partial\Omega
\end{equation}
for some $p_0<2$, where $u_{\e, \delta}$ is a solution of (\ref{DP}) and
$\mathcal{M}_{\partial\Omega}$ denotes the Hardy-Littlewood maximal operator on $\partial\Omega$.
This is done by  applying  a localized version of (\ref{R-e})  to the Green function $G_{\e, \delta} (x, y)$ and
using the estimate,
\begin{equation}\label{G-e}
|G_{\e, \delta} (x, y)|
\le \frac{ C [ \text{\rm dist} (x, \partial\Omega)]^\sigma }{|x-y|^{d-2+\sigma}},
\end{equation}
which holds if  either $x, y\in \Omega^\e $ or $x, y\in \Omega$ with  $|x-y|\ge c_d \e$.
Estimate (\ref{G-e}) follows from  the boundary H\"older estimate for the operator  $\mathcal{L}_{\e, \delta}$  in Lipschitz domains, which is proved by
an approximation argument. A great amount of efforts, involving a reverse H\"older argument above the scale $\e$,  is also needed to
lower the exponent in (\ref{N-M}) to some $p_0<2$.

To prove Theorem \ref{main-thm-2}, we exploit the fact that
 the operator
$\mathcal{L}_{\e, \delta}$ is invariant under the translation $x \to x+\e z$, where $z\in \mathbb{Z}^d$. 
As in \cite{KS-2011-H}, this allows us to dominate $|\nabla u_{\e, \delta} (x) |$ for dist$(x, \partial\Omega)\ge C \e$  by
$$
\mathcal{M}_{\partial\Omega}  \big\{ |\nabla_{\tan} f |
+M_{rad} (Q_\e (u_{\e\, \delta})) \big\},
$$
where $M_{rad}$ is a radial maximal operator  and
$$
Q_\e (u)= \e^{-1} \big\{ u(x+ \e e_d) -u(x) \big\}, 
$$
 with
$e_d =(0, \dots, 0, 1)$.
Since $Q_\e (u_{\e, \delta}) $ is  a solution,
the term involving $Q_\e (u_{\e, \delta})$ may be  handled by using  a localized version of Theorem \ref{main-thm-1} and
the boundary layer estimates in (\ref{layer-e}).
Finally, Theorem \ref{main-thm-3} follows from Theorem \ref{main-thm-2} and (\ref{layer-e}).

Our proof of Theorem \ref{main-thm-1} yields the estimate
$\| N(u_{\e, \delta}) \|_{L^p(\partial\Omega)}
\le C_p \| f\|_{L^p(\partial\Omega)}$
for $2-\gamma< p\le \infty$, where $\gamma>0$ depends on $\Omega$, $\omega$, and $A$.
In view of  the results in \cite{KS-2011-H}, it would be interesting to establish
the $L^p$ estimates of $N(\nabla u_{\e, \delta}) $ for $1<p<2$, uniform in $\e >0 $ and $\delta\in [0, 1]$, for the regularity and 
Neumann problems considered in Theorems \ref{main-thm-2} and \ref{main-thm-3}.
Another  interesting problem would be the study  of the uniform boundary regularity estimates for the case   $1< \delta\le \infty$, 
the so-called stiff problem \cite{JKO-1993}.
The interior Lipschitz estimates have been established recently in \cite{Shen-2020}.


\section{\bf Preliminaries}\label{section-2}

Throughout this paper we assume that the matrix $A=A(y)$
satisfies conditions (\ref{ellipticity})-(\ref{periodicity}),  $\omega$ is a connected, unbounded and 1-periodic open set in $\mathbb{R}^d$,
and that 
\begin{equation}\label{F1}
\mathbb{R}^d\setminus \omega =\cup_k \overline{F_k}, 
\end{equation}
where each $\overline{F_k}$ is the closure of  a bounded Lipschitz domain $F_k$ with connected boundary.
We assume that $\overline{F_k} $'s are mutually disjoint and satisfy the condition (\ref{dis}).
This allows us to construct a sequence  of mutually disjoint open sets $\{\widetilde{F}_k \}$ with connected smooth  boundary such that $\overline{F_k} \subset \widetilde{F}_k$,
\begin{equation}\label{dist}
\left\{
\aligned
&  (1/100)\kappa  \le \text{dist}( F_k, \partial \widetilde{F}_k), \\
& (1/100)\kappa  \le \text{dist} (\widetilde{F}_k, \widetilde{F}_\ell ) \text{ for } k\neq \ell.
\endaligned
\right.
\end{equation}
Since diam$(F_k)\le\text{\rm diam}(Y)=\sqrt{d}$, we may assume diam$(\widetilde{F}_k) \le d$.
By the periodicity of $\omega$, 
$F_k$'s are the shifts of a finite number of bounded Lipschitz domains.
As a result,
we  may assume that $\widetilde{F}_k$'s are the shifts of a finite number of  bounded smooth domains.


\subsection{Extension operators}

Since $\widetilde{F}_k \setminus \overline{F_k}$
is a bounded Lipschitz domain,
there exists an extension operator $T_k$  from $\widetilde{F}_k \setminus \overline{F_k} $ to $\widetilde{F}_k$
such that
$$
\aligned
\| T_k (u)\|_{L^p(\widetilde{F}_k)}  & \le C_p\,  \| u\|_{L^p(\widetilde{F}_k \setminus \overline{F_k} )}, \\
\| T_k (u)\|_{W^{1, p} (\widetilde{F}_k)}
& \le C_p\,  \| u \|_{W^{1, p} (\widetilde{F}_k \setminus \overline{F_k} )},
\endaligned
$$
for $1< p< \infty$ \cite{Stein}.
Let 
$$
\widetilde{T}_k (u) = \fint_{\widetilde{F}_k \setminus F_k} u
+ T_k \Big(u-\fint_{\widetilde{F}_k \setminus F_k} u \Big).
$$
Then $\widetilde{T}_k$ is an extension operator from $\widetilde{F}_k \setminus \overline{F_k} $ to $\widetilde{F}_k$,
$$
\| \widetilde{T}_k (u)\|_{L^p(\widetilde{F}_k)}   \le C_p\,  \| u\|_{L^p(\widetilde{F}_k \setminus \overline{F_k} )}, 
$$
and
$$
\aligned
\| \nabla \widetilde{T}_k (u)\|_{L^p (\widetilde{F}_k)}
& \le C_p\,  \| u-\fint_{\widetilde{F}_k \setminus \overline{F_k}} u  \|_{W^{1, p} (\widetilde{F}_k \setminus \overline{F_k} )}
\le C_p \, \| \nabla u \|_{L^p(\widetilde {F}_k \setminus \overline{F_k})},
\endaligned
$$
where we have used a Poincar\'e inequality.
 By dilation  there exist  extension operators $E_{\e, k}$ from $\e \widetilde{F}_k \setminus \e \overline{ F_k}$ to $ \e \widetilde{F}_k$ such that
 \begin{equation}\label{ext-e}
\aligned
\| E_{\e, k}  (u)\|_{L^p(\e \widetilde{F}_k)}  & \le C_p\,  \| u\|_{L^p(\e\widetilde{F}_k \setminus  \e \overline{ F_k} )}, \\
\| \nabla E_{\e, k}  (u)\|_{L^p (\e \widetilde{F}_k)}
& \le C_p\,  \|\nabla u \|_{L^p (\e \widetilde{F}_k \setminus \e \overline{ F_k} )},
\endaligned
\end{equation}
for $1< p< \infty$ and $\e>0$, where $C_p$ depends only on $d$, $p$,  and $\omega$.
As a result, we obtain the following.

\begin{lemma}\label{lemma-ext}
Let $\Omega$ be a bounded Lipschitz domain satisfying \eqref{F-k}. Let  $1< p<\infty$.
Then, for any $u\in W^{1, p} (\Omega^\e )$, there exists $\widetilde{u}\in W^{1, p} (\Omega)$ such that
$\widetilde{u}=u$ in $\Omega^\e$,
\begin{equation}\label{ext-u}
\| \widetilde{u} \|_{L^p(\Omega)}
\le C_p\,  \| u \|_{L^p(\Omega^\e)}
\quad \text{ and } \quad
\| \nabla \widetilde{u} \|_{L^p(\Omega)} \le C _p\,  \| \nabla u \|_{L^p(\Omega^\e)},
\end{equation}
where $C_p$ depends only on $d$, $p$, $\kappa$, and $\omega$.
\end{lemma}

The next lemma will be used to treat regions where the ellipticity constant is small.

\begin{lemma}\label{lemma-ext-F}
Suppose $u\in H^1(\e \widetilde{ F}_k)$ and $\text{\rm div} (A(x/\e) \nabla u)=0$ in $\e F_k$.
Then
\begin{equation}\label{ext-F}
\| \nabla u \|_{L^2(\e F_k )} \le C \| \nabla u \|_{L^2(\e \widetilde{F}_k \setminus \e \overline{F_k})} ,
\end{equation}
where $C$ depends only on $d$, $\mu$, and $\omega$.
\end{lemma}

\begin{proof}
By dilation we may assume $\e=1$.
Let $\widetilde{u}\in H^1(\widetilde{F}_k)$ be an extension of $u|_{\widetilde{F}_k \setminus F_k}$ such that
$\|\nabla \widetilde{u}\|_{L^2(F_k)} \le C \| \nabla u \|_{L^2(\widetilde{F}_k \setminus \overline{F_k})}$.
Since $\widetilde{u}-u\in H_0^1(F_k)$ and $\text{\rm div}(A\nabla (\widetilde{u}-u))=\text{\rm div} (A \nabla \widetilde{u}) $ in $F_k$,
by energy estimates,
$$
\| \nabla (\widetilde{u} -u) \|_{L^2(F_k)} \le C \| \nabla \widetilde{u}\|_{L^2(F_k)},
$$
from which the inequality \eqref{ext-F} with $\e=1$  follows.
\end{proof}

We now give the energy estimates for \eqref{DP} and \eqref{NP}.

\begin{lemma}
Let $\Omega$ be a bounded Lipschitz domain satisfying \eqref{F-k}.
Let $u=u_{\e, \delta}$ be a weak solution of \eqref{DP} with $f\in H^{1/2}(\partial\Omega)$.
Then 
\begin{equation}\label{energy-D}
\| u \|_{H^1(\Omega)} \le C \| f\|_{H^{1/2}(\partial\Omega)},
\end{equation}
where $C$ depends only on $d$, $\mu$, $\kappa$, $\omega$, and $\Omega$.
\end{lemma}

\begin{proof}
Let $v\in H^1(\Omega)$ be a function such that $v=f$ on $\partial\Omega$ and
$\| v \|_{H^1(\Omega)} \le 2 \| f\|_{H^{1/2}(\partial\Omega)}$. 
Since $u-v \in H_0^1(\Omega)$ and
$\text{\rm div} (A_\delta^\delta \nabla (u-v)) = -\text{\rm div}(A_\delta^\e \nabla v)$ in $\Omega$,
it follows that
\begin{equation}\label{weak-1}
\int_\Omega A_\delta^\e \nabla (u-v) \cdot \nabla \psi \, dx
=-\int_\Omega A_\delta^\e \nabla v \cdot \nabla \psi\, dx
\end{equation}
 for any $\psi \in H_0^1(\Omega)$.
 Let $ \psi =u-v$ in \eqref{weak-1}. By using \eqref{ellipticity} and  the Cauchy inequality,
 $$
 \int_\Omega |\Lambda_\delta^\e \nabla (u-v)|^2\, dx \le C \int_\Omega |\nabla v|^2\, dx,
 $$
where $\Lambda^\e _\delta (x) =\Lambda_\delta (x/\e)$ and $C$ depends only on $\mu$.
It follows that $\| \nabla u \|_{L^2(\Omega^\e)} \le C \| f\|_{H^{1/2}(\partial\Omega)}$.
Using \eqref{F-k} and Lemma \ref{lemma-ext-F}, we obtain 
$$
\| \nabla u \|_{L^2(\Omega)} \le C \| \nabla u\|_{L^2(\Omega^\e)} \le C \| f\|_{H^{1/2}(\partial\Omega)}.
$$
This, together with Poincar'e's inequality $\| u-v \|_{L^2(\Omega)} \le C \| \nabla (u-v)\|_{L^2(\Omega)}$,
gives \eqref{energy-D}.
\end{proof}

\begin{lemma}
Let $\Omega$ be a bounded Lipschitz domain satisfying \eqref{F-k}.
Let $u=u_{\e, \delta}  $ be a weak solution of \eqref{NP} with $g\in H^{-1/2} (\partial\Omega)$ and $\langle  g, 1\rangle =0$.
Then 
\begin{equation}\label{energy-N}
\| \nabla u \|_{L^2(\Omega)} \le C \| g\|_{H^{-1/2}(\partial\Omega)},
\end{equation}
where $C$ depends only on $d$, $\mu$, $\kappa$, $\omega$,  and $\Omega$.
\end{lemma}

\begin{proof} Note that for any $\psi \in H^1(\Omega)$,
\begin{equation}\label{weak-2}
\int_\Omega A_\delta^\e \nabla u \cdot \nabla \psi \, dx = \langle g, \psi \rangle_{H^{-1/2}(\partial\Omega)\times  H^{1/2}(\partial\Omega)}.
\end{equation}
By letting $\psi =u$ in \eqref{weak-1}, we obtain 
\begin{equation}\label{weak-3}
\| \Lambda_\delta^\e \nabla u\|_{L^2(\Omega)}^2 
\le C \| g \|_{H^{-1/2}(\partial\Omega)} \| \nabla u \|_{L^2(\Omega)}.
\end{equation}
By Lemma \ref{lemma-ext-F}, $\| \nabla u\|_{L^2(\Omega)} \le C \| \nabla u\|_{L^2(\Omega^\e)}$.
This, together with \eqref{weak-3},  yields \eqref{energy-N}.
\end{proof}

The next lemma gives  a Caccioppoli  inequality for the operator $\text{\rm div}(A^\e_\delta \nabla )$.

\begin{lemma}\label{lemma-4.1}
Let $u=u_{\e, \delta}\in H^1(\Omega )$ be a weak solution of 
$\text{\rm div}(A_\delta^\e \nabla u)=0$  in a bounded Lipschitz domain $\Omega$.
Then
\begin{equation}\label{4.1-0}
\int_{\Omega} |\Lambda_\delta^\e \nabla (u\varphi) |^2 \, dx
\le C \int_{\Omega} | \Lambda_\delta^\e u |^2 |\nabla \varphi|^2\, dx
\end{equation}
for any $\varphi\in C^1_0(\Omega)$,
where  $C$ depends only on $d$ and $\mu$.
\end{lemma}

\begin{proof}
This may be  proved by using the test function $u\varphi^2$ in the weak formulation of (\ref{4.1}), 
as in the proof of the standard Caccioppoli inequality for $\delta=1$.
\end{proof}


\subsection{Correctors}

For $0\le \delta\le 1$, let
$A_\delta (y)  =\Lambda_{\delta^2} (y) A(y)=[ \Lambda_{\delta} (y) ]^2  A (y)$, where  $\Lambda_\delta$ is defined by \eqref{Lambda-e}.
Observe that the $d\times d$ matrix  $A_\delta(y) $ is 1-periodic.
Let  $\chi_\delta (y) = (\chi_{\delta, 1}  (y), \dots, \chi_{\delta, d}  (y) )$ be the  corrector for the operator 
$-\text{div} ( A_\delta \nabla)$, where, for $1\le j\le d$, the function 
$\chi_{\delta, j}\in H^1_{\loc} (\mathbb{R}^d) $ is a weak solution of
\begin{equation}\label{cor-1}
\left\{
\aligned
& -\text{\rm div} (A_\delta  \nabla  \chi_{\delta, j} )=
\text{\rm div} ( A_\delta  \nabla y_j ) \quad \text{ in } \mathbb{R}^d,\\
& \chi_{\delta, j}  \text{ is 1-periodic},
\endaligned
\right.
\end{equation}
with
\begin{equation}\label{cor-2}
\left\{
\aligned
& \int_{ Y } \chi_{\delta, j} \, dy=0  & \quad & \text{ if } \delta>0,\\
& \int_{ Y \cap \omega} \chi_{\delta, j} \, dy=0  & \quad & \text{ if } \delta=0,\\
\endaligned
\right.
\end{equation}
and  $Y=[0, 1]^d$.
If $\delta>0$, the existence and uniqueness of correctors follow readily  from the Lax-Milgram Theorem 
by using the bilinear form, 
$$
\int_Y  A_\delta  \nabla \phi \cdot \nabla \psi\, dy,
$$
on the Hilbert space $H^1_{\text{per}} (Y)$,
the closure of 1-periodic $C^\infty$ functions in $H^1(Y)$.
In the case $\delta=0$, one uses  the bilinear form
$$
\int_{Y\cap \omega} A\nabla \phi \cdot \nabla \psi \, dy
$$
on the Hilbert space $H^1_{\text{per}} (Y\cap \omega)$,
the closure of 1-periodic $C^\infty$ functions in $H^1(Y\cap \omega)$.
This gives the definition of $\chi_{0, j}$ on $\omega$.
Recall that
  $\mathbb{R}^d \setminus \overline{\omega} =F=\cup_k { F_k} $.
  We extend $\chi_{0, j}$ to each $F_k$ by using the weak solution in $H^1(F_k)$  of
$
-\text{\rm div} (A\nabla u) =\text{\rm div} (A\nabla y_j) \  \text{ in } F_k,
$
with Dirichlet data $u=\chi_{0, j}$ on $\partial F_k$.

\begin{lemma}\label{lemma-2.1}
Let $0\le \delta\le 1$.
Then
\begin{equation}\label{2.1-0}
\int_Y  \Big( |\nabla \chi_\delta|^2 + |\chi_\delta|^2 \Big) \, dy \le C,
\end{equation}
where $C$ depends only on $d$, $\mu$, and $\omega$.
\end{lemma}

\begin{proof}
Let $0<\delta\le 1$.
By energy estimates, 
$$
\int_{Y} |\Lambda_\delta  \nabla \chi_\delta |^2\, dy \le C \int_Y |\Lambda_\delta A|^2\, dy \le C.
$$
This gives $\| \nabla \chi_\delta\|_{L^2(Y\cap \omega)} \le C$.
Next, note that
$
\text{\rm div} (A \nabla (\chi_{\delta, j} +y_j)) =0
$
in $ F_k$.
It follows by Lemma \ref{lemma-ext-F} that
$$
\|\nabla \chi_\delta \|_{L^2(F_k)}
\le C + C  \| \nabla \chi_\delta \|_{L^2(Y\cap \omega)} \le C.
$$
Since $Y \cap F_k \neq \emptyset$ only for a finite number of $k$'s,
we obtain  $\|\nabla \chi_\delta \|_{L^2(Y\setminus \omega)} \le C$.
As a result, we have proved that $\|\nabla \chi_\delta\|_{L^2(Y)} \le C$.
In view of (\ref{cor-2}), the estimate $\| \chi_\delta \|_{L^2(Y)} \le C$ follows by Poincar\'e's  inequality.

If $\delta =0$, the energy estimate gives $\|\nabla \chi_0\|_{L^2(Y\cap \omega)} \le C$.
By Poincar\'e's  inequality and (\ref{cor-2}), we obtain 
$\|\chi_0 \|_{H^1(Y\cap \omega)} \le C$.
In view of the definition of $\chi_0$ on $F_k$, we have
$$
\aligned
\|\chi_0\|_{H^1(F_k)}  & \le C +  C \|\chi_0 \|_{H^{1/2} (\partial F_k)}
 \le C + C \| \chi_0\|_{H^1(\widetilde{F}_k\setminus F_k)}\\
& \le  C + C \| \chi_0 \|_{H^1(Y\cap \omega)}
\le C.
\endaligned
$$
It follows that $\|\chi_0\|_{H^1(Y\setminus \omega)} \le C$.
\end{proof}


\subsection{Homogenized operator}

The homogenized matrix for the operator $-\text{\rm div} (A_\delta \nabla )$ is given by
\begin{equation}\label{h-matrix}
\widehat{A_\delta}
=\fint_Y  \Big\{ A_\delta  + A_\delta  \nabla \chi_\delta \Big\} \, dy.
\end{equation}

\begin{lemma}\label{lemma-2.2}
Let $0< \delta\le 1$.
Then
\begin{equation}\label{h-diff}
|\widehat{A_\delta} - \widehat{A_0} |
\le C \delta^2,
\end{equation}
where $C$ depends only on $d$, $\mu$, and $\omega$.
\end{lemma}

\begin{proof}
We provide a proof, which  also may be found in \cite{Chase-Russell-2018}, for the reader's convenience.
Note that
$$
\aligned
|\widehat{A_\delta}-\widehat{A_0}|
&\le \int_Y | \Lambda_{\delta^2} -\Lambda_0| |A|\, dy
+ \int_Y | \Lambda_{\delta^2} A \nabla \chi_\delta
-\Lambda_0 A \nabla \chi_0|\, dy\\
&\le C \delta^2 + C \int_{Y\cap \omega}
|\nabla(  \chi_\delta -\chi_0) |\, dy + C \delta^2 \int_{Y\setminus \omega} |\nabla \chi_\delta|\, dy\\
&\le C \delta^2 + C \int_{Y\cap \omega}
|\nabla ( \chi_\delta -\chi_0) |\, dy,
\endaligned
$$
where we have used (\ref{2.1-0}) for the last inequality.
To bound $|\nabla (\chi_\delta -\chi_0)|$, we observe that
$$
-\text{\rm div} (\Lambda_{\delta^2} A \nabla (\chi_\delta -\chi_0))
=\text{\rm div} ( ( \Lambda_{\delta^2} -\Lambda_0) A \nabla  (y+ \chi_0)).
$$
Thus,  for any $\psi\in H^1_{\text{per}} (Y)$,
\begin{equation}\label{ext-99}
\int_Y \Lambda_{\delta^2} A \nabla (\chi_\delta -\chi_0) \cdot \nabla \psi \, dy
=-\delta^2 \int_{Y\setminus  \omega}
A\nabla (y +\chi_0) \cdot \nabla \psi\, dy.
\end{equation}
We now choose $\psi\in H^1_{\text per}(Y)$ such that
$\psi=\chi_\delta -\chi_0$ on $Y\cap \omega$ and
\begin{equation}\label{ext-98}
\|\nabla \psi\|_{L^2(Y)} \le C \|\nabla (\chi_\delta -\chi_0)\|_{L^2(Y\cap \omega)}.
\end{equation}
It follows from (\ref{ext-99}) and (\ref{ext-98}) that
$$
\aligned
\int_{Y\cap \omega}
|\nabla (\chi_\delta -\chi_0)|^2\, dy
&\le C \delta^2 \int_{Y\setminus \omega}
|\nabla (\chi_\delta -\chi_0)| |\nabla \psi|\, dy
+ C \delta^2 \int_{Y\setminus \omega}
|\nabla (y +\chi_0)| |\nabla \psi|\, dy\\
&\le C \delta^2 \|\nabla (\chi_\delta -\chi_0)\|_{L^2(Y\cap \omega)}.
\endaligned
$$
Hence,
$$
\| \nabla (\chi_\delta -\chi_0)\|_{L^1(Y\cap \omega)}\le 
\| \nabla (\chi_\delta -\chi_0)\|_{L^2(Y\cap \omega)}
\le C \delta^2.
$$
\end{proof}

Note that for $\xi\in \mathbb{R}^d$,
\begin{equation}\label{h-matrix-1}
\big(  \widehat{A_\delta} \xi\big) \cdot  \xi
=\int_Y   A_\delta  \nabla \big( (y+ \chi_\delta) \cdot \xi\big)
\cdot   \nabla \big( ( y +\chi_\delta) \cdot \xi \big) \, dy.
\end{equation}
It follows  that 
$\widehat{A_\delta}$ is symmetric if $A$ is symmetric.

\begin{thm}
Let $0\le \delta\le 1$. Then for any $\xi \in \mathbb{R}^d$,
\begin{equation}\label{ellipticity-1}
\mu_0 |\xi|^2 \le \big( \widehat{A_\delta} \xi\big) \cdot  \xi \quad \text{ and } \quad |\widehat{A_\delta}|\le \mu_0^{-1},
\end{equation}
where $\mu_0>0$ depends only on $d$, $\mu $, and $\omega$.
\end{thm}

\begin{proof}
The second inequality in (\ref{ellipticity-1}) follows readily  from (\ref{2.1-0}).
To see the first, note that 
$$
\aligned
\big(  \widehat{A_\delta} \xi\big) \cdot  \xi
&\ge \mu \min (1, \delta^2)
 \int_Y |\nabla ( (y +\chi_\delta) \cdot \xi) |^2\, dy\\
 & =\mu \min (1, \delta^2 )
 \int_Y  \Big\{ |\nabla (\chi_\delta \cdot \xi)|^2 + |\xi|^2
 +2 (\nabla \chi_\delta )\xi  \cdot \xi \Big\}\, dy\\
 & \ge \mu \min (1 , \delta^2) |\xi|^2.
 \endaligned
 $$
Thus, in view of (\ref{h-diff}),
 it suffices to consider  the case $\delta=0$.
To this end,
suppose there exists a sequence $\{ A^\ell\}$ of 1-periodic matrices 
satisfying (\ref{ellipticity}) and a sequence $\{ \xi^\ell\}\subset \mathbb{R}^d$ with $|\xi^\ell|=1$ 
such that $(\widehat{A^\ell_0} \xi^\ell )\cdot \xi^\ell \to  0$,
as $\ell \to \infty$.
By passing to a subsequence we may assume  $\xi^\ell \to \xi$.
It follows that $(\widehat{A^\ell_0} \xi )\cdot \xi\to  0$,
as $\ell \to \infty$, where $|\xi|=1$.
Let $\chi^\ell_0$ denote the corrector for the matrix $A^\ell_0$.
Then
$$
\int_{Y\cap \omega} |\nabla \big( (\chi^\ell_0 +y)\cdot \xi\big)|^2\, dy \to 0,
$$
as $\ell \to \infty$.
Let $E_\ell$ denote the  average of $(\chi^\ell_0+y)\cdot \xi$ over $Y\cap \omega$.
Then, $(\chi^\ell_0 +y)\cdot \xi -E_\ell $ converges to zero in $H^1(Y\cap\omega)$.
Since the sequence $\{ \chi^\ell_0\cdot \xi -E_\ell \}$ is bounded in $H^1_{\text{per}} (Y\cap \omega)$, by 
passing to a subsequence, 
we may assume it converges weakly  in $H^1_{\text{per}} (Y\cap \omega)$.
This implies that $y\cdot \xi\in H^1_{\text per} (Y\cap \omega)$.
However, since all connected components of $\mathbb{R}^d \setminus \omega$ are bounded,
for each $1\le j\le d$, 
there exists $y\in \partial Y\cap \omega$ such that $y +e_j \in \partial Y \cap \omega$.
It follows that $(y+ e_j)\cdot \xi =y\cdot \xi$.
Consequently,  we obtain $\xi=0$, which contradicts with the fact $|\xi|=1$.
The argument above shows that $( \widehat{A_0}\xi)\cdot \xi \ge \mu_0|\xi|^2$
for some $\mu_0>0$ depending only on $d$, $\mu$, and $\omega$ (not directly on $A$).
\end{proof}


\subsection{Flux correctors}

Let
\begin{equation}\label{B}
B_\delta =A_\delta + A_\delta \nabla \chi_\delta -\widehat{A_\delta}.
\end{equation}
Write $B_\delta = ( b_{\delta, ij} )_{d\times d}$. By the definitions of $\chi_\delta$ and $\widehat{A_\delta}$,
\begin{equation}\label{B-1}
\frac{\partial}{\partial y_i} b_{\delta, ij}=0
\quad \text{ and } \quad
\int_Y b_{\delta, ij} \, dy =0,
\end{equation}
where the repeated index is summed.

\begin{lemma}\label{lemma-2.3}
There exist $\phi_{\delta, kij} \in H^1_{\text{per}} (Y)$, where $1\le i, j, k \le d$,  such that
$\int_Y \phi_{\delta, kij} \, dy=0$, 
\begin{equation}\label{2.3-0}
b_{\delta, ij} =\frac{\partial}{\partial y_k } \phi_{\delta, kij}
\quad \text{ and } \quad
\phi_{\delta, kij} =-\phi_{\delta, ikj}.
\end{equation}
Moreover, for $0\le \delta\le 1$,
\begin{equation}\label{2.3-1}
\int_Y  \Big( |\nabla \phi_{\delta, kij}|^2 + |\phi_{\delta, kij}|^2 \Big)\, dy \le C,
\end{equation}
where $C$ depends only on $d$, $\mu$, and $\omega$.
\end{lemma}

\begin{proof}
The proof is similar to the case $\delta=1$ \cite{JKO-1993, KLS-2012}.
Since $\int_Y b_{\delta, ij} \, dy=0$, there exists $f_{ij} \in H^2_{\text{per}} (Y)$ such that
$$
\Delta f_{ij} = b_{\delta, ij} \quad \text{ in } Y \quad 
\text{ and } \quad
\| f_{ij}\|_{H^2_{\text{per}} (Y)}\le C \| b_{\delta, ij} \|_{L^2(Y)}.
$$
Let
$$
\phi_{\delta, kij} =\frac{\partial }{\partial y_k} f_{ij}
-\frac{\partial}{\partial y_i} f_{kj}.
$$
The second equation in (\ref{2.3-0}) is obvious, while the first follows from the first equation in (\ref{B-1}).
Finally, note that if $0\le \delta \le 1$,
$$
\aligned
\| \phi_{\delta, kij} \|_{H^1 (Y)}
 & \le C \big\{ \| f_{ij} \|_{H^2(Y)}
 + \| f_{kj}\|_{H^2(Y)} \big\}\\
 &\le  C \big\{ \| b_{\delta, ij} \|_{L^2(Y)} + \| b_{\delta, kj}\|_{L^2(Y)} \big\}\\
 &\le C,
 \endaligned
 $$
 where we have used (\ref{2.1-0}) for the last inequality.
\end{proof}


\subsection{Small-scale  H\"older estimates in Lipschitz domains}

Let $B_r=B(0, r)$, 
$$
B_r^+= B_r\cap \{ (x^\prime, x_d) \in  \mathbb{R}^d: x_d> \psi (x^\prime)\}, \quad
 B_r^- =B_r\cap \{ (x^\prime, x_d)\in \mathbb{R}^d: x_d<\psi (x^\prime)  \},
 $$
  and
 $\Delta_r = B_r\cap \{ (x^\prime, x_d): x_d=\psi (x^\prime) \}$,
 where $\psi : \mathbb{R}^{d-1} \to \mathbb{R}$ is a Lipschitz function with
 $\psi (0)=0$ and $\|\nabla \psi\|_\infty\le M_1$.
 Consider the weak solution of $\text{\rm div} (\widetilde{A}_\delta \nabla u)=0$ in $B_1$,
 where $\widetilde{A}_\delta (x)=A(x)$ for $x\in B_1^+$ and $\widetilde{A}_\delta(x)=\delta^2 A(x)$ for $x\in B_1^-$.
 In the case $\delta=0$ we also  assume $\text{\rm div}(A\nabla u)=0$ in $B_1^-$.

\begin{lemma}\label{lemma-local-1}
Suppose $A$ satisfies the ellipticity condition \eqref{ellipticity}.
Let $u\in H^1(B_1 )$ be a weak solution of $
\text{\rm div}(\widetilde{A}_\delta \nabla u)=0$ in $B_1$.
Then
\begin{align}
\| u\|_{C^{ \sigma} (B^+_{1/2})}
 & \le C \left\{ 
\left( \fint_{B_1^+} |u|^2 \right)^{1/2}
+ \delta \left( \fint_{B^-_1} |u|^2 \right)^{1/2} \right\}
 \label{6.10-1},\\
\| u\|_{C^{ \sigma} (B^-_{1/2})}
&\le C 
\left( \fint_{B_1} |u|^2 \right)^{1/2} ,
\label{6.10-01}
\end{align}
where $C>0$ and $\sigma\in (0, 1)$ depend only on $d$,  $\mu$, and $M_1$.
 \end{lemma}
 
 \begin{proof}
 We first note that   if $0<\delta_0\le \delta\le 1$,
 the estimates (\ref{6.10-1})-(\ref{6.10-01})  follow directly from the De Giorgi - Nash estimates,
 with $C$ and $\sigma$ depending  on $d$, $\mu$, and $\delta_0$.
 Next,   consider the case $\delta=0$.
 Since $\left(\frac{\partial u}{\partial \nu}\right)_+=0$ on $\Delta_1$,
 it follows from the De Giorgi - Nash theory by a reflection argument that  there exists $\rho>0$, depending only on $d$, $\mu$,
  and $M_1$, such that
 \begin{equation}\label{6.10-2}
 \| u\|_{C^{\rho}(B^+_{3/4})}
 \le C \left(\fint_{B_1^+} |u|^2 \right)^{1/2}.
 \end{equation}
 Using Theorem 8.29 in   \cite{Gilbarg-Trudinger}  and (\ref{6.10-2}), we obtain 
 \begin{equation}\label{6.10-3}
 \aligned
  \| u\|_{C^{\rho/2}(B^-_{1/2})}
  & \le  C \| u\|_{C^{\rho } (\Delta _{3/4})}
 + C \left(\fint_{B_1^-} |u|^2 \right)^{1/2}\\
 &\le
  C \left(\fint_{B_1} |u|^2 \right)^{1/2}.
 \endaligned
 \end{equation}
 
 To treat the case $0< \delta< \delta_0$, we use a perturbation argument.
 Define
 $$
 \Phi (r; u)=\inf_k 
 \left\{ 
  \left(\fint_{B^+_r} |u-k|^2\right)^{1/2}
  +  \delta \left(\fint_{B^-_r} |u-k|^2\right)^{1/2}  \right\}.
 $$
 Let $v\in H^1(B_R)$ be a solution of $\text{\rm div} (A\nabla v)=0$ in $B_R^+$ and $B_R^-$
 such that
 $v=u$ on $\partial B_R$,  $\left(\frac{\partial v}{\partial \nu}\right)_+=0$ on $\Delta_R$, and
 $v_+=v_-$ on $\Delta_R$.
 By the estimates for the case $\delta=0$,
 \begin{equation}\label{6.10-5}
 \Phi (r; v) \le C \left(\frac{r}{R}\right)^{\rho/2}  \Phi (R; v)
 \end{equation}
 for $0<r<R$.
 Also, note that for any $\varphi\in H_0^1(B_R)$,
 \begin{equation}\label{6.10-6}
 \aligned
 \int_{B_R^+} A \nabla (u-v) \cdot \nabla \varphi \, dx
  =-\delta^2 \int_{B_R^-} A\nabla u  \cdot \nabla \varphi\, dx.
 \endaligned
 \end{equation}
We now choose $\varphi\in H^1_0(B_R)$ in (\ref{6.10-6})  to be an extension of $(u-v)|_{B_R^+}$
such that 
$$
\|\nabla \varphi \|_{L^2(B_R^-)} \le C \| \nabla (u-v)\|_{L^2(B_R^+)}.
$$
 This gives
 \begin{equation}\label{6.10-7}
 \| \nabla (u-v)\|_{L^2(B_R^+)}
 \le C \delta^2 \| \nabla u\|_{L^2(B_R^-)}.
 \end{equation}
 Since $\varphi= u-v$ on $\partial B_R^-$ and $\text{\rm div} (A\nabla (u-v))=0$ in $B_R^-$,
 by the energy estimate,
 $$
 \|\nabla (u-v)\|_{L^2(B_R^-)}
 \le C \|\nabla \varphi \|_{L^2(B_R^-)}
 \le C \|\nabla (u-v)\|_{L^2(B_R^+)}.
 $$
 As a consequence, we have proved that
 $$
 \aligned
  \left(\fint_{B_R} |\nabla (u-v)|^2 \right)^{1/2}
 & \le 
C \delta^2 \left(
 \fint_{B_R^-} |\nabla u|^2 \right)^{1/2}\\
 & 
  \le \frac{C \delta}{R}
  \left\{ \left(
   \fint_{B^+_{2R}} |u|^2 \right)^{1/2}
   +\delta \left(
   \fint_{B^-_{2R}} |u|^2 \right)^{1/2}\right\},
\endaligned
 $$
 where we have used Caccioppoli's  inequality (\ref{4.1-0})  for the last step.
 By Poincar\'e's  inequality,
 $$
 \aligned
  \left(\fint_{B_R} |u-v|^2 \right)^{1/2}
  \le C \delta 
  \left\{ \left(
   \fint_{B^+_{2R}} |u|^2 \right)^{1/2}
   +\delta \left(
   \fint_{B^-_{2R}} |u|^2 \right)^{1/2}\right\}.
\endaligned
 $$
 It follows that
 $$
 \aligned
 \Phi(\theta R; u) 
 &\le  \Phi(\theta R; v) +  C_\theta
 \left(\fint_{B_{\theta R} } | u-v|^2 \right)^{1/2}\\
 &\le C_0\theta^{\rho/2}  \Phi(R; v)
 +  C_\theta
  \left(\fint_{B_{\theta R} } | u-v|^2 \right)^{1/2}\\
 &\le C_0 \theta^{\rho/2}   \Phi(R; u)
+  C_\theta   \delta  
\left\{ 
 \left(\fint_{B^+_{ 2R} } | u|^2 \right)^{1/2}
 + \delta 
  \left(\fint_{B^-_{ 2R} } | u|^2 \right)^{1/2}\right\}.
 \endaligned
 $$
Since $u-k$ is also a solution, we obtain 
$$
\Phi(\theta R; u) \le \left\{ C_0\theta^{\rho/2}   + C_\theta \delta \right\} \Phi(2R; u).
$$
Choose $\theta\in (0, 1/4)$ so small that $C_0\theta^{\rho/2}  \le (1/4)$.
With $\theta$ chosen, we then choose $\delta_0$ so small that 
$C_\theta  \delta_0 \le (1/4)$.
This yields that if $0<\delta< \delta_0$, then
$$
\Phi(\theta R; u) \le \frac12 \Phi (2R; u).
$$
 It follows that for $0<r<R\le 1$,
 \begin{equation}\label{6.10-8}
 \Phi (r; u) \le C\left(\frac{r}{R}\right)^{\sigma} \Phi (R; u),
 \end{equation}
 where $\sigma$ depends only on $d$, $\mu$ and $M_1$.
 The inequality  (\ref{6.10-8}), together with the interior H\"older estimates for $u$,
 gives (\ref{6.10-1}). 
 As in (\ref{6.10-3}), since $\text{\rm div}(A\nabla u) =0$ in $B_1^-$,
 the estimate (\ref{6.10-01}) follows from (\ref{6.10-1}). 
  \end{proof}

\begin{thm}\label{thm-local-2}
Suppose $A$ satisfies \eqref{ellipticity}.
Let $u=u_{\e, \delta}$  be a weak solution of
$\text{\rm div} (A_\delta^\e \nabla u )=0$ in $B_{2r}=B(x_0, 2r)$,
where $A_\delta^\e (x) = [\Lambda_\delta (x/\e)]^2  A(x/\e)$  and $0<r< (8d)\e$.
Then 
\begin{align}
\|   u\|_{L^\infty(B_r\cap \e \omega)} 
 & \le C \left(\fint_{B_{2r}} |\Lambda_\delta^\e u|^2 \right)^{1/2} \label{LL},\\
 \|   u\|_{L^\infty(B_r\cap \e F )} 
 & \le C \left(\fint_{B_{2r}} | u|^2 \right)^{1/2} \label{LL-F},\\
\| u  \|_{C^{0, \sigma} (B_r\cap \e \omega)}
& \le C r^{-\sigma} 
 \left(\fint_{B_{2r}}
| \Lambda^\e_\delta  u |^2 \right)^{1/2}
\label{6.11-0},\\
\| u  \|_{C^{0, \sigma} (B_r\cap \e F )}
&\le  C r^{-\sigma}
\left(\fint_{B_{2r}}
|u |^2 \right)^{1/2},\label{6.11-01}
\end{align}
where $\Lambda_\delta^\e (x)=\Lambda_\delta (x/\e)$ and
 $C>0$, $\sigma \in (0, 1)$ depend only on $d$, $\mu$, and $\omega$.
\end{thm}

\begin{proof}
 By rescaling we may assume $\e=1$.
There are three cases:  (1) $B_{3r/2} \subset \omega$;
 (2) $B_{3r/2}\subset \mathbb{R}^d \setminus\overline{ \omega}$;  and 
 (3) $B_{3r/2} \cap \partial \omega \neq \emptyset$.
The first two cases follow readily from the interior H\"older estimates for the 
elliptic operator $-\text{div}(A(x)\nabla)$.
 To treat the third case, without loss of generality, we assume that $r$ is small and $x_0 \in \partial\omega$.
 By a change of variables, the desired estimates follow  from Lemma \ref{lemma-local-1}.
\end{proof}


\section{\bf Convergence rates in $H^1$}\label{section-3}

Let $\Omega$ be a bounded Lipschitz domain in $\mathbb{R}^d$.
In this section we establish a sub-optimal convergence rate in $H^1(\Omega)$,
without the condition \eqref{F-k}.
Let  $u_{\e, \delta} \in H^1(\Omega)$ and $v_\delta \in H^1(\Omega) \cap H^2_{\loc} (\Omega)$.
Suppose that
\begin{equation}\label{3-1}
\left\{
\aligned
\text{\rm div} ({A^\e_{ \delta}} \nabla u_{\e, \delta})
& = \text{\rm div} (\widehat{A_\delta} \nabla v_\delta) & \quad & \text{ in } \Omega,\\
u_{\e, \delta}  &= v_\delta & \quad & \text{ on } \partial \Omega,
\endaligned
\right.
\end{equation}
where $\widehat{A_\delta}$ is the homogenized matrix  given by (\ref{h-matrix-1}).
Let
\begin{equation}\label{D-t}
\Sigma_t =\big\{ x\in \Omega: \ \text{dist}(x, \partial\Omega)< t  \big\}.
\end{equation}
Consider the function 
\begin{equation}\label{r-1}
w_{\e, \delta} =u_{\e, \delta} -v_\delta -\e \chi_\delta (x/\e) S_\e (\eta_\e (\nabla v_\delta)),
\end{equation}
where $\chi_\delta$ is the corrector defined by (\ref{cor-1}) and $\eta_\e\in C_0^\infty(\Omega)$  is a cut-off function satisfying 
$0\le \eta_\e \le 1$, $|\nabla \eta_\e|\le C/\e$, and
\begin{equation}\label{cut-off}
\left\{
\aligned
& \eta_\e (x) =1 & \quad & \text{ if } x\in \Omega \setminus \Sigma_{4d \e},\\
& \eta_\e (x) =0 & \quad & \text{ if  } x\in \Sigma_{3d\e}.
\endaligned
\right.
\end{equation}
The smoothing operator $S_\e$ in (\ref{r-1}) is defined by
\begin{equation}\label{S}
S_\e (f) (x) =\int_{\mathbb{R}^d} f(x-y) \varphi_\e (y)\, dy,
\end{equation}
where  $\varphi_\e (y) =\e^{-d} \varphi (y/\e)$ and $\varphi$ is a (fixed) function in $C_0^\infty(B(0, 1/2))$ such that
$\varphi\ge 0$ and $ \int_{\mathbb{R}^d} \varphi\, dx=1$.

\begin{lemma}
Let $S_\e$ be the operator  given by (\ref{S}). 
Let $1\le p<\infty$.
Then
\begin{equation}\label{S1}
\aligned
\| g^\e S_\e (f)\|_{L^p(\mathbb{R}^d)}
& + \e \| g^\e \nabla S_\e (f)\|_{L^p(\mathbb{R}^d)}\\
&\le C \sup_{x\in \mathbb{R}^d} \left(\fint_{B(x, 1/2)}
|g|^p\right)^{1/p}
\| f\|_{L^p(\mathbb{R}^d)},
\endaligned
\end{equation}
where $g^\e (x)=g(x/\e)$, $C$ depends only on $d$ and  $p$, and
\begin{equation}\label{S2}
\| f-S_\e (f)\|_{L^p(\mathbb{R}^d)}
\le \e \| \nabla f\|_{L^p(\mathbb{R}^d)}.
\end{equation}
\end{lemma}

\begin{proof}
See  e.g. \cite[pp.37-38]{Shen-book}.
\end{proof}

\begin{lemma}\label{thm-3.1}
Let $\Omega$ be a bounded Lipschitz domain.
For $0<\e\le 1$ and  $0\le \delta\le 1$, 
let $u_{\e, \delta}$, $v_\delta$, $w_{\e, \delta}$ be given  as above.
Then, for any $\psi \in H_0^1(\Omega)$,
\begin{equation}\label{3.1-0}
\Big|\int_\Omega A^\e_{ \delta} \nabla w_{\e, \delta}
\cdot \nabla \psi \, dx \Big|
\le C \|\nabla \psi\|_{L^2(\Omega)}
\Big\{
\| \nabla v_\delta \|_{L^2 (\Sigma_{4d\e})}
+ \e \|\nabla^2 v_\delta\|_{L^2(\Omega\setminus \Sigma_{3d\e})}\Big\},
\end{equation}
where $C$ depends only on $d$, $\mu$, $\omega$,  and the Lipschitz character of $\Omega$.
\end{lemma}

\begin{proof}
The proof is similar to that for the case $\delta=1$ (see e.g. \cite{Shen-book}).
 Let $\psi\in H_0^1(\Omega)$.
Using 
\begin{equation}\label{3.1-1}
\int_\Omega A^\e_\delta \nabla u_{\e, \delta}
\cdot \nabla \psi \, dx
=\int_\Omega \widehat{A_\delta} \nabla v_\delta \cdot \nabla \psi \, dx,
\end{equation}
we obtain
\begin{equation}\label{3.1-2}
\aligned
\int_\Omega A^\e_{ \delta} \nabla w_{\e, \delta}
\cdot \nabla \psi \, dx 
&= \int_\Omega  \left( \widehat{A_\delta} - A^\e_{ \delta} \right) \nabla v_\delta \cdot \nabla \psi\, dx \\
&\qquad -\int_\Omega A^\e_{ \delta}
\nabla \chi_\delta (x/\e) S_\e (\eta_\e (\nabla v_\delta)) \cdot \nabla \psi \, dx\\
&\qquad  -\e \int_\Omega A^\e_{ \delta} 
\chi_\delta (x/\e) \nabla S_\e (\eta_\e (\nabla v_\delta))\cdot \nabla \psi\, dx\\
&=\int_\Omega  \left( \widehat{A_\delta} - A^\e_{ \delta} \right)  \left( \nabla v_\delta
-S_\e (\eta_\e (\nabla v_\delta)\right)  \cdot \nabla \psi\, dx \\
&\qquad -\int_\Omega
\left( A^\e_{ \delta} + A^\e_{ \delta} \nabla \chi_\delta (x/\e) -\widehat{A_\delta} \right)
S_\e (\eta_\e (\nabla v_\delta)) \cdot \nabla \psi\, dx\\
&\qquad  -\e \int_\Omega A^\e_{ \delta} 
\chi_\delta (x/\e) \nabla S_\e (\eta_\e (\nabla v_\delta))\cdot \nabla \psi\, dx\\
&=I_1 +I_2 +I_3.
\endaligned
\end{equation}
By the Cauchy inequality and (\ref{S2}),
\begin{equation}\label{3.1-3}
\aligned
|I_1|
&\le C \|\nabla v_\delta -S_\e (\eta_\e (\nabla v_\delta)) \|_{L^2(\Omega)} \|\nabla \psi\|_{L^2(\Omega)}\\
&\le C \e \| \nabla (\eta_\e (\nabla v_\delta) )\|_{L^2(\Omega)} \|\nabla \psi\|_{L^2(\Omega)}
+ C\| \nabla v_\delta \|_{L^2(\Sigma_{4d\e})} \| \nabla \psi\|_{L^2(\Omega)}\\
&\le C \|\nabla \psi\|_{L^2(\Omega)}
\big\{ \|\nabla v_\delta\|_{L^2(\Sigma_{4d\e})}
+\e \|\nabla^2 v_\delta\|_{L^2(\Omega\setminus \Sigma_{3d\e})} \big\}.
\endaligned
\end{equation}
It follows from  the Cauchy inequality,  (\ref{S1}) and (\ref{2.1-0}) that 
\begin{equation}\label{3.1-4}
\aligned
I_3
&\le C \e \| \chi_\delta (x/\e) \nabla S_\e (\eta_\e (\nabla v_\delta)) \|_{L^2(\Omega)}
\|\nabla \psi\|_{L^2(\Omega)} \\
& \le C \e \| \nabla (\eta_\e (\nabla v_\delta))\|_{L^2(\Omega)} \|\nabla \psi\|_{L^2(\Omega)}\\
&\le 
 C \|\nabla \psi\|_{L^2(\Omega)}
\big\{ \|\nabla v_\delta\|_{L^2(\Sigma_{4d \e})}
+\e \|\nabla^2 v_\delta\|_{L^2(\Omega\setminus \Sigma_{3d \e})} \big\}.
\endaligned
\end{equation}

Finally, to bound$I_2$, we note that
$$
A^\e_{ \delta} +A^\e_{ \delta} \nabla \chi_\delta (x/\e) -\widehat{A_\delta} 
=B_\delta(x/\e),
$$
where the matrix $B_\delta$ is given by (\ref{B}). 
Since supp$(S_\e (\eta_\e (\nabla v_\delta))\subset \Omega\setminus \Sigma_{2d\e}$,
it follows by Lemma \ref{lemma-2.3}  that
$$
\aligned
I_2
&= - \int_\Omega b_{\delta, ij} (x/\e)  S_\e \left( \eta_\e \frac{\partial v_\delta}{\partial x_j}\right)  \frac{\partial \psi}{\partial x_i} \, dx\\
&= - \e \int_\Omega
\frac{\partial}{\partial x_k} \Big( \phi_{\delta, kij} (x/\e)  \Big)
S_\e \left( \eta_\e \frac{\partial v_\delta}{\partial x_j}\right)  \frac{\partial \psi}{\partial x_i} \, dx\\
&=\e \int_\Omega
\phi_{\delta, kij} (x/\e)
\frac{\partial }{\partial x_k} S_\e \left( \eta_\e \frac{\partial v_\delta}{\partial x_j}\right)  \frac{\partial \psi}{\partial x_i }\, dx.
\endaligned
$$
Thus, as in the case of $I_3$,
$$
\aligned
|I_2|
&\le C \e \| \nabla (\eta_\e (\nabla v_\delta))\|_{L^2(\Omega)}
\|\nabla \psi\|_{L^2(\Omega)}\\
&\le C  \|\nabla \psi\|_{L^2(\Omega)}
\big\{ \|\nabla v_\delta\|_{L^2(\Sigma_{4d\e})}
+\e \|\nabla^2 v_\delta\|_{L^2(\Omega\setminus \Sigma_{3d\e})} \big\}.
\endaligned
$$
This, together with (\ref{3.1-3}) and (\ref{3.1-4}), completes the proof.
\end{proof}

\begin{lemma}\label{lemma-3.2}
Let $\Omega$ be a bounded Lipschitz domain with connected boundary.
Suppose $v_\delta\in H^1(\Omega)$ is a weak solution of
$\text{\rm div} (\widehat{A_\delta} \nabla v_\delta)=0$ in $\Omega$
with $v_\delta=f \in H^1(\partial\Omega)$.
Then,  for $0<t<\text{diam}(\Omega)$, 
\begin{equation}\label{3.2-0}
\| \nabla v_\delta\|_{L^2(\Sigma_t)}
+ t \| \nabla^2 v_\delta\|_{L^2(\Omega\setminus \Sigma_t)}
\le C t^{1/2} \| \nabla_{\tan} f\|_{L^2(\partial\Omega)},
\end{equation}
where $C$ depends only on $d$, $\mu$, and the Lipschitz character of $\Omega$.
\end{lemma}

\begin{proof}
This follows from the interior estimates and the nontangential-maximal-function estimate
for second-ordr elliptic equations with constant coefficients,
\begin{equation}
\| N (\nabla v_\delta) \|_{L^2(\partial\Omega)} 
\le C \| \nabla_{\tan} f \|_{L^2(\partial\Omega)}.
\end{equation}
See \cite{Kenig-book} for references.
\end{proof}

\begin{thm}\label{thm-3.3}
Let $\Omega$ be a bounded Lipschitz domain with connected boundary.
Assume $1\le \text{diam} (\Omega) \le 10$.
Suppose that $u_{\e, \delta}\in H^1(\Omega)$ is a weak solution 
of $\text{ \rm div} (A^\e_{ \delta} \nabla u_{\e, \delta} ) =0$ in $\Omega$
with $ u_{\e, \delta} =f\in H^1(\partial\Omega)$ on $\partial\Omega$.
Let $v_\delta$ be the solution of $\text{\rm div} ( \widehat{A_\delta} \nabla v_\delta)=0$ in $\Omega$
with $v_\delta =f $ on $\partial\Omega$.
Then,  for $0<\e<1$ and $0\le \delta\le 1$, 
\begin{equation}\label{3.3-0}
\| \Lambda^\e_{ \delta}  \nabla w_{\e, \delta} \|_{L^2(\Omega)}
\le C \e^{1/4} \|\nabla_{\tan} f \|^{1/2} _{L^2(\partial\Omega)}
\Big\{  \|\nabla_{\tan} f \|^{1/2} _{L^2(\partial\Omega)}
+ \|\nabla u_{\e, \delta} \|_{L^2(\Omega)}^{1/2} \Big\},
\end{equation}
where $\Lambda_\delta^\e (x)=\Lambda_\delta (x/\e)$,
 $w_{\e, \delta}$ is given by (\ref{r-1}) and
$C$ depends only on $d$, $\mu$, $\omega$,  and the Lipschitz character of $\Omega$.
\end{thm}

\begin{proof}
Let $\psi = w_{\e, \delta}$ in (\ref{3.1-0}).
Since $A^\e_{ \delta} (x)= [\Lambda_{ \delta} (x/\e)]^2 A(x/\e)$,
it follows from (\ref{3.1-0}) and (\ref{3.2-0}) that 
\begin{equation}\label{3.3-1}
\| \Lambda^\e_{ \delta} \nabla w_{\e, \delta} \|^2_{L^2(\Omega)}
\le C
\e^{1/2} \| \nabla_{\tan} f\|_{L^2 (\partial\Omega)}  \| \nabla w_{\e, \delta} \|_{L^2(\Omega)}.
\end{equation}
Using (\ref{S1}), it is not hard to see that
$$
\aligned
\|\nabla w_{\e, \delta}\|_{L^2(\Omega)}
& \le C \big\{ \|\nabla u_{\e, \delta} \|_{L^2(\Omega)} + \|\nabla v_\delta\|_{L^2(\Omega)} \big\}\\
&\le C \big\{  \|\nabla u_{\e, \delta} \|_{L^2(\Omega)} +  \| \nabla_{\tan} f \|_{L^2(\partial\Omega)} \big\}.
\endaligned
$$
This, together with (\ref{3.3-1}), gives (\ref{3.3-0}).
\end{proof}


\section{\bf Large-scale interior  estimates}

In this section we investigate large-scale  interior estimates for weak solutions of
\begin{equation}\label{4.1}
\text{\rm div} ( A_\delta ^\e \nabla u_{\e, \delta}) =0
\end{equation}
in $\Omega$, where
$
A_\delta^\e (x)=A_\delta (x/\e) =[\Lambda_{\delta} (x/\e)]^2  A(x/\e)
$
 and  $A$ satisfies conditions (\ref{ellipticity})-(\ref{periodicity}). 
 No smoothness condition for $A$ is needed.
 In the case $\delta=0$, it is  assumed that  $u_{\e, 0} \in H^1(\Omega)$ and
$\text{\rm div} (A(x/\e) \nabla u_{\e, 0})=0$ in $\Omega\setminus \e\omega$.
The main  results in this section were already established in \cite{Chase-Russell-2017, Chase-Russell-2018}.
 We  present a unified and simplified proof here for the reader's convenience.
 A different approach, which also works for the case $2< \delta\le \infty$,  may be found in \cite{Shen-2020}.

Observe  that if $u_{\e, \delta}$ is a solution of (\ref{4.1}) and $v(x)=u_{\e, \delta} (rx)$,
then $\text{\rm div} (A^{\e/r}_\delta \nabla v)=0$.
This rescaling property will be used repeatedly.
Let $B_r=B(x_0, r)$ for some fixed $x_0\in \mathbb{R}^d$.
The goal of this section is to prove the following.

\begin{thm}\label{thm-IL}
Suppose $A$ satisfies \eqref{ellipticity} and \eqref{periodicity}.
Let $u_{\e, \delta}\in H^1(B_R)$ be a weak solution of \eqref{4.1}  in $B_R$
for some $R> (10d) \e$.
Then, for $0\le \delta\le 1$ and $\e\le r<R/2$,
\begin{equation}\label{IL-0}
\left(\fint_{B_r} |\nabla u_{\e, \delta}|^2 \right)^{1/2}
\le C \left(\fint_{B_R\cap \e \omega}  |\nabla u_{\e, \delta}|^2 \right)^{1/2},
\end{equation}
where $C$ depends only on $d$, $\mu$, and $\omega$.
\end{thm}

\begin{lemma}\label{lemma-4.2}
Suppose $A$ satisfies \eqref{ellipticity}.
Let $u_{\e, \delta}$ be a weak solution of  \eqref{4.1} in $B_{2r}$ and $r\ge 4d \e$.
Then, for $0\le \delta\le 1$,
\begin{align}
\int_{B_r}
|\nabla u_{\e, \delta}|^2\, dx 
& \le C\int_{B_{3r/2}\cap \e \omega} |\nabla u_{\e, \delta}|^2 \, dx\label{4.2-3},\\
\int_{B_r }
|\nabla u_{\e, \delta}|^2\, dx 
& \le \frac{C}{r^2} \int_{B_{2r}\cap \e \omega} | u_{\e, \delta}|^2 \, dx\label{4.2-2},\\
\int_{B_r}
|u_{\e, \delta}|^2\, dx 
& \le C\int_{B_{2r}\cap \e \omega} |u_{\e, \delta}|^2 \, dx\label{4.2-1},
\end{align}
where $C$ depends only on $d$, $\mu$, and $\omega$.
\end{lemma}

\begin{proof}
To show  (\ref{4.2-3}),
by rescaling, we may assume $\e=1$.
Recall that $\mathbb{R}^d\setminus \omega=\cup_k \overline{F_k}$,
where $\overline{F_k}$'s are connected components of $\mathbb{R}^d\setminus \omega$.
Suppose $F_k \cap B_r \neq \emptyset$.
Since diam$(\widetilde{F}_k)\le d$ and  $r\ge 4d$,  $\widetilde{F}_k \subset B_{3r/2}$. 
Note that  $\text{\rm div} (A\nabla u_{1, \delta}) =0$ in $F_k$. By Lemma \ref{lemma-ext-F},  we obtain 
\begin{equation}\label{4.2-5}
\|\nabla u_{1, \delta}\|_{L^2(F_k)}
\le C \|\nabla u_{1, \delta}\|_{L^2(\widetilde{F}_k\setminus F_k)},
\end{equation}
which yields (\ref{4.2-3}) for $\e=1$,
using the fact  that  $\widetilde{F}_k$'s are disjoint.

To prove (\ref{4.2-2}), we first show that
\begin{equation}\label{4.2-11}
\delta^2 \int_{B_{3r/2}\setminus \e\omega} |u_{\e, \delta}|^2 \, dx 
\le C  \int_{B_{2r}\cap \e\omega} |u_{\e, \delta} |^2\, dx.
\end{equation}
By rescaling we may assume $\e=1$.
Suppose $B_{3r/2}\cap F_k\neq \emptyset$ for some $k$.
It follows from (\ref{4.1-0}) that
$$
\int_{\widetilde{F}_k}
|\Lambda_\delta\nabla (u_{1, \delta} \varphi)|^2\, dx
\le C \int_{\widetilde{F}_k} |\Lambda_\delta u_{1, \delta} |^2 |\nabla \varphi|^2\, dx
$$
for any $\varphi \in C_0^1(\widetilde{F}_k)$.
Since $0\le \delta\le 1$, we obtain
$$
\delta^2 \int_{\widetilde{F}_k}
|\nabla (u_{1, \delta} \varphi)|^2\, dx
\le C \int_{\widetilde{F}_k} | u_{1, \delta} |^2 |\nabla \varphi|^2\, dx.
$$
Choose $\varphi\in C_0^1(\widetilde{F}_k)$ such that $\varphi=1$ on $F_k$.
By Poincar\'e's  inequality, 
\begin{equation}\label{4.2-00}
\delta^2 \int_{F_k} |u_{1,\delta}|^2\, dx 
\le C \int_{\widetilde{F}_k\setminus F_k} |u_{1, \delta}|^2\, dx.
\end{equation}
Hence,
\begin{equation}{\label{4.2-4}}
\delta^2 \int_{B_{3r/2}\setminus \omega} |u_{1, \delta}|^2\, dx
\le  C \int_{B_{3r/2+2 {d} }\cap \omega} |u_{1, \delta}|^2\, dx,
\end{equation}
which gives (\ref{4.2-11}) for $\e=1$, since $3r/2 +2d\le 2r$.
It follows from  (\ref{4.1-0}) and (\ref{4.2-11}) that  
\begin{equation}\label{4.2-12}
\aligned
\int_{B_{5r/4}\cap \e\omega} |\nabla u_{\e, \delta}|^2\, dx
&\le \frac{C}{r^2}
\left\{ \int_{B_{3r/2}\cap \e\omega} |u_{\e, \delta}|^2\, dx
+ \delta^2 \int_{B_{3r/2}\setminus \e\omega} |u_{\e, \delta}|^2 \, dx \right\}\\
&\le \frac{C}{r^2}
\int_{B_{2r}\cap \e\omega} |u_{\e, \delta} |^2\, dx.
\endaligned
\end{equation}

Finally, to prove (\ref{4.2-1}), 
we again assume $\e=1$.
Suppose $F_k \cap B_r\neq \emptyset$.
By Poincar\'e's inequality,
$$
\int_{F_k} |u_{1, \delta}|^2\, dx
\le C \int_{\widetilde{F}_k\setminus F_k} |u_{1, \delta}|^2\, dx
+ C \int_{\widetilde{F}_k} |\nabla u_{1, \delta}|^2\, dx.
$$
It follows that
$$
\aligned
\int_{B_r} |u_{1, \delta}|^2\, dx
&\le C\int_{B_{5r/4}\cap \omega} |u_{1, \delta} |^2\, dx
+ C \int_{B_{5r/4}} |\nabla u_{1, \delta}|^2\, dx\\
& \le C \int_{B_{2r}\cap \omega} |u_{1, \delta}|^2\, dx,
\endaligned
$$
where we have used (\ref{4.2-12}) as well as the fact $r\ge 1$ for the last inequality.
\end{proof}

The next lemma provides an approximation of $u_{\e, \delta}$ by solutions of
elliptic equations with constant coefficients.

\begin{lemma}\label{lemma-4.5}
Suppose $A$ satisfies \eqref{ellipticity} and \eqref{periodicity}.
Let $u_{\e, \delta}\in H^1(B_{2})$ be a weak solution of
$\text{\rm div} (A_\delta^\e \nabla u_{\e, \delta})=0$ in $B_{2}$.
Then there exists $v\in H^1(B_1)$ such that
$\text{\rm div} (\widehat{A_\delta} \nabla v)=0$ in $B_1$ and
\begin{equation}
\left(\fint_{B_1\cap \e\omega} | u_{\e, \delta} -v|^2\right)^{1/2}
\le  C \e^{1/4}
\left(\fint_{B_{2}\cap \e\omega } |u_{\e, \delta} |^2 \right)^{1/2},
\end{equation}
where $0< \e<1$, $0\le \delta\le 1$,  and $C>0$ depends only on $d$, $\mu$, and $\omega$.
\end{lemma}

\begin{proof}
We assume $\e>0$ is sufficiently small; for otherwise the estimate is trivial with $v=0$.
By Lemma \ref{lemma-4.2} (or rather  its proof),
\begin{equation}\label{4.5-10}
\int_{B_{7/4}} |\nabla u_{\e, \delta}|^2 \, dx \le C \int_{B_2\cap \e\omega}| u_{\e, \delta}|^2\, dx.
\end{equation}
It follows that there exists some $t\in (3/2, 7/4)$ such that 
\begin{equation}\label{4.5-11}
\int_{\partial B_t} |\nabla u_{\e, \delta}|^2\, d\sigma
\le C \int_{B_2\cap \e\omega}| u_{\e, \delta}|^2\, dx.
\end{equation}
Let $ v\in H^1(B_t)$ be the weak solution of
$\text{\rm div} (\widehat{A_\delta} \nabla v )=0$ in $B_t$
with $v=u_{\e, \delta}$ on $\partial B_t$.
Let $w_{\e, \delta}$ be defined by (\ref{r-1}), with $v$ in the place of $v_\delta$.
In view of Theorem \ref{thm-3.3}, we have 
\begin{equation}\label{4.5-12}
\aligned
\|\nabla w_{\e, \delta} \|_{L^2(B_t\cap \e\omega)}
&\le 
\|\Lambda^\e_\delta \nabla w_{\e, \delta} \|_{L^2(B_t)}\\
& \le C \e^{1/4}
\|\nabla u_{\e,\delta}\|^{1/2}_{L^2(\partial B_t)}
\Big\{ \|\nabla u_{\e,\delta}\|^{1/2}_{L^2(\partial B_t)}
+ \|\nabla u_{\e, \delta}\|^{1/2}_{L^2(B_t)} \Big\}\\
&\le  C \e^{1/4}
\| u_{\e, \delta}\|_{L^2(B_2\cap \e\omega)},
\endaligned
\end{equation}
where we have used (\ref{4.5-10}) and (\ref{4.5-11}) for the last inequality.
Since $w_{\e, \delta} =0$ on $\partial B_t$, by Poincar\'e's inequality
\begin{equation}\label{ext-1}
\| \phi \|_{L^2(\Omega \cap \e\omega)} \le C \| \nabla \phi \|_{L^2(\Omega\cap \e\omega)}
\end{equation}
for any $\phi \in H^1_0(\Omega)$, 
 and (\ref{4.5-12}),
$$
\|w _{\e, \delta} \|_{L^2(B_t\cap \e\omega)}
\le C \e^{1/4}
\| u_{\e, \delta}\|_{L^2(B_2\cap \e\omega)}.
$$
It follows that
$$
\aligned
\| u_{\e, \delta} - v\|_{L^2(B_1\cap \e \omega)}
 & \le  \| w_{\e, \delta}\|_{L^2(B_1\cap \e \omega) }
+ C \e  \| \nabla v\|_{L^2(B_t)} \\
& \le C \e^{1/4}
\| u_{\e, \delta}\|_{L^2(B_2\cap \e\omega)}.
\endaligned
$$

Finally, to see \eqref{ext-1},
we first extend $\phi $  from $\Omega$ to $\mathbb{R}^d$ by zero.
Let $\widetilde{\phi}=\phi|_{\e \omega}$.
We then extend $\widetilde{\phi}$ from $\e \omega$  to $\mathbb{R}^d$ such that for each $k$,
$$
\|\nabla \widetilde{\phi}\|_{L^2(\e F_k)}
\le C \| \nabla \phi\|_{L^2(\e \widetilde{F}_k \setminus \e F_k)},
$$
where $C$ depends on $\kappa$ and  the Lipschitz character of $F_k$.
It follows that
$
\|\nabla \widetilde{\phi}\|_{L^2(\mathbb{R}^d)}
\le C \|\nabla \phi \|_{L^2(\Omega\cap \e \omega)}.
$
Note that $\widetilde{\phi }$ has compact support.
Hence,
$$
\| \phi \|_{L^2(\Omega\cap \e\omega)}
\le \| \widetilde{\phi }\|_{L^2(\Omega)}
\le C \|\nabla \widetilde{\phi}\|_{L^2(\mathbb{R}^d)}
\le C \|\nabla \phi \|_{L^2(\Omega\cap \e \omega)},
$$
where we have used the Poincar\'e inequality for the second inequality.
\end{proof}

\begin{lemma}\label{lemma-4.3}
Suppose $A$ satisfies \eqref{ellipticity}-\eqref{periodicity}  and $0\le \delta\le 1$.
Let $v\in H^1(B_1)$ be a solution of $\text{\rm div}(\widehat{A_\delta}\nabla v)=0$ in
$B_1$.
Then, for any $0<\theta, \e < 1/(8d)$,
\begin{equation}\label{4.3-0}
\inf_P \frac{1}{\theta}
\left(\fint_{B_\theta} |v- P|^2\right)^{1/2}
\le C_0\theta\, 
\inf_P \left(\fint_{B_1\cap \e\omega} |v-P|^2 \right)^{1/2},
\end{equation}
where $P$ is a linear function and $C_0$ depends only on $d$, $\mu$, and $\omega$.
\end{lemma}

\begin{proof}
It follows from (\ref{4.2-1}) with $\delta=1$ and $A$ being a constant matrix that
$$
\int_{B_{1/2}} |v|^2\, dx \le C \int_{B_1\cap \e\omega} |v|^2\, dx.
$$
Thus, by using the $C^2$ estimates for second-order elliptic equations with 
constant coefficients,
$$
\aligned
\inf_P \frac{1}{\theta}
\left(\fint_{B_\theta} |v- P|^2\right)^{1/2}
&\le \theta \|\nabla^2 v\|_{L^\infty(B_\theta)}
\le  \theta \|\nabla^2 v\|_{L^\infty(B_{1/4})}\\
&\le C \theta \left(\fint_{B_{1/2}} |v|^2\right)^{1/2}\\
&\le C \theta \left(\fint_{B_1\cap \e\omega} |v|^2\right)^{1/2},
\endaligned
$$
from which the estimate (\ref{4.3-0}) follows, as $v-P$ is also a solution.
\end{proof}

For $u\in L^2(B_r)$, define
\begin{equation}\label{H}
H( r; u)=\inf_P \frac{1}{r}
\left(\fint_{B_r\cap \e\omega} | u -P|^2\right)^{1/2},
\end{equation}
where the infimum is taken over all linear functions $P$.

\begin{lemma}\label{lemma-4.6}
Suppose $A$ satisfies \eqref{ellipticity}-\eqref{periodicity}.
Let $u_{\e, \delta} \in H^1(B_{2r})$ be a weak solution of 
$\text{\rm div} (A_\delta^\e \nabla u_{\e, \delta}) =0$ in $B_{2r}$ for some
$r> \e$.
Then there exists $\theta \in (0, 1/4)$, depending only on $d$, $\mu$ and $\omega$, such that
\begin{equation}\label{4.6-0}
H(\theta r; u_{\e, \delta})
\le \frac12
H(r; u_{\e,\delta})
+ C \left(\frac{\e}{r}\right)^{1/4}
\inf_{q\in \mathbb{R}}
\left(\fint_{B_{2r}\cap \e\omega}
|u_{\e, \delta} -q|^2 \right)^{1/2},
\end{equation}
where $C$ depends only on $d$, $\mu$, and $\omega$.
\end{lemma}

\begin{proof}
By rescaling we may assume $r=1$.
We may further assume that $\e>0$ is sufficiently small; for otherwise the estimate is trivial.
Given $u_{\e, \delta}$ in $B_2$,
 let $v$ be the solution of $\text{\rm div} (\widehat{A_\delta}\nabla v)=0$ in $B_1$,
given by Lemma \ref{lemma-4.5}.
Note that
$$
\aligned
H(\theta; u_{\e, \delta})
&\le H(\theta; v) + \frac{1}{\theta} \left(\fint_{B_\theta\cap \e\omega} |u_{\e, \delta} -v|^2 \right)^{1/2}\\
&\le C_0 \theta H(1; v) + \frac{1}{\theta} \left(\fint_{B_\theta\cap \e\omega} |u_{\e, \delta} -v|^2 \right)^{1/2}\\
& \le C_0\theta H(1; u_{\e, \delta})
+ C_\theta \left(\fint_{B_1 \cap \e\omega} |u_{\e, \delta} -v|^2 \right)^{1/2}\\
&\le C_0\theta H(1; u_{\e, \delta})
+ C_\theta \e^{1/4}  \left(\fint_{B_2\cap \e\omega} |u_{\e, \delta}|^2 \right)^{1/2},
\endaligned
$$
where we have used Lemma \ref{lemma-4.3} for the second inequality and
Lemma \ref{lemma-4.5} for the last.
Since $u_{\e, \delta}-q$ is also a solution of (\ref{4.1}) for any $q\in \mathbb{R}$, we see that 
$$
H(\theta; u_{\e, \delta})
\le C_0\theta H(1; u_{\e, \delta})
+ C_\theta   \e^{1/4} \inf_{q\in \mathbb{R}}
\left(\fint_{B_2\cap \e\omega} |u_{\e, \delta}-q |^2 \right)^{1/2}.
$$
The proof is complete by choosing $\theta$ so small that $C_0 \theta\le (1/2)$.
\end{proof}

The proof of the following lemma may be found in \cite[pp.157-158]{Shen-book}.

\begin{lemma}\label{g-lemma}
Let $H(r)$ and $h(r)$ be two nonnegative, continuous functions on  $(0, 1]$.
Let $0<\e< (1/4)$.
Suppose that there exists a constant $C_0$ such that
\begin{equation}\label{g-1}
\max_{r\le t, s\le 2r} H(t) \le C_0 H(2r)
\quad \text{ and } \quad
\max_{r\le t, s \le 2r}
|h(t)-h(s)| \le C_0 H(2r)
\end{equation}
for any $r\in [\e, 1/2]$.
Suppose further that
\begin{equation}\label{g-2}
H(\theta r) \le \frac12 H(r) + C _0 \beta (\e/r)
\Big\{ H(2r) + h(2r) \Big\}
\end{equation}
for any $r\in [\e, 1/2]$, where $\theta \in (0, 1/4)$ and $\beta (t)$ is a nonnegative, nondecreasing function on $[0, 1]$
such that $\beta(0)=0$ and
\begin{equation}\label{g-3}
\int_0^1 \frac{\beta (t)}{t} \, dt < \infty.
\end{equation}
Then
\begin{equation}\label{g-4}
\max_{\e\le r\le 1}
\big\{ H(r) + h(r) \big\}
\le C \big\{ H(1) + h(1) \big\},
\end{equation}
where $C$ depends only on $C_0$, $\theta$, and the function $\beta (t)$.
\end{lemma}

We are now in a position to give the proof of Theorem \ref{thm-IL}.

\begin{proof}[\bf Proof of Theorem \ref{thm-IL}]
By rescaling we may assume $R=2$.
We may also assume that $\e>0$ is sufficiently small; for otherwise, the estimate is trivial.
Let $\text{\rm div}(A^\e_\delta\nabla u_{\e, \delta}) =0$ in $B_2$.
We shall apply Lemma \ref{g-lemma} with the function $H(t)= H (t; u_{\e, \delta})$ and
$h(t) = |E_t|$ for $(10d)  \e < t<1$, where
 $H(t; u_{\e, \delta})$ is defined by (\ref{H}) and $E_t$ is a vector in $\mathbb{R}^d$ such that
$$
H(t; u_{\e, \delta})
=\inf_{q\in \mathbb{R}}
\frac{1}{t} \left(\fint_{B_t\cap \e\omega} |u_{\e, \delta}  -E_t \cdot x -q|^2 \right)^{1/2}.
$$
The first inequality in (\ref{g-1}) follows from the observation 
$|B_r \cap \e \omega|\approx  r^d$ if $r\ge (10 d)\e $.
To see the second, we note that estimate (\ref{4.2-2}) gives
$$
|E_t-E_s|
\le  C \inf_{q\in \mathbb{R}}
\frac{1}{r} 
\left(\fint_{B_r \cap \e \omega} |(E_t -E_s)\cdot x -q|^2 \right)^{1/2}
$$
if $r\ge (8 d)\e$.
Furthermore, the condition  (\ref{g-2})  is given by Lemma \ref{lemma-4.6} with $\beta (t) = t^{1/4}$.
Consequently, by Lemma \ref{g-lemma}, we obtain 
$$
\aligned
\inf_{q\in \mathbb{R}} \frac{1}{r}
\left(\fint_{B_r \cap \e\omega} 
|u_{\e, \delta} -q|^2 \right)^{1/2}
&\le H(r) +h(r) \\
&\le C \big\{ H(1) + h(1) \big\}\\
&\le C \left(\fint_{B_1} |u_{\e,\delta}|^2 \right)^{1/2}
\endaligned
$$
for any $r\in ((10d) \e, 1)$. By replacing $u_{\e, \delta}$ with
$u_{\e, \delta} -q$, where $q$ is the average of $u_{\e, \delta}$ over $B_1$,
and using Poincar\'e's inequality, we see that
$$
\inf_{q\in \mathbb{R}} \frac{1}{r}
\left(\fint_{B_r \cap \e\omega} 
|u_{\e, \delta} -q|^2 \right)^{1/2}
\le C  \left(\fint_{B_1} |\nabla u_{\e,\delta}|^2 \right)^{1/2}.
$$
for any $r\in ((10d) \e, 1)$.
This, together with (\ref{4.2-2}), yields
$$
\left(\fint_{B_r} |\nabla u_{\e, \delta}|^2\right)^{1/2}
\le C 
 \left(\fint_{B_1} |\nabla u_{\e,\delta}|^2 \right)^{1/2}
 \le C  \left(\fint_{B_2\cap \e \omega } |\nabla u_{\e,\delta}|^2 \right)^{1/2},
$$
for $\e \le r<1$, where we have used (\ref{4.2-3}) for the last  inequality.
\end{proof}

As a corollary of Theorem \ref{thm-IL}, we obtain the large-scale $L^\infty$ estimate.

\begin{thm}\label{thm-IE}
Suppose $A$ satisfies \eqref{ellipticity} and \eqref{periodicity}.
Let $u=u_{\e, \delta}\in H^1(B_R)$ be a weak solution of
$\text{\rm div} (A_\delta^\e \nabla u )=0$ in $B_R$
for some $R> (10d) \e$.
Then, for $0\le \delta\le 1$ and $\e\le r<R/2$,
\begin{equation}\label{IL-00}
\left(\fint_{B_r} | u |^2 \right)^{1/2}
\le C \left(\fint_{B_R\cap \e \omega}  | u |^2 \right)^{1/2},
\end{equation}
where $C$ depends only on $d$, $\mu$, and $\omega$.
\end{thm}

\begin{proof}
We may assume $r<R/4$.
By Theorem \ref{thm-IL},
$$
\left(\fint_{B_r} |\nabla u |^2\right)^{1/2}
\le C \left(\fint_{B_{R/2} } |\nabla u |^2\right)^{1/2}
$$
Let $J\ge 1$ be the integer such that $2^{J} r \le R/2 <2^{J+1} r $.
Note that for $1\le j \le J$,
$$
\aligned
\Big| \fint_{B_{2^{j-1} r}} u  -\fint_{B_{2^j r}} u \Big|
& \le C 2^j r \left( \fint_{B_{2^j r}} |\nabla u |^2 \right)^{1/2}\\
& \le C  2^j r \left( \fint_{B_{R/2}} |\nabla u |^2 \right)^{1/2},
\endaligned
$$
where we have used Poincar\'e's inequality for the first inequality 
and Theorem \ref{thm-IL} for the second.
It follows by summation that
$$
\Big| \fint_{B_r} u  -\fint_{B_{2^J r}} u \Big|
\le C R \left( \fint_{B_{R/2}} |\nabla u |^2 \right)^{1/2}.
$$
Hence, by Poincar\'e's inequality and Theorem \ref{thm-IL},
$$
\aligned
\left(\fint_{B_r} |u |^2 \right)^{1/2}
& \le Cr \left(\fint_{B_r} |\nabla u |^2 \right)^{1/2}
+ \Big| \fint_{B_r} u  \Big|\\
&\le CR \left(\fint_{B_{R/2} } |\nabla u |^2 \right)^{1/2}
+ C \left(\fint_{B_{R/2} } |u |^2 \right)^{1/2}\\
&\le C \left(\fint_{B_R\cap \e\omega} |u |^2 \right)^{1/2},
\endaligned
$$
where we have used Lemma \ref{lemma-4.2} for the last inequality.
\end{proof}

\begin{remark}\label{remark-LL}
{\rm 
Let $u$ be a weak solution of $\text{\rm div}(A_\delta^\e \nabla u)=0$ in $B_{2r}$ for some $r>0$.
Then
\begin{equation}\label{LL-1}
\|\Lambda_\delta^\e u\|_{L^\infty(B_{r})}  \le C \left(\fint_{B_{2r} } |\Lambda_\delta^\e u|^2 \right)^{1/2},
\end{equation}
where $C$ depends only on $d$, $\mu$, and $\omega$.
The small-scale case $0<r<100 d\e$ is contained in Theorem \ref{thm-local-2}.
To prove the estimate for the large-scale case $r\ge 100 d\e$, 
let $x\in B_r =B(x_0, r) $. Then
$$
\aligned
|\Lambda_\delta^\e (x) u(x)|  & \le C \left(\fint_{B(x, 10 d\e)} |u|^2 \right)^{1/2}
\le C \left(\fint_{B(x, r)\cap \e \omega} |u|^2\right)^{1/2}\\
& \le C  \left(\fint_{B(x_0, 2r)\cap \e \omega} |u|^2\right)^{1/2}
\le C  \left(\fint_{B(x_0, 2r)\cap \e \omega} |\Lambda_\delta^\e u|^2\right)^{1/2},
\endaligned
$$
where we have used Theorem \ref{thm-local-2} for the first inequality and Theorem \ref{thm-IE} for the second.
By a logarithmic convexity argument, one may deduce from \eqref{LL-1} that 
\begin{equation}\label{LL-2}
\|\Lambda_\delta^\e u\|_{L^\infty(B_{r})}  \le C_p \left(\fint_{B_{2r} } |\Lambda_\delta^\e u|^p \right)^{1/p}
\end{equation}
for any $p>0$, where $C_p$ depends only on $d$, $\mu$, $p$, and $\omega$.
}
\end{remark}


\section{Boundary estimates}

In this section we study the boundary regularity for solutions of \eqref{4.1}.
Let  $\Omega$ be  a bounded Lipschitz domain satisfying \eqref{F-k} 
and  $ \Omega^\e=\Omega\setminus \overline{F^\e}$,
 where $F^\e=\cup_k \e F_k$.
Fix $x_0\in \partial\Omega$.
Let 
\begin{equation}\label{D-D}
D_r=B(x_0, r)\cap \Omega \quad \text{ and  } \quad
\Delta_r = B(x_0, r)\cap \partial\Omega,
\end{equation}
for $0<r<r_0=c_0 \text{\rm diam}(\Omega)$.
We assume $c_0>0$ is sufficiently small so that in a suitable coordinated system, obtained from the 
standard one through  translation and rotation,  $x_0=0$ and
\begin{equation}\label{Lip-g}
\aligned
B(0, (100d) r_0)\cap \Omega & =B(0, (100d) r_0)\cap \big\{ (x^\prime, x_d)\in \mathbb{R}^d: \
x_d >\psi(x^\prime) \big\},\\
B(0, (100d) r_0)\cap \partial\Omega
&= B(0, (100d) r_0) \cap 
\big\{ (x^\prime, x_d)\in \mathbb{R}^d: \
x_d =\psi(x^\prime) \big\},
\endaligned
\end{equation}
where $\psi: \mathbb{R}^{d-1} \to \mathbb{R}$ is a Lipschitz function 
with $\psi (0)=0$ and $\|\nabla \psi\|_\infty \le M_1$. 

Let 
\begin{equation}\label{D-e}
D_r^\e= B(x_0, r) \cap \Omega^\e.
\end{equation}

\begin{lemma}\label{lemma-5.1}
Suppose $A$ satisfies \eqref{ellipticity}-\eqref{periodicity}.
Let $4d \e\le  r< r_0$.
Let $u_{\e, \delta} \in H^1(D_{2r} )$ be a weak solution of \eqref{4.1} in 
$D_{2r} $ with $u_{\e, \delta}=0$ on $\Delta _{2r}$.
Then, for $0\le \delta\le 1$,
\begin{align}
\int_{D_r} |\nabla u_{\e, \delta}|^2\, dx
 & \le C \int_{D_{2r}^\e} |\nabla u_{\e, \delta}|^2\, dx \label{5.1-0},\\
\int_{D_r} |\nabla u_{\e, \delta}|^2 \, dx
 & \le \frac{C}{r^2}
\int_{D_{2r}^\e}
|u_{\e, \delta}|^2\, dx,  \label{5.1-1}\\
\int_{D_r} |u_{\e, \delta}|^2\, dx
&\le C \int_{D_{2r}^\e}
|u_{\e, \delta}|^2\, dx, \label{5.1-2}
\end{align}
where $C$ depends only on $d$, $\mu$, $\omega$, $\kappa$, and $M_1$.
\end{lemma}

\begin{proof}
The proof is similar to that of Lemma \ref{lemma-4.2}.
We point out  that since $u_{\e, \delta}=0$ on $\Delta_{2r}$,
the Caccioppoli inequality, 
\begin{equation}\label{5.1-4}
\int_{D_{2r}} |\Lambda^\e_{ \delta }  \nabla (u_{\e, \delta}
\varphi)|^2\, dx
\le C \int_{D_{2r}} |\Lambda^\e_{\delta}  u_{\e, \delta} |^2 |\nabla \varphi|^2\, dx,
\end{equation}
holds for any $\varphi \in C_0^1(B_{2r})$.
Also observe  that by \eqref{F-k},   if $D_r\cap \e F_k\neq \emptyset$,
then $\e \widetilde{F}_k\subset D_{r+d\e }$.
We omit the details.
\end{proof}

\begin{remark}\label{remark-Ca}
Since $\Sigma_{\kappa \e} \subset \Omega^\e$, 
it follows from (\ref{5.1-4}) that  (\ref{5.1-1}) also holds for $0< r\le \kappa \e/2$.
If $\kappa \e/2 < r< 4d\e$, we note that 
$$
\aligned
\int_{D_r} |\nabla u_{\e, \delta}|^2\, dx
&\le \int_{D_{4d\e} } |\nabla u_{\e, \delta}|^2\, dx
\le \frac{C}{r^2}
\int_{D^\e_{8d\e}} |u_{\e, \delta}|^2\, dx\\
&\le \frac{C}{r^2}\int_{D^\e_{cr}} |u_{\e, \delta}|^2\, dx,
\endaligned
$$
where $c>1$ depends on $d$ and $\kappa$.
\end{remark}

\begin{lemma}\label{lemma-5.2}
Suppose $A$ satisfies \eqref{ellipticity}-\eqref{periodicity}.
Let $u_{\e, \delta} \in H^1(D_{2r})$ be a weak solution of (\ref{4.1}) in 
$D_{2r}$ with $u_{\e, \delta}=0$ on $\Delta_{2r}$,
where $(10d)\e<r< r_0$.
Then there exists $v\in H^1(D_r)$ such that 
$\text{\rm div} (\widehat{A_\delta} \nabla v)=0$ in $D_r$, 
$v=0$ on $\Delta_r$, and
\begin{equation}\label{5.2-0}
\left(\fint_{D_r^\e } |u_{\e, \delta} -v|^2 \right)^{1/2}
\le C \left(\frac{\e}{r}\right)^{1/4} \left(\fint_{D_{2r}^\e} |u_{\e, \delta}|^2 \right)^{1/2},
\end{equation}
where  $C$ depends only on $d$, $\mu$, $\omega$, $\kappa$, and $M_1$.
\end{lemma}

\begin{proof}
By rescaling we may assume $r=1$ and $\e>0$ is sufficiently small.
The  argument is similar to that for Lemma \ref{lemma-4.5}, with $D_r$ in the place of $B_r$.
By the proof of Lemma \ref{lemma-5.1},
$$
\int_{D_{7/4}} |\nabla u_{\e, \delta}|^2\, dx
\le C \int_{D_2^\e} |u_{\e, \delta}|^2\, dx.
$$
It follows that there exists some $t\in (3/2, 7/4)$ such that
$$
\int_{\partial D_t} |\nabla u_{\e, \delta}|^2\, d\sigma
\le C 
 \int_{D_2^\e} |u_{\e, \delta}|^2\, dx.
 $$
 Let $v\in H^1(D_t)$ be the weak solution of $\text{\rm div} (\widehat{A_\delta}\nabla v)=0$ in $D_t$
 with $v=u_{\e, \delta}$ on $\partial D_t$. Note that $v=0$ on $\Delta_1$.
 Let $w_{\e, \delta}$ be defined by (\ref{r-1}), with $v$ in the place of $v_\delta$.
 By Theorem \ref{thm-3.3},
 \begin{equation}\label{5.2-1}
 \| \nabla w_{\e, \delta} \|_{L^2(D^\e_{3/2})}
 \le C \e^{1/4} \| u_{\e, \delta} \|_{L^2(D_2^\e)}.
 \end{equation}
 We claim that
 \begin{equation}\label{5.2-2}
 \| w_{\e, \delta} \|_{L^2(D_1^\e)}
 \le C \| \nabla w_{\e, \delta} \|_{L^2(D^\e_{3/2})}.
 \end{equation}
 This, together with (\ref{5.2-1}), gives
 $$
 \aligned
 \| u_{\e, \delta} - v\|_{L^2(D_1^\e)}
 &\le \| w_{\e, \delta} \|_{L^2(D_1^\e)}
 + C \e  \|\nabla v\|_{L^2(D_{t})}\\
 &\le C \e^{1/4} \| u_{\e, \delta} \|_{L^2(D_2^\e)}.
 \endaligned
 $$
 
 Finally, to prove the Poincar\'e-type inequality (\ref{5.2-2}),
 we use an extension argument.
 Observe that if  $\e F_k$ is one of the holes in $F^\e$ such that $\e F_k \cap D_1\neq \emptyset$,
 then $\e\widetilde{F}_k \subset D_{3/2}$.
 We introduce a new function $\widetilde{w}_{\e, \delta}$ by replacing   $w_{\e, \delta}$ in such $\e{F}_k$ by the harmonic 
 extension of $w_{\e, \delta}|_{\partial (\e F_k)}$.
 Note that 
 $$
 \| \nabla \widetilde{w}_{\e, \delta} \|_{L^2(\e F_k)}
 \le C \| \nabla w_{\e, \delta} \|_{L^2(\e\widetilde{F}_k \setminus \e F_k)}.
 $$
 Since $\widetilde{w}_{\e, \delta}=w_{\e, \delta}=0$ on $\Delta_1$, 
 it  follows by Poincar\'e's inequality that 
 $$
 \|\widetilde{w}_{\e, \delta} \|_{L^2(D_1)}
 \le C \| \nabla \widetilde{w}_{\e, \delta} \|_{L^2(D_1)}
 \le C\|\nabla w_{\e, \delta} \|_{L^2(D_{3/2}^\e)},
 $$
Since $\widetilde{w}_{\e, \delta}=w_{\e, \delta}$ on $D_1^\e$, this gives (\ref{5.2-2}).
\end{proof}

With the approximation result in Lemma \ref{lemma-5.2} at our disposal,
we  establish the large-scale H\"older estimate on Lipschitz domains as well as the Lipschitz estimate on 
$C^{1, \alpha}$ domains.

\begin{thm}\label{thm-5.3}
Suppose $A$ satisfies \eqref{ellipticity}-\eqref{periodicity} and $0\le \delta\le 1$.
Let $\Omega$ be a bounded Lipschitz domain satisfying \eqref{F-k}.
Let $u_{\e, \delta}\in H^1(D_R)$ be a weak solution of 
$\text{\rm div} (A^\e_{ \delta} \nabla u_{\e, \delta})=0$ in $D_R$
with $u_{\e, \delta}=0$ on $\Delta_R$
for some $(10d) \e<R<r_0$.
Then, for $\e< r< R/2$,
\begin{equation}\label{5.3-0}
\left(\fint_{D_r} |u_{\e, \delta}|^2 \right)^{1/2}
\le C \left(\frac{r}{R}\right)^\sigma
\left(\fint_{D_R^\e} |u_{\e, \delta} |^2 \right)^{1/2},
\end{equation}
where $\sigma>0$ and $C>0$ depend only on
$d$, $\mu$, $\omega$, $\kappa$, and $M_1$.
\end{thm}

\begin{proof}
By rescaling we may assume $R=2$.
For $(10d)\e< r<1$, let 
$$
\Phi(r)= \frac{1}{r^\sigma}  \left(\fint_{D_r^\e} |u_{\e, \delta}|^2\right)^{1/2}.
$$
Let $v\in H^1(D_{r})$ be a solution of $\text{\rm div} (\widehat{A_\delta}\nabla v)=0$ in $D_{r}$
satisfying $v=0$ on $\Delta_r$ and (\ref{5.2-0}).
Since $\Omega$ is a Lipschitz domain, by the boundary H\"older estimate
for second-order elliptic operators with constant coefficients and (\ref{5.1-2}),
\begin{equation}\label{5.3-2}
\left(\fint_{D^\e_{\theta r}} |v|^2\right)^{1/2}
\le C \theta^\lambda \left(\fint_{D_r^\e} |v|^2 \right)^{1/2}
\end{equation}
for any $\theta\in (0, 1/4)$ such that $\theta r\ge d\e$, 
where $\lambda>0$ and $C>0$ depend only on $d$, $\mu$, $\omega$, $\kappa$, and $M_1$.
It follows that
$$
\aligned
\Phi(\theta r)
&\le  \frac{1}{(\theta r)^\sigma} \left(\fint_{D_{\theta r}^\e} |v|^2 \right)^{1/2}
+  \frac{1}{(\theta r)^\sigma} \left(\fint_{D_{\theta r}^\e} |u_{\e, \delta} -v|^2 \right)^{1/2}\\
& \le C \theta^{\lambda-\sigma } \frac{1}{r^\sigma}  \left(\fint_{D_{ r}^\e} |v|^2 \right)^{1/2}
+\frac{C_\theta}{r^\sigma}   \left( \fint_{D_r^\e} |u_{\e, \delta} -v|^2 \right)^{1/2}\\
&\le \left\{ C \theta^{\lambda-\sigma} 
+ C_\theta \left(\frac{\e }{r}\right)^{1/4}\right\} 
 \Phi (2r).
 \endaligned
 $$
 Fix $\sigma \in (0, \lambda)$.
Choose $\theta\in (0, 1/4)$ so small that  $C \theta^{\lambda-\sigma} \le (1/4)$.
With $\theta$ fixed, we see that if $r\ge N \e$ and $N>1$ is large, 
$$
\Phi (\theta r) \le  (1/2)   \Phi (2r).
$$
By an iteration argument this implies that $\Phi(r) \le C \Phi (1)$ for any $r\in (\e, 1)$.
\end{proof}

 \begin{thm}\label{thm-5.55}
 Suppose that $\text{\rm div} (A^\e_{ \delta} \nabla u_{\e, \delta} ) =0$
 in $B(x_0, r) \cap \Omega$ and
 $u_{\e, \delta}=0$ on $B(x_0, r)\cap \partial\Omega$
 for some $x_0\in \Omega $.
 Assume that $r\ge (4d) \e$.
 Then
 \begin{equation}\label{6.3-0}
 |u_{\e, \delta} (x_0)|
 \le C\min \left( \left(\frac{d(x_0)}{r} \right)^\sigma, 1 \right)
  \left(\fint_{B(x_0, r)\cap \Omega^\e } | u_{\e, \delta} |^2\right)^{1/2},
 \end{equation}
 where $d(x_0)=\text{\rm dist} (x_0, \partial \Omega)$,
  and  $C>0, \sigma\in (0, 1)$ depend only on $d$, $\mu$, $\omega$, $\kappa$, and the Lipschitz character of $\Omega$.
 \end{thm}

\begin{proof}
In the case $d(x_0)\ge r$, we use the large-scale  interior estimate in Theorem \ref{thm-IE},
$$
\left(\fint_{B(x_0, 2d\e) } |u_{\e, \delta} |^2 \right)^{12}
\le C 
\left(\fint_{B(x_0, r)\cap \e\omega   } |u_{\e, \delta} |^2 \right)^{12},
$$
and the small-scale estimate ,
$$
|u_{\e, \delta} (x_0)| \le \left(\fint_{B(x_0, 2d\e) } |u_{\e, \delta} |^2 \right)^{12},
$$
which follows from Theorem \ref{thm-local-2}.
Next, suppose  $d(x_0)\le \kappa \e/2$.
Since $\Sigma_{\kappa \e}  \cap F^\e =\emptyset$, it follows from the small-scale boundary H\"older estimate that
$$
\aligned
|u_{\e, \delta} (x_0)|
 & \le C \left(\frac{ d(x_0)}{\e} \right)^\sigma 
\left( \fint_{B(x_0,  \e) \cap \Omega}
|u_{\e, \delta} |^2 \right)^{1/2}\\
&\le C 
\left(\frac{ d(x_0)}{r} \right)^\sigma 
\left( \fint_{B(x_0, r) \cap \Omega}
|u_{\e, \delta} |^2 \right)^{1/2},
\endaligned
$$
where we have used (\ref{5.3-0}) for the last inequality.
Finally, suppose that $\kappa \e/2 < d(x_0)<r$.
Then
$$
\aligned
|u_{\e, \delta} (x_0)|
&\le C \left(\fint_{B(x_0, d(x_0))} |u_{\e, \delta}|^2 \right)^{1/2}\\
&\le C \left(\frac{d(x_0)}{r}\right)^\sigma 
\left(\fint_{B(x_0, r)\cap \Omega } |u_{\e, \delta}|^2\right)^{1/2},
\endaligned
$$
where we have used (\ref{5.3-0}) for the last inequality.
\end{proof}

\begin{remark}
{\rm
If $\Omega$ is a $C^1$ domain, then the estimate (\ref{5.3-2}) holds for any $\lambda\in (0,1)$.
It follows from the proof of Theorem \ref{thm-5.3} that the large-scale boundary H\"older estimate
(\ref{5.3-0}) holds for any $\sigma \in (0, 1)$.
}
\end{remark}

If $\Omega$ is $C^{1, \alpha}$ for some $\alpha>0$,
we obtain the large-scale boundary Lipschitz estimate.

\begin{thm}\label{thm-5.5}
Suppose $A$ satisfies \eqref{ellipticity}-\eqref{periodicity} and $0\le \delta\le 1$.
Let $\Omega$ be a bounded $C^{1, \alpha}$ domain satisfying \eqref{F-k}.
Let $u_{\e, \delta} \in H^1(D_R)$ be a weak solution of
$\text{\rm div}(A^\e_{ \delta} \nabla u_{\e, \delta} ) =0$ in $D_R$
with $u_{\e, \delta}=0$ on $\Delta_R$ for some $(10d) \e< R< r_0$.
Then for $\e< r< R/2$,
\begin{equation}\label{5.5-0}
\left(\fint_{D_r} |\nabla u_{\e, \delta}|^2 \right)^{1/2}
\le C \left(\fint_{D_R^\e} |\nabla u_{\e, \delta} |^2 \right)^{1/2}
\end{equation}
where $C$ depends only on $d$, $\mu$, $\omega$, $\kappa$, and $\Omega$.
\end{thm}

\begin{proof}
The proof uses the boundary $C^{1, \sigma}$ estimate for second-order
elliptic equations with constant coefficients  in $C^{1, \alpha}$ domains. 
With the approximation result in Lemma \ref{lemma-5.2} at our disposal,
the proof is similar to the case $\delta=1$ in \cite{A-Shen-2016}.
\end{proof}


\section{\bf Estimates of Green's functions}

Let  $\Omega$ be  a bounded Lipschitz domain in $\mathbb{R}^d$, $d\ge 2$, satisfying \eqref{F-k}.
For $0< \delta\le 1$, let $G_{\e, \delta}(x, y)$ denote the Green's function for the
operator $\mathcal{L}_{\e, \delta} =-\text{\rm div} (A^\e_{ \delta} \nabla)$ in $\Omega$, where $A$ satisfies  (\ref{ellipticity})-(\ref{periodicity}).
No smoothness condition on $A$ is needed.
The goal of this section is to prove  the following.

\begin{thm}\label{thm-6.1}
Suppose $d\ge 3$ and $0< \delta\le 1$.
Then,  for $x, y\in \Omega$ and $x\neq y$,
\begin{equation}\label{G-estimate}
|G_{\e, \delta} (x, y)|\le \frac{C}{|x-y|^{d-2}}
\quad \text{ if } \ x, y\in \Omega^\e \text{ or } 
|x-y|\ge 4d\e,
\end{equation}
where 
$C$  depends only on $d$, $\mu$, $\omega$,  $\kappa$, and the Lipschitz 
character of $\Omega$.
\end{thm}
 
 We begin with some energy estimates.
 
\begin{lemma}\label{lemma-6.0}
Suppose $d\ge 3$.
Let $u_{\e, \delta} \in H_0^1(\Omega)$ be a weak solution of
$-\text{\rm div} (A^\e _{ \delta} \nabla u_{\e, \delta} ) =f$ in $\Omega$, where $f\in L^2(\Omega)$ with 
support in $ B(y_0, r) \cap \Omega$, where $y_0\in \overline{\Omega}$ 
and $0<r<\text{\rm diam}(\Omega)$.
Then 
\begin{equation}\label{6.0-0}
\|  [ \Lambda^\e_{ \delta} ]^2   \nabla u_{\e, \delta}\|_{L^2(\Omega)}
\le C r  \| f\|_{L^2(\Omega)},
\end{equation}
where  $C$ depends only on $d$, $\mu$, $\omega$, $\kappa$, and the Lipschitz character of $\Omega$.
\end{lemma}

\begin{proof}
Let $\varphi  \in H^1_0(\Omega)$.
Then
\begin{equation}\label{6.0-10}
\int_\Omega 
[\Lambda_{ \delta} (x/\e )]^2 A(x/\e) \nabla u_{\e, \delta} \cdot  \nabla \varphi \, dx
=\int_{ B(y_0, r)\cap\Omega}
 f \cdot \varphi  \, dx.
\end{equation}
By letting  $\varphi =u_{\e, \delta}$ in (\ref{6.0-10})  and using 
$$
 \| \varphi\|_{L^2(B(y_0, r)\cap\Omega)}
\le C r \|\varphi \|_{L^p(\Omega)}
\le C r  \|\nabla \varphi \|_{L^2(\Omega)},
$$
 where  $p=\frac{2d}{d-2}$,
we obtain 
\begin{equation}\label{6.0-11}
\delta^2 \| \nabla u_{\e, \delta} \|_{L^2(\Omega)}
\le Cr   \| f\|_{L^2(\Omega)}.
\end{equation}
Next, in (\ref{6.0-10}),  we let $\varphi \in H_0^1(\Omega)$ be an extension of $u_{\e, \delta}|_{\Omega^\e}$ such that
\begin{equation}\label{6.0-12}
\|\nabla \varphi  \|_{L^2(\Omega)}
\le C  \|\nabla u_{\e, \delta} \|_{L^2(\Omega^\e)}
\end{equation}
(see Lemma \ref{lemma-ext}).
It follows that
$$
\aligned
\|\nabla u_{\e, \delta}\|_{L^2(\Omega^\e)}^2
 & \le C\delta^2 \|\nabla u_{\e, \delta} \|_{L^2(\Omega\setminus \Omega^\e)} \|\nabla \varphi \|_{L^2(\Omega\setminus \Omega^\e)}
+ Cr  \| f\|_{L^2(\Omega)} \|\nabla \varphi  \|_{L^2(\Omega)}\\
& \le Cr \| f\|_{L^2(\Omega)} \|\nabla u_{\e, \delta} \|_{L^2(\Omega^\e)},
\endaligned
$$
where we have used  (\ref{6.0-11}) and (\ref{6.0-12}) for the last inequality.
This, together with (\ref{6.0-11}), gives (\ref{6.0-0}).
\end{proof}

\begin{lemma}\label{lemma-6.30}
Suppose $d\ge 3$. Let $u_{\e, \delta} \in H_0^1(\Omega)$ be a weak solution of
$-\text{\rm div} (A^\e _{ \delta} \nabla u_{\e, \delta} ) =\Lambda^\e_{ \delta} f$
in $\Omega$, where $f\in L^2( B(y_0, r) \cap \Omega)$.
Then
\begin{equation}\label{6.30-0}
\|\Lambda^\e_{\delta}  \nabla u_{\e, \delta} \|_{L^2(\Omega)}
\le C r \| f\|_{L^2(\Omega)},
\end{equation}
where 
 $C$ depends only on $d$, $\mu$, $\omega$, $\kappa$, and the Lipschitz character of $\Omega$.
\end{lemma}

\begin{proof}
By using the test function $u_{\e, \delta}$ and Cauchy inequality,  we obtain 
\begin{equation}\label{6.30-1}
\|\Lambda^\e_{\delta} \nabla u_{\e, \delta} \|^2_{L^2(\Omega)}
\le C \| f\|_{L^2(\Omega)}
\|\Lambda^\e_{ \delta} u_{\e, \delta} \|_{L^2( B(y_0, r)\cap \Omega )}.
\end{equation}
Note that
$$
\|\Lambda^\e_{ \delta}  u_{\e, \delta} \|_{L^2( B(y_0, r)\cap \Omega )}
\le C r\|\Lambda^\e_{ \delta} u_{\e, \delta}\|_{L^p(\Omega)}
\le C r \|\Lambda^\e_{ \delta} \nabla u_{\e, \delta} \|_{L^2(\Omega)},
$$
where we have used  H\"older's inequality as well as the Sobolev type  inequality,
\begin{equation}\label{Sob}
\| \Lambda^\e_{\delta} \varphi \|_{L^p (\Omega)}
\le C\|\Lambda^\e_{ \delta} \nabla \varphi \|_{L^2(\Omega)},
\end{equation}
with  $p=\frac{2d}{d-2}$, 
for $\varphi \in H^1_0(\Omega)$.
This, together with (\ref{6.30-1}), gives (\ref{6.30-0}).
Finally, to see (\ref{Sob}, we let $\widetilde{\varphi}\in H^1_0(\Omega)$ be an extension 
of $\varphi|_{\Omega^\e}$ such that 
$\|\nabla \widetilde{\varphi} \|_{L^2(\Omega)}
\le C \|\nabla \varphi\|_{L^2(\Omega^\e)}$.
Then
$$
\aligned
\|\Lambda^\e_{ \delta} \varphi\|_{L^p(\Omega)}
&\le \| \widetilde{\varphi} \|_{L^p(\Omega)}
+ \delta \| \varphi\|_{L^p(\Omega)}\\
&\le C \|\nabla \widetilde{\varphi}\|_{L^2(\Omega)} + C\delta  \|\nabla \varphi\|_{L^2(\Omega)}
\le C \|\Lambda^\e_{ \delta} \nabla \varphi\|_{L^2(\Omega)},
\endaligned
$$
where we have used the usual Sobolev inequality for the second inequality.
\end{proof}

\begin{proof}[\bf Proof of Theorem \ref{thm-6.1}]

Fix $x_0, y_0\in \Omega$ and let $r=|x_0 -y_0|/2$.
We first consider the case $r\ge 2d \e$.
Let $f\in L^2 (B(y_0, r)\cap \Omega)$ and
$$
u_{\e, \delta} (x, y)=\int_\Omega G_{\e, \delta } (x, y) f(y) \, dy.
$$
Then $-\text{\rm div} (A^\e_{\delta} \nabla u_{\e, \delta} ) =f$ in $\Omega$ and $u_{\e, \delta} =0$ on $\partial\Omega$.
It follows from Lemma \ref{lemma-6.0} that
$$
\|\nabla u_{\e, \delta} \|_{L^2(\Omega^\e)}
\le C r \| f\|_{L^2(\Omega)}.
$$
Let $\widetilde{u}_{\e, \delta} \in H_0^1(\Omega)$ be an extension of $u_{\e, \delta}|_{\Omega^\e}$ such that
$\|\nabla \widetilde{u}_{\e, \delta} \|_{L^2(\Omega)}
\le C \|\nabla u_{\e, \delta} \|_{L^2(\Omega^\e)}$.
By (\ref{5.1-2}), 
$$
\aligned
\|u_{\e, \delta} \|_{L^2(\Omega\cap B(x_0, r))}
 & \le C \|u_{\e, \delta} \|_{L^2(\Omega^\e\cap B(x_0, 2r))}
 \le C r  \|\widetilde{u}_{\e, \delta} \|_{L^p(\Omega)}\\
&\le Cr  \|\nabla \widetilde{u}_{\e, \delta} \|_{L^2(\Omega)}
\le C  r^2 \| f\|_{L^2(\Omega)},
\endaligned
$$
where $p=\frac{2d}{d-2}$ and we have used a Sobolev inequality for $\widetilde{u}_{\e, \delta}$.
Since  $\text{\rm div} (A^\e _{ \delta} \nabla u_{\e, \delta} ) =0$ in $B(x_0, r)\cap \Omega$,
$u_{\e, \delta}=0$ on $\partial\Omega$, and $r> 2d\e$, 
it  follows from (\ref{6.3-0})  that
$$
| u_{\e, \delta} (x_0)|
\le C \left(\fint_{B(x_0, r)\cap \Omega} |u_{\e, \delta}|^2 \right)^{1/2}
\le  C   r^{2-\frac{d}{2}}
\| f\|_{L^2(B(y_0, r)\cap \Omega)}.
$$
By duality this implies that
$$
\left(\int_{B(y_0, r)\cap \Omega} |G_{\e, \delta} (x_0, y)|^2\, dy \right)^{1/2}
\le C r^{2-\frac{d}{2}}.
$$
Since $\text{\rm div} (A^{\e*} _{ \delta} \nabla_y G_{\e, \delta} (x_0, \cdot))=0$ in $B(y_0, r)$
and $G_{\e, \delta} (x_0, \cdot) =0$ on $\partial\Omega$, we obtain 
$$
|G_{\e,\delta} (x_0, y_0)|
\le C \left(\fint_{B(y_0, r)\cap \Omega} |G_{\e, \delta} (x_0, y)|^2\, dy \right)^{1/2}
\le C r^{2-d}.
$$

We now consider the case where  $r=(1/2) |x_0 -  y_0|<2d \e$.
Let $g\in L^2( B(y_0, r)\cap \Omega)$ and
$$
w_{\e, \delta} = \int_\Omega G_{\e, \delta} (x, y) \Lambda_{ \delta} (y/\e) g (y)\, dy.
$$
Then $-\text{\rm div} (A^\e_{ \delta} \nabla w_{\e, \delta} )=\Lambda^\e_{ \delta} \, g$ in $\Omega$ and
$w_{\e, \delta}=0$ on $\partial \Omega$.
By Lemma \ref{lemma-6.30},
$$
\|\Lambda^\e_{\delta} \nabla w_{\e, \delta} \|_{L^2(\Omega)}
\le C r\| g\|_{L^2(\Omega)}.
$$
It follows that
$$
\aligned
\|\Lambda^\e_{ \delta}  w_{\e, \delta} \|_{L^2(B(x_0, r)\cap \Omega)}
&\le C r \| \Lambda^\e_{\delta} w_{\e, \delta}  \|_{L^p(\Omega)}\\
&\le C r\| \Lambda^\e_{\delta} \nabla w_{\e, \delta}\|_{L^2(\Omega)}
\le C r^2 \| g\|_{L^2(\Omega)},
\endaligned
$$
where we have used (\ref{Sob}) for the second inequality.
Hence, by Theorem \ref{thm-local-2},
$$
\aligned
\Lambda^\e_{ \delta} (x_0) | w_{\e, \delta} (x_0)|
 & \le C \left(\fint_{B(x_0, r)\cap \Omega} |\Lambda^\e_{ \delta} w_{\e, \delta} |^2\right)^{1/2}\\
&\le C r^{2-\frac{d}{2}}
\| g\|_{L^2(\Omega)}.
\endaligned
$$
By duality we obtain
$$
\left(\fint_{B(y_0, r)\cap \Omega}
|\Lambda^\e_{\delta} (x_0)
G_{\e, \delta} (x_0, y) \Lambda^\e_{ \delta} (y)|^2\, dy \right)^{1/2}
\le C r^{2-d},
$$
which, by Theorem \ref{thm-local-2} again,   leads to
$$
\Lambda^\e_{ \delta} (x_0 )  \Lambda^\e_{\delta} (y_0)
G_{\e, \delta} (x_0, y_0)
\le C r^{2-d},
$$
since $\text{\rm div} (A^{\e*}_{\delta} \nabla G_{\e, \delta} (x_0, \cdot))=0$ in $ \Omega\setminus \{ x_0\} $ and $G_{\e, \delta} (x_0, \cdot)=0$ on $\partial\Omega$.
\end{proof}

\begin{remark}\label{re-6.40}
Suppose $d\ge 3$.
It follows from (\ref{6.3-0}) and (\ref{G-estimate}) that if $x, y\in \Omega$ and  $|x-y|\ge 4d \e$, then
\begin{equation}\label{6.40-0}
|G_{\e, \delta} (x, y)|\le \frac{ C [ d(x) ]^\sigma [d(y)]^\sigma }{|x-y|^{d-2+2\sigma }}
\end{equation}
for some $\sigma\in (0, 1)$,
where $d(x)=\text{\rm dist} (x, \partial\Omega)$.
The estimate (\ref{6.40-0}) also holds if
$|x-y| <  4d \e$ and $x, y\in \Omega^\e$.
\end{remark}

\begin{remark}\label{re-6.41}
In the case $d=2$,  the inequality $\|\varphi\|_{L^\infty(\Omega)} \le C \|\nabla \varphi\|_{L^2(\Omega)}$
fails for $\varphi\in H^1_0(\Omega)$.
However, an inspection of the proof of Lemma \ref{lemma-6.0} as well as the proof of (\ref{5.1-0}) 
shows that
one only needs
$$
\|\varphi\|_{L^2(B(y_0, r)\cap \Omega)}
\le C r\|\nabla \varphi\|_{L^2(\Omega)}
\quad \text{ for } \varphi \in H_0^1(\Omega).
$$
This Poincar\'e-type inequality holds  if $r\ge c_0 d(y_0)$ for $d\ge 2$.
As a result, the estimate (\ref{6.40-0})  continues to hold in the case $d=2$, if $|x-y|\ge 8\e$ and
$|x-y|\ge c_0\max (  d(x), d(y))$.
The estimate also holds if $c_0\max (  d (x), d(y)) < |x-y|< 8\e$ and $x, y\in \Omega^\e$.
\end{remark}

\begin{remark}

In the case where $\omega$,  $\Omega$ are smooth,
estimates of Green functions  may be found \cite{Yeh-2010, Yeh-2016}. 

\end{remark}


\section{\bf Boundary layer  estimates}

Throughout this section we assume that $\Omega$ is a bounded Lipschitz domain satisfying \eqref{F-k},
$A$ satisfies (\ref{ellipticity})-(\ref{periodicity}), and 
$A^\e_{ \delta}$ is given by (\ref{A-e}).
We use the scale-invariant $H^1(\partial\Omega)$ norm,
$$
\| f\|_{H^1(\partial\Omega)}
=\| \nabla_{\tan} f\|_{L^2(\partial\Omega)}
+ [\text{\rm diam} (\partial\Omega) ] ^{-1} \| f\|_{L^2(\partial\Omega)}.
$$

\begin{thm}\label{thm-7.1}
Let $u_{\e, \delta}\in H^1(\Omega)$ be a weak solution of \eqref{DP}
with  $f\in H^1(\partial\Omega)$. Then, for $0<\e\le 1$ and $0\le \delta\le 1$,
\begin{equation}\label{7.1-0}
\|\nabla u_{\e, \delta}\|_{L^2(\Sigma_{\alpha \e})}
\le  C_\alpha  \e^{1/2}  \| f\|_{H^1(\partial\Omega)},
\end{equation}
where $\alpha \ge  d$ and
$C_\alpha $ depends only on $d$, $\mu$, $\omega$, $\kappa$, $\alpha$, and the Lipschitz character of $\Omega$.
\end{thm}

\begin{proof}
Let $v$ be the solution of the homogenized problem:
$\text{\rm div}(\widehat{A_\delta} \nabla v)=0$ in $\Omega$ and 
$v=f$ on $\partial\Omega$.
Recall  that  for $0<t<1$, 
\begin{equation}\label{7.1-2}
\aligned
\|\nabla v\|_{L^2(\Sigma_t)}
+t \|\nabla ^2 v\|_{L^2(\Omega\setminus \Sigma_t)}  & \le
C t^{1/2} \|N(\nabla v) \|_{L^2(\partial\Omega)}\\
& \le C t^{1/2} \| f \|_{H^1(\partial\Omega)},
\endaligned
\end{equation}
where $C$ depends only on $d$, $\mu$, and the Lipschitz character of $\Omega$.
Let $w_{\e, \delta}$ be defined by (\ref{r-1}), with $v$ in the place of $v_\delta$.
The cut-off function $\eta_\e$ in \eqref{r-1} is chosen so that
$\eta_\e   =1$ in $\Omega\setminus \Sigma_{3\alpha \e}$ and
$\eta_\e =0$ in $\Sigma_{2\alpha \e}$, and $|\nabla \eta_\e|\le C_\alpha \e^{-1}$.
It follows from (\ref{3.1-0}) and (\ref{7.1-2})  that
\begin{equation}\label{7.1-3}
\Big|
\int_\Omega A^\e_{ \delta} \nabla w_{\e, \delta} \cdot \nabla \psi\, dx \Big|
\le C \e^{1/2}  \|\nabla \psi\|_{L^2(\Omega)} \| f\|_{H^1(\partial\Omega)}
\end{equation}
for any $\psi \in H^1_0(\Omega)$.
Let $\psi=w_{\e, \delta}$ in (\ref{7.1-3}).
Using the ellipticity condition (\ref{ellipticity}) and $\Lambda_{ \delta} \ge \delta$, we obtain 
\begin{equation}\label{7.1-4}
\delta^2 \|\nabla w_{\e, \delta} \|_{L^2(\Omega)}
\le C \e^{1/2} \| f\|_{H^1(\partial\Omega)}.
\end{equation}

By Lemma \ref{lemma-ext}, there exists  $\widetilde{w}_{\e, \delta}\in H^1_0(\Omega)$ such  that $\widetilde{w}_{\e, \delta}
=w_{\e, \delta}$ in $\Omega^\e$ and 
\begin{equation}\label{7.1-5}
\|\nabla \widetilde{w}_{\e, \delta} \|_{L^2(\Omega)}
\le C \| \nabla w_{\e, \delta} \|_{L^2(\Omega^\e)}.
\end{equation}
Let  $\psi=\widetilde{w}_{\e, \delta}$ in (\ref{7.1-3}). Using (\ref{7.1-4}) and (\ref{7.1-5}), 
we see that
$$
\aligned
\|\nabla w_{\e, \delta} \|_{L^2(\Omega^\e)}^2
&\le C \delta^2 \|\nabla w_{\e, \delta} \|_{L^2(\Omega)}
\| \nabla \widetilde{w}_{\e, \delta} \|_{L^2(\Omega)}
+ C \e^{1/2} \|\nabla \widetilde{w}_{\e, \delta}\|_{L^2(\Omega)} \| f\|_{H^1(\partial\Omega)}\\
&\le C \e^{1/2}  \|\nabla w_{\e, \delta}\|_{L^2(\Omega^\e)} \| f\|_{H^1(\partial\Omega)},
\endaligned
$$
which yields  $\|\nabla w_{\e, \delta} \|_{L^2(\Omega^\e)}\le C \e^{1/2} \| f\|_{H^1(\partial\Omega)}$.
Since  $w_{\e, \delta} =u_{\e, \delta} -v$ on $\Sigma_{2\alpha  \e}$, we obtain 
$$
\|\nabla u_{\e, \delta}\|_{L^2(\Omega^\e\cap \Sigma_{2\alpha \e})}
\le \| \nabla v\|_{L^2(\Sigma_{2\alpha  \e})}
+ \|\nabla w_{\e, \delta}\|_{L^2(\Omega^\e)}
\le C \e^{1/2} \| f\|_{H^1(\partial\Omega)},
$$
where we also used (\ref{7.1-2}) for the last inequality.
Finally, we note that by Lemma \ref{lemma-ext-F}, 
$\| \nabla u_{\e, \delta} \|_{L^2(\Sigma_{\alpha \e})} \le C \| \nabla u_{\e, \delta} \|_{L^2(\Omega^\e\cap \Sigma_{2\alpha \e})}$.
This completes the proof.
\end{proof}

Next we consider the Neumann problem (\ref{NP}).

\begin{thm}\label{thm-7.2}
Assume that $\widehat{A_\delta}$ is symmetric.
Let $u_{\e, \delta}\in H^1(\Omega)$ be a weak solution of \eqref{NP}
with $g\in L^2(\partial\Omega)$ and $\int_{\partial \Omega} g=0$.
Then,  for $0<\e\le 1$ and $0\le \delta\le 1$,
\begin{equation}\label{7.2-0}
\|\nabla u_{\e, \delta}\|_{L^2(\Sigma_{\alpha \e} )}
\le  C_\alpha  \e^{1/2}  \| g\|_{L^2(\partial\Omega)},
\end{equation}
where $\alpha\ge d$ and
$C$ depends only on $d$, $\mu$, $\omega$, $\kappa$, $\alpha$, and the Lipschitz character of $\Omega$.
\end{thm}

\begin{proof}
The proof is similar to that of Theorem \ref{thm-7.1}.
Let $v$ be a solution of the homogenized problem:
$\text{\rm div} (\widehat{A_\delta} \nabla v) =0$ in $\Omega$ and
$\frac{\partial v}{\partial \nu}=g$ on $\partial\Omega$,
where $\partial u/\partial \nu = n\cdot \widehat{A_\delta} \nabla v$.
Since $\widehat{A_\delta}$ is symmetric, 
it is known that $\|N(\nabla v)\|_{L^2(\partial\Omega)}
\le C \| g\|_{L^2(\partial\Omega)}$
(see \cite{Kenig-book} for references).
Thus,  for $0<t<1$, 
\begin{equation}\label{7.2-2}
\aligned
\|\nabla v\|_{L^2(\Sigma_t)}
+t \|\nabla ^2 v\|_{L^2(\Omega\setminus \Sigma_t)}  & \le
C t^{1/2} \|N(\nabla v) \|_{L^2(\partial\Omega)}\\
& \le C t^{1/2} \| g \|_{L^2(\partial\Omega)}.
\endaligned
\end{equation}
Let $w_{\e, \delta}$ be defined as in (\ref{r-1}), with $v$ in the place of $v_\delta$.
An inspection of the proof of Theorem \ref{thm-3.1} shows that
the inequality (\ref{3.1-0}) holds for any $\psi \in H^1(\Omega)$.
In fact, the only place that the boundary condition plays a role is in (\ref{3.1-1}).
It follows that (\ref{7.1-3}) holds for any $\psi\in H^1(\Omega)$.
The rest of the proof is exactly the same as that of Theorem \ref{thm-7.1}.
\end{proof}

We end this section with a localized version of Theorem \ref{thm-7.1}.

\begin{thm}\label{thm-7.3}
Let $u_{\e, \delta} \in H^1(D_{2r})$ be a weak solution of
$\text{\rm div} (A^\e_{ \delta} \nabla u_{\e, \delta})=0$ in $D_{2r}$
for some $r> (10d)\e$.
Then
\begin{equation}\label{7.3-0}
\frac{1}{\e} \int_{\Sigma_{\kappa \e, r}} |\nabla u_{\e, \delta}|^2\, dx
\le C \int_{\Delta_{2r}} |\nabla_{\tan} u_{\e, \delta}|^2\, d\sigma
+\frac{C}{r} \int_{D_{2r}} |\nabla u_{\e, \delta}|^2\, dx,
\end{equation}
where $\Sigma_{\kappa \e, r} =\Sigma_{\kappa \e} \cap D_r$ and $C$ depends only on $d$, $\mu$,
$\omega$, $\kappa$, and the Lipschitz character of $\Omega$.
\end{thm}

\begin{proof}
By rescaling we may assume $r=1$ and $\e>0$ is sufficiently small.
By Fubini's Theorem and the co-area formula we see that 
$$
\int_{3/2}^2
\left\{ \frac{1}{\e}
\int_{D_t\setminus D_{t-(2d)\e}}
|\nabla u_{\e, \delta}|^p\, dx
+\int_{\Omega \cap \partial B_t} |\nabla u_{\e, \delta}|^p\, d\sigma  \right\} dt
\le C \int_{D_2} |\nabla u_{\e, \delta}|^p\, dx,
$$
where $2\le p< \infty$.
It follows that there exists $t\in (3/2, 2)$ such that
\begin{equation}\label{7.3-1}
\frac{1}{\e}
\int_{D_t\setminus D_{t-(2d)\e}}
|\nabla u_{\e, \delta}|^p\, dx
+\int_{\Omega \cap \partial B_t} |\nabla u_{\e, \delta}|^p\, d\sigma 
\le 
C \int_{D_2} |\nabla u_{\e, \delta}|^p\, dx.
\end{equation}
Let $v$ be the solution of $\text{\rm div} (\widehat{A_\delta} \nabla v)=0$ in $D_t$
with $v=u_{\e, \delta}$ on $\partial D_t$.
Note that 
\begin{equation}\label{7.3-3}
\aligned
\|N(\nabla v)\|_{L^2(\partial D_t)}
 & \le C \| \nabla_{\tan} v\|_{L^2(\partial D_t)}\\
& \le C \big\{ \|\nabla_{\tan} u_{\e, \delta} \|_{L^2(\Delta_2)}
+ \| \nabla u_{\e, \delta} \|_{L^2(D_2)} \big\},
\endaligned
\end{equation}
where we used (\ref{7.3-1}) with $p=2$ for the last inequality.
Let $w_{\e, \delta}$ be defined by (\ref{r-1}). It follows from (\ref{3.1-0}) and (\ref{7.3-3}) that
\begin{equation}\label{7.3-2}
\Big|
\int_{D_t} A^\e_{ \delta} \nabla w_{\e, \delta} \cdot \nabla \psi\, dx \Big|
\le C \e^{1/2} \|\nabla \psi\|_{L^2(D_t)} 
\big\{ \|\nabla_{\tan} u_{\e, \delta} \|_{L^2(\Delta_2)}
+ \| \nabla u_{\e, \delta} \|_{L^2(D_2)} \big\}
\end{equation}
for any $\psi \in H^1_0(D_t)$.
By letting $\psi= w_{\e, \delta}$ in (\ref{7.3-2}) we obtain 
\begin{equation}\label{7.3-4}
\delta^2 \|\nabla w_{\e, \delta}\|_{L^2(D_t)}
\le C\e^{1/2} \big\{ \|\nabla_{\tan} u_{\e, \delta} \|_{L^2(\Delta_2)}
+ \| \nabla u_{\e, \delta} \|_{L^2(D_2)} \big\}.
\end{equation}

To proceed, we construct a new function $\widetilde{w}_{\e, \delta}$ in $D_t $ as follows.
For each $\e F_k \subset \Omega$ with $\e \widetilde{F}_k\subset D_t$,
 we replace $w_{\e, \delta}|_{\e F_k}$ by the harmonic extension of $w_{\e, \delta} |_{\partial (\e F_k)}$.
This gives us a function $\widetilde{w}_{\e, \delta}$ in $H_0^1(D_t)$ with 
the property  that
\begin{equation}\label{7.3-5}
\|\nabla \widetilde{w}_{\e, \delta}\|_{L^2(D_{t-d\e})}
  \le C \|\nabla w_{\e, \delta}\|_{L^2(D_t^\e)}.
 \end{equation}
 Also note that
 \begin{equation}\label{7.3-6}
 \aligned
  \|\nabla \widetilde{w}_{\e, \delta}\|_{L^2(D_t)}
 &\le C \big\{ \|\nabla w_{\e, \delta}\|_{L^2(D_t \setminus D_{t-d\e})}
 + \|\nabla w_{\e, \delta} \|_{L^2(D_t^\e)} \big\}.
 \endaligned
 \end{equation}
Let $\psi=\widetilde{w}_{\e, \delta}$ in (\ref{7.3-2}).
We obtain 
$$
\aligned
\|\nabla w_{\e, \delta}\|_{L^2(D_t^\e)}^2
&\le C \delta^2 \|\nabla w_{\e, \delta} \|_{L^2(D_t)}
\| \nabla \widetilde{w}_{\e, \delta}\|_{L^2(D_t)}\\
& \qquad + C \e^{1/2} \|\nabla \widetilde{w}_{\e, \delta} \|_{L^2(D_t)}
\big\{ \|\nabla_{\tan} u_{\e, \delta} \|_{L^2(\Delta_2)}
+ \| \nabla u_{\e, \delta} \|_{L^2(D_2)} \big\}\\
&\le C \e^{1/2} \|\nabla \widetilde{w}_{\e, \delta} \|_{L^2(D_t)}
\big\{ \|\nabla_{\tan} u_{\e, \delta} \|_{L^2(\Delta_2)}
+ \| \nabla u_{\e, \delta} \|_{L^2(D_2)} \big\},
\endaligned
$$
where we have used (\ref{7.3-4}) for the last inequality.
This, together with (\ref{7.3-6}) and (\ref{7.3-1}) for $p=2$, gives 
$$
\|\nabla w_{\e, \delta}\|_{L^2(D_t^\e)}
\le C \e^{1/2}
\big\{ \|\nabla_{\tan} u_{\e, \delta} \|_{L^2(\Delta_2)}
+ \| \nabla u_{\e, \delta} \|_{L^2(D_2)} \big\}.
$$
which, combining with (\ref{7.3-4}), leads to 
\begin{equation}\label{7.3-10}
\| [\Lambda^\e_{ \delta} ]^2  \nabla w_{\e, \delta} \|_{L^2(D_1)}
\le 
C \e^{1/2}
\big\{ \|\nabla_{\tan} u_{\e, \delta} \|_{L^2(\Delta_2)}
+ \| \nabla u_{\e, \delta} \|_{L^2(D_2)} \big\}.
\end{equation}
Observe that $\Sigma_{\kappa \e, 1} \subset D_1^\e$, and
$w_{\e, \delta} =u_{\e, \delta} - v$ on $\Sigma_{\kappa \e, 1}$.
We obtain 
$$
\aligned
\|\nabla u_{\e, \delta}\|_{L^2(\Sigma_{\kappa \e, 1})}
& \le \|\nabla v\|_{L^2(\Sigma_{\kappa \e, 1})}
+ \|\nabla w_{\e, \delta} \|_{L^2(\Sigma_{\kappa \e, 1})}\\
&\le
 C \e^{1/2}
\big\{ \|\nabla_{\tan} u_{ \e, \delta} \|_{L^2(\Delta_2)}
+ \| \nabla u_{\e, \delta} \|_{L^2(D_2)} \big\},
\endaligned
$$
where we also used (\ref{7.3-3}) for the last inequality.
\end{proof}


\section{Reverse H\"older inequalities}

Recall that $\Lambda_\delta^\e (x)=\Lambda_\delta (x/\e)$.
The goal of this section is to prove the following.

\begin{thm}\label{RH-thm}
Assume $A$ satisfies  \eqref{ellipticity} and  \eqref{periodicity}.
Let $u_{\e, \delta} \in H^1(B_{6r})$ be a weak solution of
$ \text{\rm div} (A_\delta^\e \nabla u_{\e, \delta})=0$ in $B_{6r}$.
Then
\begin{equation}\label{RH-1}
\left(\fint_{B_r} |\Lambda_\delta^\e \nabla u_{\e, \delta}|^p\right)^{1/p}
\le C \left(\fint_{B_{2r}} |\Lambda_\delta^\e \nabla u_{\e, \delta}|^2 \right)^{1/2},
\end{equation}
where $p>2$ and $C>1$ depend only on $d$, $\mu$, and $\omega$.
Moreover, if $r\ge 2d \e$, then
\begin{equation}\label{RH-1a}
\left( \fint_{B_r} |\nabla u_{\e, \delta}|^p \right)^{1/p}
\le C \left(\fint_{B_{6r}\cap \e\omega } |\nabla u_{\e, \delta}|^2 \right)^{1/2}.
\end{equation}
\end{thm}

By the self-improving property of the (weak) reverse H\"older inequalities, to show \eqref{RH-1},
it suffices to prove that if $\text{\rm div}(A^\e_\delta \nabla u_{\e, \delta})=0$ in $B_{4r}$, then
\begin{equation}\label{RH-2}
\left(\fint_{B_r} |\Lambda_\delta^\e \nabla u_{\e, \delta}|^2\right)^{1/2}
\le C \left(\fint_{B_{4r}} |\Lambda_\delta^\e \nabla u_{\e, \delta}|^q \right)^{1/q},
\end{equation}
for some $q<2$ and $C>1$, depending at most on $d$, $\mu$, and $\omega$.

We first treat the case $0<r<10d\e$.

\begin{lemma}\label{RH-lemma-1}
Let $u_{\e, \delta} \in H^1(B_{4r})$ be a weak solution of
$ \text{\rm div} (A_\delta^\e \nabla u_{\e, \delta})=0$ in $B_{4r}$.
Then \eqref{RH-2} holds  if $0< r< 10 d\e$.
\end{lemma}

\begin{proof}
By dilation we may assume $\e=1$.
We may also assume $0< r\le  c_0 $ and $c_0>0$ is small.
The case $c_0  < r < 10d $ follows from the case $r=c_0$ by a simple covering argument.
There are three cases: (1) $B_{3r/2}\subset \omega$; (2) $B_{3r/2}\subset F=\R^d\setminus \overline{\omega}$;
and (3) $B_{3r/2}\cap \partial \omega\neq \emptyset$.
Since $ \text{\rm div}(A\nabla u_{1, \delta})=0$ in $B_{4r} \cap \omega$ and $B_{4r}\cap F$,
the first two cases follow directly from the well-know reverse H\"older inequalities for the
elliptic operator $-\text{\rm div}(A\nabla )$.
The third case may be reduced to the case where $B_{4r}=B(x_0, 4r)$ for some $x_0\in \partial\omega$.
By Caccioppoli's inequality \eqref{4.1-0},
\begin{equation}\label{C-0}
\left( \fint_{B_r} |\Lambda_\delta \nabla u_{1, \delta} |^2\, dx\right)^{1/2}
\le \frac{C}{r} \inf_{\beta \in \R} \left(\fint_{B_{2r}} |\Lambda_\delta (u_{1, \delta}-\beta ) |^2\, dx\right)^{1/2}.
\end{equation}
Let $B_{2r}^+ =B_{2r}\cap \omega$ and $B_{2r}^- =B_{2r} \cap F$.
Choose  $\beta$ to be the average of $u_{1, \delta}$ over $B_{3r}^+$.
By Sobolev inequality,
  the right-hand side of \eqref{C-0} is bounded by
\begin{equation}\label{RH-3}
C \left(\fint_{B_{3r}^+} |\nabla u_{1, \delta} |^q \right)^{1/q}
+ \frac{C\delta }{r}\left(\fint_{B_{2r}^-} |u_{1, \delta} -\beta|^2 \right)^{1/2},
\end{equation}
where $q=\frac{2d}{d+2}$ for $d\ge 3$, and $1<q< 2$ for $d=2$.
To  bound the second term in \eqref{RH-3},
we apply  the  inequality,
\begin{equation}\label{Sob-1}
\frac{1}{r}
\left(\fint_{B_{2r}^-} |v|^2 \right)^{1/2}
\le C \left(\fint_{B_{3r}} |\nabla v|^q\right)^{1/q}
+ \frac{C}{r} \left(\fint_{B_{3r}^+} |v|^2 \right)^{1/2}
\end{equation}
to $v= u_{1, \delta} -\beta$.
It follows that \eqref{RH-3} is bounded by the right-hand side of \eqref{RH-2}.

Finally, to see \eqref{Sob-1},
by dilation we may assume $r=1$.
Since $\omega$ is a  Lipschitz domain,
by a change of variables, we may also assume $B_{3}^+=B_{3} \cap \{ x_d>0\}$ and
$B_{3}^-=B_{3}\cap \{ x_d<0\}$.
In this case, \eqref{Sob-1} follows by a compactness argument.
\end{proof}

\begin{lemma}\label{lemma-Sob-1} 
Let $u\in H^1(B_{2r})$ for some $r\ge 2d\e$.
Then
\begin{equation}\label{Sob-1-1}
\inf_{ \beta\in \mathbb{R}}
\frac{1}{r} \left(\fint_{B_r} |\Lambda^\e_{\delta} (u-\beta )|^2 \right)^{1/2}
\le C \left(\fint_{B_{2r}} |\Lambda^\e_{\delta} \nabla u|^q \right)^{1/q},
\end{equation}
where $q=\frac{2d}{d+2}$ for $d\ge 3$, $1<q<2$ for $d=2$, and
$C$ depends only on $d$ and $\omega$.
\end{lemma}

\begin{proof}
We consider the case $d\ge 3$.
The case $d=2$ is similar.
By rescaling we may assume $\e=1$.
Let $\widetilde{u}\in H^q (B_{3r/2})$ be an extension of $u|_{B_{3r/2} \cap \omega}$ such that
$$
\|\nabla \widetilde{u}\|_{L^q(B_{3r/2})} \le C \| \nabla u\|_{L^q(B_{2r }\cap \omega)}.
$$
Note that
\begin{equation}\label{Sob-1-2}
\aligned
\left(\int_{B_{3r/2}  \cap \omega} 
|u-\beta |^2\right)^{1/2}
&\le \left(\int_{B_{3r/2} } |\widetilde{u}-\beta|^2 \right)^{1/2}\\
&\le C \left(\int_{B_{3r/2}} |\nabla \widetilde{u} |^q \right)^{1/q}
\le C \left( \int_{B_{2r}\cap \omega} 
|\nabla u|^q \right)^{1/q},
\endaligned
\end{equation}
where $\beta =\fint_{B_{3r/2} } \widetilde{u}$.
Also observe that if $F_k  \cap B_r \neq \emptyset$, then $\widetilde{F}_k \subset B_{3r/2}$ and 
$$
 \int_{F_k } |u-\beta  |^2\, dx 
\le C \left(\int_{\widetilde{F}_k} |\nabla u |^q\, dx \right)^{2/q}
+ C \int_{\widetilde{F}_k  \setminus F_k} |u-\beta  |^2\, dx.
$$
It follows that
$$
\aligned
\int_{B_r \cap F } |u-\beta |^2
& \le C \sum_k  \left(\int_{\widetilde{F}_k } |\nabla u |^q\, dx \right)^{2/q}
+ C \int_{B_{3r/2} \cap \omega} |u-\beta |^2\, dx \\
&\le C \left(\int_{B_{2r}} |\nabla u|^q\, dx \right)^{2/q},
\endaligned
$$
where we have used the Minkowski inequality  as well as (\ref{Sob-1-2}) for the last inequality.
This, together with (\ref{Sob-1-2}), gives
$$
\left(\int_{B_r} |\Lambda_{ \delta} (u-\beta )|^2\, dx \right)^{1/2}
\le C \left( \int_{B_{2r}} |\Lambda_{\delta} \nabla u |^q\, dx \right)^{1/q},
$$
from which the inequality (\ref{Sob-1-1}) follows, as $\frac{d}{q} =\frac{d}{2} +1$.
\end{proof}

\begin{proof}[\bf Proof of Theorem \ref{RH-thm}]

As we indicated before, to prove \eqref{RH-1}, 
it suffices to show that  \eqref{RH-2}  holds for some $q<2$.
The case $0<r<10d\e$ is treated in Lemma \ref{RH-lemma-1}.
To handle the case $r\ge 10 d \e$, we use \eqref{C-0} and Lemma \ref{lemma-Sob-1}.

To prove \eqref{RH-1a}, we assume $\e=1$ and
 use the fact that there exists some $p_0>2$, depending only on $d$, $\mu$, and $\omega$, such that
 if $u\in W^{1, p} (\widetilde{F}_k)$ for some $2< p< p_0$ and
  $\text{\rm div}(A\nabla u)=0$ in $F_k$, then 
\begin{equation}\label{Meyer-1}
\| \nabla u \|_{L^p(F_k)} \le C \| \nabla u \|_{L^p (\widetilde{F}_k \setminus \overline{F_k})}, 
\end{equation}
where $C>0$ depends  only on $d$, $\mu$, and $\omega$.
It follows that
$$
\aligned
\left(\fint_{B_r} |\nabla u_{1, \delta} |^p \right)^{1/p}
& \le  C
\left(\fint_{B_{3r/2} \cap \omega } |\nabla u_{1, \delta} |^p \right)^{1/p}\\
&\le  C \left(\fint_{B_{3r}} |\nabla u_{1, \delta} |^2 \right)^{1/2}
 \le C \left(\fint_{B_{6r}\cap  \omega } |\nabla u_{1, \delta} |^2 \right)^{1/2},
\endaligned
$$
where we have used \eqref{RH-1} for the second inequality and \eqref{4.2-3} for the last.
To see \eqref{Meyer-1}, let $\widetilde{u}$ be an extension of $u|_{\widetilde{F}_k \setminus \overline{F_k}}$ to $\widetilde{F}_k$ such 
that $\| \nabla \widetilde{u}\|_{L^p(\widetilde{F}_k)} \le C \| \nabla u \|_{L^p(\widetilde{F}_k \setminus \overline{F_k})}$
for $p>2$.
Since $ \text{\rm div} (A\nabla (u-\widetilde{u}))=-\text{\rm div}(A\nabla \widetilde{u})$ in $F_k$ and
$u-\widetilde{u}=0$ on $\partial F_k$,  by Meyers' estimates, there exists some $p_0>2$, depending only on $d$, $\mu$, and
$\omega$, such that for $2< p< p_0$, 
$$
\| \nabla (u-\widetilde{u}) \|_{L^p(F_k)} \le C \| \nabla \widetilde{u} \|_{L^p(F_k)},
$$
from which the inequality \eqref{Meyer-1} follows.
\end{proof}

Recall that $D_r =B(x_0, r) \cap \Omega$ and $\Delta_r =B(x_0, r) \cap \partial\Omega$, where $x_0\in \partial \Omega$.

\begin{thm}\label{RH-thm-b}
Assume that $A$ satisfies \eqref{ellipticity} and \eqref{periodicity}.
Let $\Omega$ be a bounded Lipschitz domain satisfying \eqref{F-k}.
Suppose $u_{\e, \delta}\in H^1(D_{6r})$ is a weak solution of
$\text{\rm div} (A_\delta^\e \nabla u_{\e, \delta})=0$ in $D_{6r}$ 
with $u_{\e, \delta}=0$ on $\Delta_{6r}$.
Then
\begin{equation}\label{RH-b-1}
\left(\fint_{D_r} |\Lambda_\delta^\e \nabla u_{\e, \delta}|^p \right)^{1/p}
\le C \left(\fint_{D_{2r}} |\Lambda_\delta^\e \nabla u_{\e, \delta}|^2 \right)^{1/2},
\end{equation}
where $p>2$ and $C>1$ depend only on $d$, $\mu$, $\kappa$, $\omega$, and the Lipschitz character of $\Omega$.
Moreover, if $r\ge 2d \e$, 
\begin{equation}\label{RH-b-1a}
\left(\fint_{D_r} | \nabla u_{\e, \delta}|^p \right)^{1/p}
\le C \left(\fint_{D_{6r}\cap \e \omega } | \nabla u_{\e, \delta}|^2 \right)^{1/2}.
\end{equation}
\end{thm}

\begin{proof}
The proof is similar to that of Theorem \ref{RH-thm}. To show \eqref{RH-b-1}, by \eqref{RH-2}, 
it suffices to prove that if $\text{\rm div}(A_\delta^\e \nabla u_{\e, \delta})=0$ in $D_{4r}$ and $u_{\e, \delta} =0$ on $\Delta_{4r}$,
then 
\begin{equation}\label{RH-b-2}
\left(\fint_{D_r} |\Lambda_\delta^\e \nabla u_{\e, \delta}|^2 \right)^{1/2}
\le C \left(\fint_{D_{4r}} |\Lambda_\delta^\e \nabla u_{\e, \delta}|^q \right)^{1/q},
\end{equation}
for some $q<2$ and $C>1$, depending at most on $d$, $\mu$, $\kappa$, $\omega$, and the Lipschitz character of $\Omega$.
By dilation we may assume $\e=1$.
In view of the condition \eqref{F-k} as well as the interior estimate  \eqref{RH-2},
the case $0< r< 10 d$ follows from the reverse H\"older inequalities for solutions of  $\text{\rm div}(A\nabla u )=0$,
where $q=\frac{2d}{d+2}$ for $d\ge 3$ and $1<q<2$ for $d=2$.
Suppose $r\ge 10d$ and $d\ge 3$ (the case $d=2$ is similar). 
Let $\widetilde{u}\in H^q(D_{2r})$ be an extension of $u_{1, \delta} |_{D_r\cap \omega }$ such that $\widetilde{u}=0$ on $\Delta_{2r}$ and 
$$
\|\nabla \widetilde{u}\|_{L^q(D_r)}
\le C \|\nabla u_{1, \delta} \|_{L^q (D_{2r }\cap \omega )},
$$
where $q=\frac{2d}{d+2}$. Then
$$
\aligned
\left( \int_{D_r\cap \omega } |u|^2\, dx \right)^{1/2}
&\le C \left(\int_{D_r} |\widetilde{u}|^2\, dx  \right)^{1/2}\\
&\le C \left(\int_{D_r} |\nabla \widetilde{u}|^q\, dx \right)^{1/q}
\le C \left(\int_{D_{2r}\cap \omega } |\nabla u|^q\, dx \right)^{1/q}.
\endaligned
$$
This, together with the observation, 
$$
\| u\|_{L^2(D_r\cap F )}
\le \| u\|_{L^2(D_r)}
\le C \|\nabla u\|_{L^q(D_r)},
$$
gives the Sobolev inequality,
\begin{equation}\label{Sob-b}
\frac{1}{r} 
\left(\fint_{D_r} |\Lambda_\delta  u|^2 \right)^{1/2}
\le C \left(\fint_{D_{2r}}  |\Lambda_\delta \nabla u|^q \right)^{1/q}.
\end{equation}
The desired estimate (\ref{RH-b-2}) follows from \eqref{Sob-b} and the Caccioppoli inequality \eqref{5.1-4}.

To show \eqref{RH-b-1a}, we assume $\e=1$.
It follows from \eqref{Meyer-1} that
$$
\aligned
\left(\fint_{D_r} |\nabla u_{1, \delta} |^p \right)^{1/p}
& \le  C
\left(\fint_{D_{3r/2} \cap \omega } |\nabla u_{1, \delta} |^p \right)^{1/p}\\
&\le  C \left(\fint_{D_{3r}} |\nabla u_{1, \delta} |^2 \right)^{1/2}
 \le C \left(\fint_{D_{6r}\cap  \omega } |\nabla u_{1, \delta} |^2 \right)^{1/2},
\endaligned
$$
where $2<p<p_0$ and we have used \eqref{RH-b-1} for the second inequality and \eqref{5.1-0} for the last.
\end{proof}


\section{\bf Proof of Theorem \ref{main-thm-1} }

Throughout this section we assume that $A$ satisfies the ellipticity condition (\ref{ellipticity}),
the periodicity condition (\ref{periodicity}) and  the H\"older continuity condition (\ref{smoothness}).
We also assume $A$ is symmetric.

\begin{lemma}\label{lemma-8.1}
Let $u_{\e, \delta} \in H^1(D_{3r})$ be a weak solution of
$-\text{\rm div} (A^\e_{ \delta} \nabla u_{\e, \delta}  ) =\text{\rm div} (h)$ in $D_{3r}$
with $u _{\e, \delta} =0$ on $\Delta_{3r}$.
Assume that $r\ge 4d \e$.
Let
$$
s(x)=\left(\fint_{B(x, d\e)\cap \Omega  } |  [\Lambda^\e_{ \delta } ]^2  \nabla u_{\e, \delta}  |^2 \right)^{1/2}.
$$
Then 
\begin{equation}\label{8.1-0}
\left(\fint_{D_r} | s|^p\right)^{1/p}
\le C \left(\fint_{D_{2r}} |s|^2\right)^{1/2}
+ C \left(\fint_{D_{2r}} |h|^p \right)^{1/p},
\end{equation}
where $p>2$ and $C>0$ depend only on $d$, $\mu$, $\omega$, $\kappa$, and the Lipschitz
character of $\Omega$.
\end{lemma}

\begin{proof}
By the self-improving property of the reverse H\"older inequality, it suffices to prove that 
if $y_0\in D_r$ and $0<t<cr$, then
\begin{equation}\label{8.1-2}
\left(\fint_{\Omega (y_0, t)} | s|^2\right)^{1/2}
\le C \left(\fint_{\Omega (y_0, 2t)  } |s|^q\right)^{1/q}
+ C \left(\fint_{\Omega(y_0, 2t)} |h|^2 \right)^{1/2} 
+\theta \left(\fint_{\Omega(y_0, 2t)} |s|^2\right)^{1/2}
\end{equation}
for some $1<q<2$ and $\theta<1$, where $\Omega(y_0, t)=B(y_0, t) \cap \Omega$.
In view of the definition of $s(x)$, the inequality \eqref{8.1-2} is trivial for $0<t\le 4d\e$.
For the case $4d\e< t<r$, by Fubini's Theorem and H\"older's inequality,  it is enough to show that
\begin{equation}\label{8.1-2-0}
\aligned
\left(\fint_{\Omega (y_0, t)} | [\Lambda^\e_{ \delta}]^2 \nabla u_{\e, \delta}  |^2\right)^{1/2}
 & \le C \left(\fint_{\Omega (y_0, Ct)  } | [\Lambda^\e_{ \delta}]^2 \nabla u_{\e, \delta} |^q\right)^{1/q}
+ C \left(\fint_{\Omega(y_0, Ct)} |h|^2 \right)^{1/2} \\
& \qquad
+\theta \left(\fint_{\Omega(y_0, Ct)} | [\Lambda^\e_{ \delta}]^2  \nabla u_{\e, \delta}  |^2\right)^{1/2},
\endaligned
\end{equation}
for some $1<q<2$ and $\theta\in (0,1)$  sufficiently small.
We will only consider the case $y_0=x_0\in \Delta_{r}$.
The other case  may be reduced either to this case or to the case where $B(y_0, 2t)\subset D_{2r}$,
which may be handled by using (\ref{Sob-1-1}).

To this end, we note that for $\varphi \in H_0^1(D_{3t} )$,
\begin{equation}\label{8.1-3}
\int_{D_{3t}} A^\e_{ \delta} \nabla u_{\e, \delta}  \cdot \nabla \varphi\, dx
=-\int_{D_{3t}} h\cdot \nabla \varphi\, dx.
\end{equation}
By letting $\varphi = u_{\e, \delta}  \xi^2$, where $\xi \in C_0^1(B(x_0, 2t))$, in (\ref{8.1-3}) and
using the Cauchy inequality, we obtain
\begin{equation}\label{8.1-4}
 \delta^4 \int_{D_{2t}} 
 |\nabla u_{\e, \delta}   |^2 \xi^2 \, dx
 \le C\delta^2  \int_{D_{2t}} |\Lambda^\e_{ \delta}  u_{\e, \delta}  |^2 | \nabla \xi |^2 \, dx
 + C \int_{D_{2t}} |h|^2 |\xi|^2 \, dx.
 \end{equation}
 It follows that
 \begin{equation}\label{8.1-5}
 \delta^2 \|\nabla u_{\e, \delta}   \|_{L^2(D_{3t/2})}
 \le C t^{-1}  \| [\Lambda^\e_{ \delta} ]^2 u_{\e, \delta}  \|_{L^2(D_{2t})} 
 + C \| h\|_{L^2(D_{2t})}.
 \end{equation}
 Next, let $\varphi=\widetilde{u}  \xi^2$ in (\ref{8.1-3}),
 where $\widetilde{u} $ is an extension of $u_{\e, \delta} |_{D_{3t/2}^\e}$ such that
 $\widetilde{u}  =0$ on $\Delta_{2t}$ and 
 \begin{equation}\label{8.1-6}
 \|\nabla \widetilde{u} \|_{L^q(D_{3t/2})}
 \le C_q \|\nabla u_{\e, \delta}  \|_{L^q(D_{2t}^\e)}
 \end{equation}
for all $1<q<2$.
This gives
$$
\aligned
& \int_{D_{2t}} \xi^2 A^\e_{ \delta} \nabla u_{\e, \delta}   \cdot \nabla \widetilde{u}  \, dx
+  2\int_{D_{2t}} \xi  ( A^\e_{ \delta} \nabla u_{\e, \delta}  \cdot  \nabla \xi ) \widetilde{u} \, dx\\
& \qquad
=
-\int_{D_{2t}}( h \cdot \nabla \widetilde{u}  ) \xi^2\, dx
- 2\int_{D_{2t}}
(h \cdot \nabla \xi )
\widetilde{u}  \xi \, dx.
\endaligned
$$
It follows that for $\xi \in C_0^1(B(x_0, 3t/2)) $, 
$$
\aligned
\int_{D_{2t}^\e} |\nabla u_{\e, \delta}   |^2 \xi^2\, dx
&\le C \delta^2 \int_{D_{2t} }
|\nabla u_{\e, \delta}  | |\nabla \widetilde{u}  | \xi^2\, dx
+ C \int_{D_{2t}} |\Lambda^\e_{\delta} \nabla u_{\e, \delta}   |
|\Lambda^\e_{ \delta} \widetilde{u}  | |\xi| |\nabla \xi|\, dx
\\
& 
\qquad+ C \int_{D_{2t}} | h| |\nabla \widetilde{u}  |\xi^2\, dx
+ C \int_{D_{2t} } |h| |\widetilde{u}  |  | \xi| |\nabla \xi| \, dx.
\endaligned
$$
Choose $\xi\in C_0^1(B(x_0, 3t/2))$ such that
$\xi=1$ on $B(x_0, t)$ and $|\nabla \xi|\le C t^{-1}$.
By splitting the second term in the right-hand side into integrals over $D_{2t}^\e$ and $D_{2t}\setminus D_{2t}^\e$ and using
the Cauchy inequality on the integral over $D_{2t}^\e$, we see that 
$$
\aligned
\int_{D_t^\e} |\nabla u_{\e, \delta}   |^2\, dx
&\le C \delta^2 \int_{D_{3t/2} }
|\nabla u_{\e, \delta}  | |\nabla \widetilde{u}  |\, dx
+ Ct^{-2}  \int_{D_{3t/2}} |\widetilde{u} |^2 \, dx\\
&\ \ + C \delta^4 \int_{D_{3t/2}} |\nabla u_{\e, \delta}  |^2 \, dx 
+ C \int_{D_{3t/2}} | h| |\nabla \widetilde{u}| \, dx
+ Ct^{-1}  \int_{D_{3t/2}} |h| |\widetilde{u}  |  \, dx\\
&\le C_\theta  \delta^4 \int_{D_{3t/2}} |\nabla u_{\e, \delta} |^2\, dx
+ \theta \int_{D_{3t/2}} |\nabla \widetilde{u}|^2 \, dx\\
&\qquad
+ Ct^{-2}  \int_{D_{3t/2}} |\widetilde{u} |^2 \, dx
+ C_\theta  \int_{D_{3t/2} } |h|^2\, dx.
\endaligned
$$
This, together with (\ref{8.1-5}), (\ref{8.1-6}), and \eqref{Sob-b},   gives
$$
\aligned
\left(\fint_{D_t} |[\Lambda^\e_{ \delta}]^2  \nabla u_{\e, \delta} |^2 \right)^{1/2}
&\le C_\theta \left(\fint_{D_{2t} } | [\Lambda^\e_{ \delta}]^2 \nabla u_{\e, \delta} |^q \right)^{1/q}
 + C\theta 
 \left(\fint_{D_{2t} } | [\Lambda^\e_{ \delta}]^2 \nabla u_{\e, \delta} |^2 \right)^{1/2}\\
& \qquad \qquad
 + C_\theta  \left(\fint_{D_{2t}} |h|^2 \right)^{1/2},
 \endaligned
 $$
for some $q<2$ and any $\theta\in (0, 1)$.
As a result, we have proved (\ref{8.1-2-0}) for the case $y_0\in \Delta_r$.
\end{proof}

\begin{thm}\label{thm-8.10}
Let $u_{\e, \delta}\in H^1(D_{2r}) $ be a weak solution of $\text{\rm div}(A^\e_{ \delta} \nabla u_{\e, \delta})
=0$ in $D_{2r}$ with $u_{\e, \delta} =0$ on $\Delta_{2r}$.
Then 
\begin{equation}\label{8.10-0}
\left( \fint_{\Delta_r} 
|\nabla u_{\e, \delta}|^p\, d\sigma \right)^{1/p} 
\le  C  \left( \fint_{D_{2r}} |\nabla u_{\e, \delta}|^2\, dx\right)^{1/2},
\end{equation}
for some $p>2$ and $C>1$ depending  only on $d$, $ \mu$, $\omega$, $\kappa$,
 $(M, \sigma) $ in \eqref{smoothness},  and
the Lipschitz character of $\Omega$.
\end{thm}

\begin{proof}
Since $A^\e_{ \delta} (x)=A(x/\e)$ for $x\in \Omega$ with dist$(x, \partial \Omega)< \kappa \e$, 
the inequality (\ref{8.10-0}) for  $0<r< C \e$  follows from the case $\e=1$,  by a rescaling argument.
The  H\"older continuity condition and the symmetry condition are used. See \cite{Kenig-book} for references
for the case $\e=1$ and $0<r<C$.
To treat the large-scale case $r\ge C\e$, by a rescaling argument, we may assume $r=1$.
By using the estimate (\ref{8.10-0}) for the case $r=c  \e$ and  a simple covering argument,
we see that
\begin{equation}\label{8.10-1}
\int_{\Delta_1} |\nabla u_{\e, \delta}|^p\, d\sigma
\le \frac{C}{\e}
\int_{\Sigma_{c\e, 5/4}} | T([\Lambda^\e_{ \delta}]^2 \nabla u_{\e,\delta})|^p\, dx,
\end{equation}
where $T(G)(x) $
denotes the $L^2$ average of $|G|$ over $B(x, d\e)\cap \Omega$, and
$$
\Sigma_{c\e, 5/4} =\{ x\in D_{5/4}: \text{dist}(x, \partial\Omega)< c \e \}.
$$
To bound the right-hand side of (\ref{8.10-1}),
we let  $t\in (3/2, 2)$,  $v$ and $w_{\e, \delta}$ be the same as in the proof of Theorem \ref{thm-7.3}.
It follows from the proof of Lemma \ref{thm-3.1} that
$$
-\text{\rm div} (A^\e_{ \delta} \nabla w_{\e, \delta} )
=\text{\rm div} (h) \quad \text{ in } D_t,
$$
and 
$$
\aligned
h & =-  (\widehat{A_\delta}
-A^\e_{ \delta} ) \big(\nabla v -S_\e (\eta_\e (\nabla v)\big)
+\e 
 A^\e_{ \delta} \chi_\delta (x/\e) \nabla S_\e \big(\eta_\e (\nabla v)\big)\\
&
\qquad
-\e  \phi_\e (x/\e) \nabla S_\e \big(\eta_\e (\nabla v)\big).
\endaligned
$$
Observe that
$$
\aligned
\| h\|_{L^p(D_t)}
& \le C \|\nabla v -S_\e (\eta_\e (\nabla v))\|_{L^p(D_t)}
+ C \e \| \nabla (\eta_\e (\nabla v))\|_{L^p(D_t)}\\
&\le  C \|\nabla v \|_{L^p(D_t (4\e))}
+ C \e \|\nabla^2  v \|_{L^p(D_t \setminus D_t (\e))},
\endaligned
$$
where $D_t (s) =\{ x\in D_t: \text{\rm dist}(x, \partial D_t)<s\}$.
By the $L^p$ estimate for the regularity problem for
the operator $-\text{\rm div} (\widehat{A_\delta}\nabla ) $, and (\ref{7.3-1}),
$$
\aligned
 \| h\|_{L^p(D_t)}
&  \le C \e^{1/p} \| N( \nabla v_\delta) \|_{L^p(\partial D_t)}
\le C \e^{1/p} \|\nabla_{\tan} v_\delta \|_{L^p(\partial D_t)}\\
&\le C \e^{1/p}
\|\nabla u_{\e, \delta}\|_{L^p(\partial D_t)}
\le C \e^{1/p} \| \nabla u_{\e, \delta} \|_{L^p(D_{2})},
\endaligned
$$
where $p>2$ depends only on $d$, $\mu$, $\omega$, $\kappa$, and the Lipschitz character of $\Omega$.
We now use Lemma \ref{lemma-8.1} to obtain 
$$
\aligned
\left(\fint_{D_1} |T( [\Lambda^\e_{ \delta} ]^2\nabla w_{\e, \delta}) |^p \right)^{1/p}
& \le
\left(\fint_{D_{5/4}} |T([\Lambda^\e_{ \delta }]^2  \nabla w_{\e, \delta}) |^2 \right)^{1/2}
+
C \left( \fint_{D_{5/4} } |h|^p \right)^{1/p}\\
&\le 
C \left(\fint_{D_{3/2}} | [\Lambda^\e_{ \delta }]^2  \nabla w_{\e, \delta}|^2 \right)^{1/2}
+
C \e^{1/p} \|\nabla u_{\e, \delta} \|_{L^p(D_2)}\\
&\le 
C \e^{1/p} \|\nabla u_{\e, \delta} \|_{L^p(D_2)},
\endaligned
$$
where we have used (\ref{7.3-10}) for the last inequality.

Finally, we note that $\nabla w_{\e, \delta} 
=\nabla u_{\e, \delta} -\nabla v $ on $\Sigma_{2d \e, 3/2}$.
It follows that
$$
\aligned
& \left(\frac{1}{\e} \int_{\Sigma_{c\e, 5/4}} |T([\Lambda^\e_{ \delta} ]^2 \nabla u_{\e, \delta})|^p \right)^{1/p}\\
&
\le \left(\frac{1}{\e} \int_{\Sigma_{c\e, 5/4}} |T([\Lambda^\e_{ \delta} ]^2 \nabla w_{\e, \delta})|^p \right)^{1/p}
+ \left(\frac{1}{\e} \int_{\Sigma_{c\e, 5/4}}  |T(\nabla v )|^p \right)^{1/p}\\
&\le C \|\nabla u_{\e, \delta}\|_{L^p(D_2)} 
+ C \| N(\nabla v) \|_{L^p(\partial D_t)}\\
&\le C \|\nabla u_{\e, \delta}\|_{L^2(D_3)}, 
\endaligned
$$
where we have used \eqref{RH-b-1a} for the last inequality.
This completes the proof.
\end{proof}

\begin{proof}[\bf Proof of Theorem \ref{main-thm-1} for the case $0<\delta\le 1$]

By dilation  we may assume $\text{\rm diam}(\Omega)=1$.
Let $u_{\e, \delta}$ be a weak solution of $\text{\rm div} (A^\e_{ \delta} \nabla u_{\e, \delta}) =0$
in $\Omega$ with $u_{\e, \delta}=f$ on $\partial\Omega$, where $f\in H^1(\partial\Omega)$.
Let $\mathcal{M}_{\partial\Omega}$ denote the Hardy-Littlewood maximal  operator on $\partial\Omega$.
Define
\begin{equation}\label{N-t}
\widetilde{N}
(u) (y) =\sup\Big\{ 
|u(x)|: \ x\in \Omega^\e \text{ and } |x-y|< \widetilde{C}_0\, \text{\rm dist} (x, \partial\Omega) \Big\}
\end{equation}
for $y\in \partial\Omega$.
Using the fact $\Sigma_{\kappa \e} \subset \Omega^\e$ and  the inequality (\ref{4.2-1}),
 we see that 
$$
N(u_{\e, \delta}) \le C\widetilde{N} (u_{\e, \delta})  \quad \text{ on } \partial\Omega.
$$
We will show that 
\begin{equation}\label{8.20-1}
\widetilde{N}(u_{\e, \delta} )
\le C  \left[  \mathcal{M}_{\partial \Omega} (|f|^q)\right]^{1/q},
\end{equation}
where $1<q<2$ and $C>0$ depend only on $d$, $\mu$, $\omega$, $\kappa$,  and the Lipschitz character of $\Omega$.
This would imply that $\| N (u_{\e, \delta}) \|_{L^p(\partial\Omega)}
\le C \| f\|_{L^p(\partial\Omega)}$ for any $q< p\le \infty$.
The Hardy-Littlewood maximal operator $\mathcal{M}_{\partial\Omega}$ in (\ref{8.20-1})  is defined by
$$
\mathcal{M}_{\partial\Omega} (h)
(y)=\sup\left\{
\fint_{B(y, r)\cap \partial\Omega} |h|\, d\sigma: \
0<r< \text{\rm diam}(\partial\Omega) \right\}
$$
for $y\in \partial\Omega$, and it is well known that  $\|\mathcal{M}_{\partial\Omega} (h)\|_{L^p(\partial\Omega)}
\le C_p \| h\|_{L^p(\partial\Omega)}$ for $1<p<\infty$.

To prove  (\ref{8.20-1}), we fix $y_0\in \partial \Omega$ and
$x_0 \in \Omega^\e $ with $r=|x_0-y_0|< \widetilde{C}_0 \, \text{dist} (x_0, \partial\Omega)$. Using 
the Green function representation,
$$
u_{\e, \delta} (x_0)
=-\int_{\partial\Omega}
\frac{\partial }{\partial \nu_\e (y)}
\big\{ G_{\e, \delta} (x_0, y) \big\} f(y)\, d\sigma (y),
$$
it is not hard to see that (\ref{8.20-1}) follows from the estimates,
\begin{align}
\left(\fint_{B(y_0, c r)\cap \partial\Omega} |\nabla_y G_{\e, \delta}(x_0, y) |^p \, d\sigma (y) \right)^{1/p}
 & \le \frac{C}{r^{d-1}}, \label{8.20-2}\\
\left(\fint_{B(z, R) \cap \partial \Omega } |\nabla_y G_{\e, \delta } (x_0, y ) |^p\, d\sigma (y)  \right)^{1/p}
 & \le \frac{C}{R^{d-1}} \left(\frac{r}{R}\right)^\sigma , \label{8.20-3}
\end{align}
for some $p>2$ and $\sigma\in (0, 1)$, where  $z\in \partial\Omega$, $|z-y_0|\ge cr$, and $R\approx |z-y_0|$.
Moreover, since $\text{\rm div} (A^\e_{ \delta} \nabla_y G_{\e, \delta}(x_0, y))=0$ 
in $\Omega \setminus \{ x_0\}$ and $G_{\e, \delta} (x_0, y)=0$ for $y \in \partial\Omega$,
 in view of (\ref{8.10-0}), we only need to show that
\begin{align}
\left(\fint_{B(y_0, c r)\cap \Omega} |\nabla_y G_{\e, \delta}(x_0, y) |^2 \, d y \right)^{1/2}
 & \le \frac{C}{r^{d-1}}, \label{8.20-4}\\
\left(\fint_{B(z, R) \cap \Omega } |\nabla_y G_{\e, \delta } (x_0, y ) |^2\, d y   \right)^{1/2}
 & \le \frac{C}{R^{d-1}} \left(\frac{r}{R}\right)^\sigma. \label{8.20-5}
\end{align}
Furthermore, by Caccioppoli's inequality in Remark \ref{remark-Ca}, we reduce (\ref{8.20-4})-(\ref{8.20-5}) to
\begin{align}
\left(\fint_{B(y_0, c r)\cap \Omega^\e} | G_{\e, \delta}(x_0, y) |^2 \, d y \right)^{1/2}
 & \le \frac{C}{r^{d-2}}, \label{8.20-6}\\
\left(\fint_{B(z, R) \cap \Omega^\e } | G_{\e, \delta } (x_0, y ) |^2\, d y   \right)^{1/2}
 & \le \frac{C}{R^{d-2}} \left(\frac{r}{R}\right)^\sigma. \label{8.20-7}
\end{align}
Finally, we point out that since $x_0\in \Omega^\e$,
estimates (\ref{8.20-6})-(\ref{8.20-7}) follow from
(\ref{6.40-0}) (see Remark \ref{re-6.41} for the case $d=2$).
\end{proof}

\begin{proof}[\bf Proof of Theorem \ref{main-thm-1} for the case $\delta=0$]

Let $f\in H^1(\partial\Omega)$. 
By the classical theory for mixed boundary value problems,  there exists a unique $u_{\e, 0}\in H^1(\Omega^\e)$ satisfying \eqref{DP-0}.
We extend $u_{\e, 0}$ from $\Omega^\e$ to $\Omega$ by solving the Dirichlet problem (\ref{DP-e}) for each $\e F_k \subset \Omega$. 
Let $0<\delta\le 1$ and $u_{\e, \delta}\in H^1(\Omega)$ be the unique solution of (\ref{DP}).
Then
\begin{equation}\label{8.20-8}
\int_\Omega
A^\e_{ \delta} \nabla (u_{\e, \delta} -u_{\e, 0}) \cdot \nabla \psi \, dx
=-\delta^2 \int_{\Omega \cap \e F}  A(x/\e) \nabla u_{\e, 0} \cdot \nabla \psi\, dx
\end{equation}
for any $\psi \in H^1_0(\Omega)$.
By letting $\psi=u_{\e, \delta}- u_{\e, 0}$ in (\ref{8.20-8}) and using the Cauchy inequality we obtain
\begin{equation}\label{8.20-9}
\int_{\Omega^\e}
|\nabla (u_{\e, \delta} -u_{\e, 0} ) |^2\, dx
\le C \delta^2 \int_{\Omega \cap \e F } |\nabla u_{\e, 0}|^2\, dx,
\end{equation}
where $C$ depends only on $\mu$.
Note that
$$
\int_{\Omega \cap \e F } |\nabla u_{\e, 0} |^2\, dx
\le C \int_{\Omega^\e} |\nabla u_{\e, 0}|^2\, dx 
\le C \| f\|^2_{H^{1/2}(\partial\Omega)}.
$$
This, together with (\ref{8.20-9}) and Lemma \ref{lemma-ext-F}, shows that
\begin{equation}\label{8.20-10}
\aligned
\| u_{\e, \delta} -u_{\e, 0}\|_{H^1(\Omega)}
 & \le C \|\nabla (u_{\e, \delta} -u_{\e, 0})  \|_{L^2(\Omega)}\\
& \le C \|\nabla (u_{\e, \delta} -u_{\e, 0})  \|_{L^2(\Omega^\e) }
\to 0 \quad \text{ as } \delta \to 0.
\endaligned
\end{equation}
Using  $\| N(u_{\e, \delta})\|_{L^2(\partial\Omega)} \le C \| f\|_{L^2(\partial\Omega)}$,
by a limiting argument,
 we obtain 
\begin{equation}\label{8.20-11}
\| N(u_{\e, 0}) \|_{L^2(\partial\Omega)}
\le C \| f\|_{L^2(\partial\Omega)}.
\end{equation}
A density  argument gives (\ref{8.20-11}) for solutions with any boundary data $f\in L^2(\partial\Omega)$.
\end{proof}

We end this section with a localized version of Theorem \ref{main-thm-1}.
For $y\in \partial\Omega$,  define
\begin{equation}\label{9.1-0}
N_t ( u) (y)
=\sup 
\left(\fint_{B(x, d(x)/4)} | u |^2 \right)^{1/2},
\end{equation}
where the supremum is taken over all $
 x\in \Omega$ with $ d(x) < t  \text{ and } |x-y |< C_0\,  d (x)$.

\begin{thm}\label{thm-D}
Assume $A$ and $\Omega$ satisfy the same conditions as in Theorem \ref{main-thm-1}.
Let $u=u_{\e, \delta} \in H^1(D_{4r})$ be a weak solution of $\text{\rm div}(A_\delta^\e \nabla u)=0$ in $D_{4r}$,
where $r>4d\e$.
Then
\begin{equation}\label{local-D}
\int_{\Delta_r} |N_{r} (u)|^2\, d\sigma
\le C \int_{\Delta_{4r} } |u|^2\, d\sigma
+\frac{C}{r} \int_{D_{4r}} |u|^2 \, dx,
\end{equation}
where $C$ depends only on $d$, $\mu$, $\kappa$, and the Lipschitz character of $\Omega$.
\end{thm}

\begin{proof}
By dilation we may assume $r=1$ and $\e>0$ is sufficiently small.
We may also assume $0\in \partial\Omega$ and
\begin{equation}\label{Local-D-1}
\aligned
D_{4} &= B(0, 4)\cap \Omega= B(0, 4)\cap \big\{ (x^\prime, x_d): x^\prime \in \R^{d-1} \text{ and } x_d >\psi (x^\prime) \big\},\\
\Delta_{4}  &= B(0, 4) \cap \partial\Omega = B(0, 4) \cap \big\{ (x^\prime, \psi (x^\prime)): x^\prime \in \R^{d-1} \big\}.
\endaligned
\end{equation}
Using conditions \eqref{dis} and \eqref{F-k} we may
 construct a family $\{ \Omega_t\}$, where $t\in J \subset (3, 4)$,  of Lipschitz domains with uniform Lipschitz character and satisfying the following 
conditions:
\begin{itemize}

\item $ D_2\subset  \Omega_t\subset D_4$ and  $\Delta_2 \subset \partial \Omega_t$;

\item  dist$(\partial \Omega_t; \e F ) \ge c_1\e$, 
$$
\int_{t\in J} \int_{\partial\Omega_t \setminus \Delta_4} |v|^2\, d\sigma dt \le C_1 \int_{D_4} |v|^2\, dx
$$
for any $v\in H^1(D_4)$, and 
$|J|\ge c_1$, 
\end{itemize}
where $C_1>0, c_1>0$ depend only on $d$, $\kappa$, $\omega$, and the Lipschitz character of $\Omega$.
Note that
$$
\aligned
\int_{\Delta_1} |N_1 (u)|^2\, d\sigma
&\le \int_{\partial\Omega_t} |N(u)|^2\, d\sigma
 \le C \int_{\partial\Omega_t} |u|^2\, d\sigma\\
& \le C \int_{\Delta_4} |u|^2\, d\sigma 
+ C \int_{\partial\Omega_t\setminus \Delta_4} |u|^2\,  d\sigma,
\endaligned
$$
where we have used Theorem \ref{main-thm-1} for the second inequality.
By integrating the inequalities above in $t$ over $J$, we obtain \eqref{local-D} for $r=1$.
\end{proof}


\section{\bf Proof of Theorem \ref{main-thm-2}}

Let $\Omega$ be a bounded Lipschitz domain in $\mathbb{R}^d$ satisfying \eqref{F-k}.
By dilation  we  assume diam$(\Omega)=1$.
Let $u=u_{\e, \delta}$ be a weak solution of $\text{\rm div}(A^\e_{ \delta} \nabla u) =0$ in $\Omega$
with $u =f$ on $\partial\Omega$,
where $f\in H^1(\partial\Omega)$.
 Since $A^\e_{\delta} (x) =A(x/\e)$ if $d(x)< \kappa \e$, we obtain
 \begin{equation}\label{local-R}
 \int_{B(y, c_1 \e)\cap \partial\Omega} 
 | N_{ c_1\e} (\nabla u )  |^2\, d\sigma
 \le C \int_{B(y,  \e)  \cap \partial\Omega} 
 |\nabla_{\tan} f  |^2\, d\sigma
 + \frac{C}{\e} \int_{B(y, \e)\cap \Omega}
 |\nabla u  |^2\, dx
 \end{equation}
 for any $y\in \partial\Omega$, where $c_1>0$ is sufficiently small.
 We remark that the estimate \eqref{local-R} is known for $\e=1$ under the conditions that $A$ is elliptic, symmetric, and H\"older continuous
 \cite{Kenig-book}. The case $0< \e<1$ follows from the case $\e=1$ by a rescaling argument.
 Note that if $c_1 \e \le  d(x)< \alpha \e$ and $\alpha>1$ is large,
 the $L^2$  average of $|\nabla u_{\e, \delta}|$ over  $B(x, d(x)/4)$ is controlled 
 by its average over $B(y, C\e)\cap \Omega$, where $y\in \partial\Omega$ and  $|y-x|< C_0 d(x)$.
 It follows that
 $$
 \int_{B(y, c_1 \e)\cap \partial\Omega} 
 | N_{ \alpha \e} (\nabla u )  |^2\, d\sigma
 \le C \int_{B(y,  \e)  \cap \partial\Omega} 
 |\nabla_{\tan} f  |^2\, d\sigma
 + \frac{C}{\e} \int_{B(y, C \e)\cap \Omega}
 |\nabla u  |^2\, dx,
$$
where $C$ depends on $\alpha$.
 By covering $\partial\Omega$ with balls with radius $c_1 \e$, we see that
 \begin{equation}\label{9.1-3}
 \aligned
 \int_{ \partial\Omega} 
 | N_{ \alpha \e} (\nabla u  )  |^2\, d\sigma
 &  \le C \int_{\partial\Omega} 
 |\nabla_{\tan}  f |^2\, d\sigma
 + \frac{C}{\e} \int_{\Sigma_{C\e}}
 |\nabla u  |^2\, dx\\
 & \le C \| f\|^2 _{H^1(\partial\Omega)}, 
 \endaligned
 \end{equation}
 where $\Sigma_{C\e}  =\{ x\in \Omega: d(x)<  C \e\}$ and
 we have used the boundary layer estimate (\ref{7.1-0}) for the last inequality.
 
 To treat the case where $d(x)\ge \alpha \e$ and $\alpha>1$ is large, we will use the estimate in Theorem \ref{thm-D}.
 Without loss of generality we assume that $0\in \partial\Omega$ and
$$
B(0, r_0)\cap \Omega
= B(0, r_0) \cap  \{ (x^\prime, x_d)\in \mathbb{R}^d: x_d >\psi (x^\prime ) \}.
$$
We will show that if $y=(y^\prime, \psi (y^\prime)) \in \partial\Omega$ and $|y|\le c r_0$, then
\begin{equation}\label{claim-9}
N(\nabla u ) (y)
  \le 
  C \mathcal{M}_{\partial\Omega}  \Big\{ {M}_{\text{\rm rad}}  (Q_\e (u )) 
 +|\nabla_{\tan} f|
 + N_{\alpha\e} (\nabla  u )  \Big\}  (y),
\end{equation}
where
\begin{equation}\label{Q}
Q_\e (u) (x^\prime, x_d) = \big( u(x^\prime, x_d+\e) -u(x^\prime, x_d)\big)/\e
\end{equation}
is a difference operator, 
and
\begin{equation}\label{rad-M}
{M}_{\text{\rm rad}} (u) (y^\prime, \psi (y^\prime))
=\sup_{(\alpha-1) \e <s< cr_0}\fint_{B((y^\prime, \psi(y^\prime) +s), d\e)}
 |u|
\end{equation}
a radial maximal function.
To see (\ref{claim-9}), we fix $y\in \partial\Omega$ with $|y|\le cr_0$.
Let $z \in \Omega$  with  $d(z)\ge \alpha \e$ and 
$r=|z-y|< C_0\,  d(z) $.
Note that by \eqref{4.2-2} and \eqref{LL-2}
$$
\left(\fint_{B(z, d(z) /4)}
|\nabla u |^2 \right)^{1/2}
\le \frac{C}{r^{d+1} }
\int_{|x_d-z_d|<d(z)/2 }
\int_{|x^\prime-z^\prime|<  d(z)/2 }
|u -E| \, dx^\prime dx_d,
$$
where $E\in \mathbb{R}$, and that for $x=(x^\prime, x_d)$,
$$
\aligned
|u  (x)-E|
 & \le |u  (x^\prime, x_d ) -u (x^\prime, x_d -\e)|
+\cdots 
+ |u  (x^\prime, x_d -k\e  ) -u (x^\prime, \psi (x^\prime))|\\
& \qquad\qquad 
+|u(x^\prime, \psi (x^\prime))  -E|\\
&= \e | Q_\e (u ) (x^\prime, x_d-\e)|  +\cdots +
\e | Q_\e (u ) (x^\prime, x_d -(k-1)\e)| \\
&\qquad \qquad
+ |u (x^\prime, x_d -k\e  ) -u  (x^\prime, \psi (x^\prime))|
+|u  (x^\prime, \psi (x^\prime))  -E|,
\endaligned
$$
where $k=k(x) \approx \e^{-1} r  $ is chosen so that $\psi (x^\prime)\le x_d -k\e  < \psi (x^\prime) +\alpha \e $.
Let
$$
E=\fint_{|x^\prime-z^\prime|< d(z)/2 }  u  (x^\prime, \psi (x^\prime))  \, dx^\prime.
$$
Using
$$
|u  (x^\prime, x_d -k\e  ) -u  (x^\prime, \psi (x^\prime))|
\le  \int_0^{\alpha\e}  |\nabla u (x^\prime, \psi (x^\prime) +s)|\, ds
$$
and Poincar\'e's inequality for the term $|u  (x^\prime, \psi (x^\prime)) -E|$,
we see that
$$
\left(\fint_{B(z, d(z)/4)}
|\nabla u  |^2 \right)^{1/2}
\le C \mathcal{M}_{\partial\Omega}
\Big\{ M_{\text{\rm rad}}  (Q_\e ( u ))   + |\nabla_{\tan} f|
+ N_{\alpha \e} (\nabla u )  \Big\} (y), 
$$
which yields (\ref{claim-9}).
It follows from (\ref{claim-9}) and (\ref{9.1-3})  that
\begin{equation}\label{reg-10}
\aligned
\int_{ B(0, cr_0)\cap \partial\Omega}
|N(\nabla u )|^2\, d\sigma
& \le C \int_{ B(0, 2cr_0) \cap \partial \Omega}
|M_{\text{\rm rad}}
(Q_\e (u ))|^2\, d\sigma
+ C \| f\|_{H^1(\partial\Omega)}^2.
\endaligned
\end{equation}

Finally, to handle  $M_{\text{\rm rad}} (Q_\e (u ))$, we note that
$$
\text{\rm div} (A^\e_{\delta} \nabla Q_\e ( u )) =0
$$
in $ B(0, 100  cr_0)\cap \Omega$. By applying Theorem \ref{thm-D} to $Q_\e(u )$, we obtain 
$$
\aligned
 \int_{ B(0, 2cr_0) \cap \partial \Omega}
|M_{\text{\rm rad}}
(Q_\e (u ))|^2\, d\sigma
&\le C \int_{B(0, 3cr_0)\cap \partial\Omega} | N_{3cr_0} (Q_\e (u))|^2\, d\sigma\\
& \le C \int_{B(0, 12 cr_0)\cap \partial\Omega} 
|Q_\e (u) |^2\, d\sigma
+ C \int_{B(0, 12 cr_0)\cap \Omega} |Q_\e (u)|^2\, dx\\
& \le \frac{C}{\e}
\int_{\Sigma_{C\e}} | \nabla u|^2\, dx
+ C \int_\Omega |\nabla u|^2\, dx,
\endaligned
$$
where, for the last step,  we have used 
 the inequality,
$$
|Q_\e (u ) (x^\prime, x_d)|
\le \left( \frac{1}{\e}
\int_{x_d } ^{x_d+\e } |\nabla u (x^\prime, s)|^2 \, ds\right)^{1/2}.
$$
This, together with  (\ref{7.1-0}), (\ref{reg-10}), \eqref{energy-D},   and a simple covering argument, gives 
(\ref{R-estimate}).


\section{Proof of Theorem \ref{main-thm-3}}

By dilation  we may assume diam$(\Omega)=1$.
Let $u=u_{\e, \delta}\in H^1(\Omega)$ be a weak solution of the Neumann problem (\ref{NP}) with $g\in L^2(\partial\Omega)$.
It follows from the proof of Theorem \ref{main-thm-2} that
\begin{equation}\label{NP-60}
\int_{\partial\Omega}
|N(\nabla u ) |^2\, d\sigma
\le C \int_{\partial\Omega} 
|\nabla  u |^2\, d\sigma
+ C \int_{\Omega} |\nabla u |^2\, dx.
\end{equation}
Since $A^\e_{ \delta}(x)=A(x/\e)$ for $d(x)<\kappa \e$, we obtain 
\begin{equation}\label{local-N}
\int_{B(y, c\e)\cap \partial\Omega} |\nabla u |^2\, d\sigma
\le C \int_{B(y, \kappa \e)\cap \partial \Omega} |g|^2\, d\sigma
+\frac{C}{\e} \int_{B(y, \kappa \e)\cap \Omega} |\nabla u |^2\, dx
\end{equation}
for any $y\in \partial\Omega$, where $c>0$ is sufficiently small.
We point out that the estimate \eqref{local-N} is known for $\e=1$, under the conditions that 
$A$ is elliptic, symmetric, and H\"older continuous \cite{Kenig-book}.
The case $0< \e<1$ follows from the case $\e=1$ by a rescaling argument.
By a covering argument it follows from \eqref{local-N} that
$$
\aligned
\int_{ \partial\Omega} |\nabla u |^2\, d\sigma
& \le C \int_{ \partial \Omega} |g|^2\, d\sigma
+\frac{C}{\e} \int_{\Sigma_{\kappa \e} } |\nabla u |^2\, dx\\
&\le C \| g\|_{L^2(\partial\Omega)}^2,
\endaligned
$$
where we have used (\ref{7.2-0}) for the last inequality.
This, together with (\ref{NP-60}) and the energy estimate \eqref{energy-N}, gives (\ref{N-estimate}).

 \bibliographystyle{amsplain}
 
\bibliography{nt-perforated.bbl}

\bigskip

\begin{flushleft}

Zhongwei Shen,
Department of Mathematics,
University of Kentucky,
Lexington, Kentucky 40506,
USA.

E-mail: zshen2@uky.edu
\end{flushleft}

\bigskip

\end{document}